\documentclass[11pt]{amsart}  
\usepackage{amsbsy,amsmath,amsthm,amssymb,color,verbatim,ulem,pdflscape}
\usepackage{svg}
\usepackage{fullpage,cancel,hyperref,caption}
\usepackage{wrapfig}
\RequirePackage{etex}
\usepackage{graphicx}
\usepackage{microtype} 
\usepackage{multicol}
\usepackage{multirow}
\usepackage{makecell}
\usepackage{enumitem}
\setcellgapes{4pt}
\usepackage{supertabular,longtable}
\usepackage{tikz}
\usepackage{tkz-graph}
\tikzstyle{vertex}=[circle, draw, inner sep=0pt, minimum size=4pt]

\tikzstyle{vtx}=[circle, draw, inner sep=0pt, minimum size=8pt]

\definecolor{darkgreen}{cmyk}{.9,0,.9,.2}
\definecolor{midgray}{gray}{0.60}
\definecolor{lightgray}{gray}{0.90}
\definecolor{lmgray}{gray}{0.70}
\definecolor{alt_col0}{rgb}{1.00000000000000, 0.000000000000000, 0.000000000000000}
\definecolor{alt_col1}{rgb}{1.00000000000000, 0.666666666666667, 0.000000000000000}
\definecolor{alt_col2}{rgb}{0.607843137254902, 0.701960784313725, 0.239215686274510}
\definecolor{alt_col3}{rgb}{0.278431372549020, 0.819607843137255, 0.568627450980392}
\definecolor{alt_col4}{rgb}{0.000000000000000, 0.372549019607843, 0.701960784313725}
\definecolor{alt_col5}{rgb}{0.666666666666667, 0.000000000000000, 1.00000000000000}
\definecolor{alt_col6}{rgb}{1.00000000000000, 0.000000000000000, 0.533333333333333}
\definecolor{alt_col7}{rgb}{0.941176470588235, 0.000000000000000, 0.000000000000000}
\definecolor{alt_col8}{rgb}{0.819607843137255, 0.545098039215686, 0.000000000000000}
\definecolor{alt_col9}{rgb}{0.666666666666667, 1.00000000000000, 0.000000000000000}
\definecolor{alt_col10}{rgb}{0.000000000000000, 1.00000000000000, 0.666666666666667}
\definecolor{alt_col11}{rgb}{0.298039215686275, 0.607843137254902, 0.878431372549020}
\definecolor{alt_col12}{rgb}{0.588235294117647, 0.000000000000000, 0.878431372549020}
\definecolor{alt_col13}{rgb}{0.819607843137255, 0.000000000000000, 0.439215686274510}
\definecolor{alt_col14}{rgb}{0.878431372549020, 0.000000000000000, 0.000000000000000}
\definecolor{alt_col15}{rgb}{0.698039215686274, 0.466666666666667, 0.000000000000000}
\definecolor{alt_col16}{rgb}{0.588235294117647, 0.878431372549020, 0.000000000000000}
\definecolor{alt_col17}{rgb}{0.000000000000000, 0.701960784313725, 0.466666666666667}
\definecolor{alt_col18}{rgb}{0.000000000000000, 0.400000000000000, 1.00000000000000}
\definecolor{alt_col19}{rgb}{0.505882352941176, 0.000000000000000, 0.760784313725490}
\definecolor{alt_col20}{rgb}{0.698039215686274, 0.000000000000000, 0.372549019607843}
\definecolor{alt_col21}{rgb}{0.819607843137255, 0.000000000000000, 0.000000000000000}
\definecolor{alt_col22}{rgb}{1.00000000000000, 0.780392156862745, 0.341176470588235}
\definecolor{alt_col23}{rgb}{0.505882352941176, 0.760784313725490, 0.000000000000000}
\definecolor{alt_col24}{rgb}{0.298039215686275, 0.878431372549020, 0.686274509803922}
\definecolor{alt_col25}{rgb}{0.000000000000000, 0.329411764705882, 0.819607843137255}
\definecolor{alt_col26}{rgb}{0.686274509803922, 0.298039215686275, 0.878431372549020}
\definecolor{alt_col27}{rgb}{1.00000000000000, 0.341176470588235, 0.690196078431373}
\definecolor{alt_col28}{rgb}{0.760784313725490, 0.000000000000000, 0.000000000000000}
\definecolor{alt_col29}{rgb}{0.819607843137255, 0.639215686274510, 0.278431372549020}
\definecolor{alt_col30}{rgb}{0.733333333333333, 0.941176470588235, 0.321568627450980}
\definecolor{alt_col31}{rgb}{0.000000000000000, 1.00000000000000, 0.800000000000000}
\definecolor{alt_col32}{rgb}{0.000000000000000, 0.278431372549020, 0.701960784313725}
\definecolor{alt_col33}{rgb}{0.800000000000000, 0.000000000000000, 1.00000000000000}
\definecolor{alt_col34}{rgb}{0.819607843137255, 0.278431372549020, 0.568627450980392}
\definecolor{alt_col35}{rgb}{1.00000000000000, 0.341176470588235, 0.341176470588235}
\definecolor{alt_col36}{rgb}{0.698039215686274, 0.545098039215686, 0.239215686274510}
\definecolor{alt_col37}{rgb}{0.639215686274510, 0.819607843137255, 0.278431372549020}
\definecolor{alt_col38}{rgb}{0.000000000000000, 0.760784313725490, 0.607843137254902}
\definecolor{alt_col39}{rgb}{0.341176470588235, 0.603921568627451, 1.00000000000000}
\definecolor{alt_col40}{rgb}{0.705882352941177, 0.000000000000000, 0.878431372549020}
\definecolor{alt_col41}{rgb}{0.701960784313725, 0.239215686274510, 0.486274509803922}
\definecolor{alt_col42}{rgb}{0.819607843137255, 0.278431372549020, 0.278431372549020}
\definecolor{alt_col43}{rgb}{1.00000000000000, 0.800000000000000, 0.000000000000000}
\definecolor{alt_col44}{rgb}{0.533333333333333, 1.00000000000000, 0.000000000000000}
\definecolor{alt_col45}{rgb}{0.341176470588235, 1.00000000000000, 0.866666666666667}
\definecolor{alt_col46}{rgb}{0.278431372549020, 0.494117647058824, 0.819607843137255}
\definecolor{alt_col47}{rgb}{0.607843137254902, 0.000000000000000, 0.760784313725490}
\definecolor{alt_col48}{rgb}{1.00000000000000, 0.000000000000000, 0.400000000000000}
\definecolor{alt_col49}{rgb}{0.698039215686274, 0.239215686274510, 0.239215686274510}
\definecolor{alt_col50}{rgb}{0.878431372549020, 0.705882352941177, 0.000000000000000}
\definecolor{alt_col51}{rgb}{0.470588235294118, 0.878431372549020, 0.000000000000000}
\definecolor{alt_col52}{rgb}{0.278431372549020, 0.819607843137255, 0.713725490196078}
\definecolor{alt_col53}{rgb}{0.239215686274510, 0.423529411764706, 0.701960784313725}
\definecolor{alt_col54}{rgb}{0.866666666666667, 0.341176470588235, 1.00000000000000}
\definecolor{alt_col55}{rgb}{0.819607843137255, 0.000000000000000, 0.329411764705882}
\definecolor{alt_col56}{rgb}{1.00000000000000, 0.427450980392157, 0.341176470588235}
\definecolor{alt_col57}{rgb}{0.760784313725490, 0.607843137254902, 0.000000000000000}
\definecolor{alt_col58}{rgb}{0.372549019607843, 0.701960784313725, 0.000000000000000}
\definecolor{alt_col59}{rgb}{0.239215686274510, 0.701960784313725, 0.607843137254902}
\definecolor{alt_col60}{rgb}{0.000000000000000, 0.266666666666667, 1.00000000000000}
\definecolor{alt_col61}{rgb}{0.658823529411765, 0.258823529411765, 0.760784313725490}
\definecolor{alt_col62}{rgb}{0.698039215686274, 0.000000000000000, 0.278431372549020}
\definecolor{alt_col63}{rgb}{0.819607843137255, 0.352941176470588, 0.278431372549020}
\definecolor{alt_col64}{rgb}{1.00000000000000, 0.866666666666667, 0.341176470588235}
\definecolor{alt_col65}{rgb}{0.690196078431373, 1.00000000000000, 0.341176470588235}
\definecolor{alt_col66}{rgb}{0.000000000000000, 1.00000000000000, 0.933333333333333}
\definecolor{alt_col67}{rgb}{0.000000000000000, 0.235294117647059, 0.878431372549020}
\definecolor{alt_col68}{rgb}{0.933333333333333, 0.000000000000000, 1.00000000000000}
\definecolor{alt_col69}{rgb}{1.00000000000000, 0.341176470588235, 0.603921568627451}
\definecolor{alt_col70}{rgb}{0.698039215686274, 0.298039215686275, 0.239215686274510}
\definecolor{alt_col71}{rgb}{0.878431372549020, 0.764705882352941, 0.298039215686275}
\definecolor{alt_col72}{rgb}{0.305882352941176, 0.760784313725490, 0.000000000000000}
\definecolor{alt_col73}{rgb}{0.000000000000000, 0.819607843137255, 0.764705882352941}
\definecolor{alt_col74}{rgb}{0.000000000000000, 0.203921568627451, 0.760784313725490}
\definecolor{alt_col75}{rgb}{0.819607843137255, 0.000000000000000, 0.878431372549020}
\definecolor{alt_col76}{rgb}{0.819607843137255, 0.278431372549020, 0.494117647058824}
\definecolor{alt_col77}{rgb}{1.00000000000000, 0.266666666666667, 0.000000000000000}
\definecolor{alt_col78}{rgb}{0.760784313725490, 0.658823529411765, 0.258823529411765}
\definecolor{alt_col79}{rgb}{0.533333333333333, 0.878431372549020, 0.298039215686275}
\definecolor{alt_col80}{rgb}{0.321568627450980, 0.941176470588235, 0.898039215686275}
\definecolor{alt_col81}{rgb}{0.341176470588235, 0.517647058823529, 1.00000000000000}
\definecolor{alt_col82}{rgb}{0.709803921568627, 0.000000000000000, 0.760784313725490}
\definecolor{alt_col83}{rgb}{0.698039215686274, 0.239215686274510, 0.423529411764706}
\definecolor{alt_col84}{rgb}{0.819607843137255, 0.219607843137255, 0.000000000000000}
\definecolor{alt_col85}{rgb}{1.00000000000000, 0.933333333333333, 0.000000000000000}
\definecolor{alt_col86}{rgb}{0.458823529411765, 0.760784313725490, 0.258823529411765}
\definecolor{alt_col87}{rgb}{0.258823529411765, 0.760784313725490, 0.725490196078431}
\definecolor{alt_col88}{rgb}{0.278431372549020, 0.423529411764706, 0.819607843137255}
\definecolor{alt_col89}{rgb}{0.843137254901961, 0.298039215686275, 0.878431372549020}
\definecolor{alt_col90}{rgb}{1.00000000000000, 0.000000000000000, 0.266666666666667}
\definecolor{alt_col91}{rgb}{0.698039215686274, 0.188235294117647, 0.000000000000000}
\definecolor{alt_col92}{rgb}{0.878431372549020, 0.819607843137255, 0.000000000000000}
\definecolor{alt_col93}{rgb}{0.266666666666667, 1.00000000000000, 0.000000000000000}
\definecolor{alt_col94}{rgb}{0.000000000000000, 0.933333333333333, 1.00000000000000}
\definecolor{alt_col95}{rgb}{0.239215686274510, 0.360784313725490, 0.701960784313725}
\definecolor{alt_col96}{rgb}{1.00000000000000, 0.000000000000000, 0.933333333333333}
\definecolor{alt_col97}{rgb}{0.819607843137255, 0.000000000000000, 0.219607843137255}
\definecolor{alt_col98}{rgb}{1.00000000000000, 0.517647058823529, 0.341176470588235}
\definecolor{alt_col99}{rgb}{0.760784313725490, 0.709803921568627, 0.000000000000000}
\definecolor{alt_col100}{rgb}{0.235294117647059, 0.878431372549020, 0.000000000000000}
\definecolor{alt_col101}{rgb}{0.000000000000000, 0.709803921568627, 0.760784313725490}
\definecolor{alt_col102}{rgb}{0.000000000000000, 0.133333333333333, 1.00000000000000}
\definecolor{alt_col103}{rgb}{0.878431372549020, 0.000000000000000, 0.819607843137255}
\definecolor{alt_col104}{rgb}{0.698039215686274, 0.000000000000000, 0.188235294117647}
\definecolor{alt_col105}{rgb}{0.819607843137255, 0.423529411764706, 0.278431372549020}
\definecolor{alt_col106}{rgb}{1.00000000000000, 0.956862745098039, 0.341176470588235}
\definecolor{alt_col107}{rgb}{0.517647058823529, 1.00000000000000, 0.341176470588235}
\definecolor{alt_col108}{rgb}{0.298039215686275, 0.843137254901961, 0.878431372549020}
\definecolor{alt_col109}{rgb}{0.000000000000000, 0.117647058823529, 0.878431372549020}
\definecolor{alt_col110}{rgb}{0.760784313725490, 0.000000000000000, 0.709803921568627}
\definecolor{alt_col111}{rgb}{1.00000000000000, 0.341176470588235, 0.517647058823529}
\definecolor{alt_col112}{rgb}{0.698039215686274, 0.360784313725490, 0.239215686274510}
\definecolor{alt_col113}{rgb}{0.878431372549020, 0.843137254901961, 0.298039215686275}
\definecolor{alt_col114}{rgb}{0.349019607843137, 0.819607843137255, 0.278431372549020}
\definecolor{alt_col115}{rgb}{0.000000000000000, 0.800000000000000, 1.00000000000000}
\definecolor{alt_col116}{rgb}{0.000000000000000, 0.101960784313725, 0.760784313725490}
\definecolor{alt_col117}{rgb}{1.00000000000000, 0.341176470588235, 0.956862745098039}
\definecolor{alt_col118}{rgb}{0.819607843137255, 0.278431372549020, 0.423529411764706}
\definecolor{alt_col119}{rgb}{1.00000000000000, 0.400000000000000, 0.000000000000000}
\definecolor{alt_col120}{rgb}{0.760784313725490, 0.725490196078431, 0.258823529411765}
\definecolor{alt_col121}{rgb}{0.298039215686275, 0.701960784313725, 0.239215686274510}
\definecolor{alt_col122}{rgb}{0.000000000000000, 0.654901960784314, 0.819607843137255}
\definecolor{alt_col123}{rgb}{0.298039215686275, 0.376470588235294, 0.878431372549020}
\definecolor{alt_col124}{rgb}{0.701960784313725, 0.239215686274510, 0.670588235294118}
\definecolor{alt_col125}{rgb}{0.698039215686274, 0.239215686274510, 0.360784313725490}
\definecolor{alt_col126}{rgb}{0.819607843137255, 0.329411764705882, 0.000000000000000}
\definecolor{alt_col127}{rgb}{0.933333333333333, 1.00000000000000, 0.000000000000000}
\definecolor{alt_col128}{rgb}{0.317647058823529, 0.941176470588235, 0.321568627450980}
\definecolor{alt_col129}{rgb}{0.000000000000000, 0.560784313725490, 0.701960784313725}
\definecolor{alt_col130}{rgb}{0.341176470588235, 0.341176470588235, 1.00000000000000}
\definecolor{alt_col131}{rgb}{1.00000000000000, 0.000000000000000, 0.800000000000000}
\definecolor{alt_col132}{rgb}{1.00000000000000, 0.000000000000000, 0.133333333333333}
\definecolor{alt_col133}{rgb}{0.698039215686274, 0.278431372549020, 0.000000000000000}
\definecolor{alt_col134}{rgb}{0.819607843137255, 0.878431372549020, 0.000000000000000}
\definecolor{alt_col135}{rgb}{0.000000000000000, 0.760784313725490, 0.101960784313725}
\definecolor{alt_col136}{rgb}{0.341176470588235, 0.866666666666667, 1.00000000000000}
\definecolor{alt_col137}{rgb}{0.278431372549020, 0.278431372549020, 0.819607843137255}
\definecolor{alt_col138}{rgb}{0.819607843137255, 0.000000000000000, 0.654901960784314}
\definecolor{alt_col139}{rgb}{0.819607843137255, 0.000000000000000, 0.109803921568627}
\definecolor{alt_col140}{rgb}{1.00000000000000, 0.603921568627451, 0.341176470588235}
\definecolor{alt_col141}{rgb}{0.709803921568627, 0.760784313725490, 0.000000000000000}
\definecolor{alt_col142}{rgb}{0.000000000000000, 1.00000000000000, 0.266666666666667}
\definecolor{alt_col143}{rgb}{0.278431372549020, 0.713725490196078, 0.819607843137255}
\definecolor{alt_col144}{rgb}{0.239215686274510, 0.239215686274510, 0.701960784313725}
\definecolor{alt_col145}{rgb}{0.701960784313725, 0.000000000000000, 0.560784313725490}
\definecolor{alt_col146}{rgb}{0.698039215686274, 0.000000000000000, 0.0941176470588235}
\definecolor{alt_col147}{rgb}{0.819607843137255, 0.494117647058824, 0.278431372549020}
\definecolor{alt_col148}{rgb}{0.956862745098039, 1.00000000000000, 0.341176470588235}
\definecolor{alt_col149}{rgb}{0.000000000000000, 1.00000000000000, 0.400000000000000}
\definecolor{alt_col150}{rgb}{0.239215686274510, 0.607843137254902, 0.701960784313725}
\definecolor{alt_col151}{rgb}{0.517647058823529, 0.341176470588235, 1.00000000000000}
\definecolor{alt_col152}{rgb}{1.00000000000000, 0.341176470588235, 0.866666666666667}
\definecolor{alt_col153}{rgb}{1.00000000000000, 0.341176470588235, 0.427450980392157}
\definecolor{alt_col154}{rgb}{0.698039215686274, 0.423529411764706, 0.239215686274510}
\definecolor{alt_col155}{rgb}{0.843137254901961, 0.878431372549020, 0.298039215686275}
\definecolor{alt_col156}{rgb}{0.000000000000000, 0.878431372549020, 0.352941176470588}
\definecolor{alt_col157}{rgb}{0.000000000000000, 0.666666666666667, 1.00000000000000}
\definecolor{alt_col158}{rgb}{0.392156862745098, 0.258823529411765, 0.760784313725490}
\definecolor{alt_col159}{rgb}{0.819607843137255, 0.278431372549020, 0.709803921568627}
\definecolor{alt_col160}{rgb}{0.819607843137255, 0.278431372549020, 0.349019607843137}
\definecolor{alt_col161}{rgb}{1.00000000000000, 0.533333333333333, 0.000000000000000}
\definecolor{alt_col162}{rgb}{0.725490196078431, 0.760784313725490, 0.258823529411765}
\definecolor{alt_col163}{rgb}{0.000000000000000, 0.760784313725490, 0.305882352941176}
\definecolor{alt_col164}{rgb}{0.000000000000000, 0.505882352941176, 0.760784313725490}
\definecolor{alt_col165}{rgb}{0.400000000000000, 0.000000000000000, 1.00000000000000}
\definecolor{alt_col166}{rgb}{1.00000000000000, 0.000000000000000, 0.666666666666667}
\definecolor{alt_col167}{rgb}{0.698039215686274, 0.239215686274510, 0.298039215686275}
\definecolor{alt_col168}{rgb}{0.819607843137255, 0.439215686274510, 0.000000000000000}
\definecolor{alt_col169}{rgb}{0.800000000000000, 1.00000000000000, 0.000000000000000}
\definecolor{alt_col170}{rgb}{0.321568627450980, 0.941176470588235, 0.568627450980392}
\definecolor{alt_col171}{rgb}{0.341176470588235, 0.780392156862745, 1.00000000000000}
\definecolor{alt_col172}{rgb}{0.352941176470588, 0.000000000000000, 0.878431372549020}
\definecolor{alt_col173}{rgb}{0.819607843137255, 0.000000000000000, 0.545098039215686}
\definecolor{alt_col174}{rgb}{0.698039215686274, 0.372549019607843, 0.000000000000000}
\definecolor{alt_col175}{rgb}{0.705882352941177, 0.878431372549020, 0.000000000000000}
\definecolor{alt_col176}{rgb}{0.000000000000000, 1.00000000000000, 0.533333333333333}
\definecolor{alt_col177}{rgb}{0.278431372549020, 0.639215686274510, 0.819607843137255}
\definecolor{alt_col178}{rgb}{0.305882352941176, 0.000000000000000, 0.760784313725490}
\definecolor{alt_col179}{rgb}{0.701960784313725, 0.000000000000000, 0.466666666666667}
\definecolor{alt_col180}{rgb}{1.00000000000000, 0.690196078431373, 0.341176470588235}
\definecolor{alt_col181}{rgb}{0.607843137254902, 0.760784313725490, 0.000000000000000}
\definecolor{alt_col182}{rgb}{0.000000000000000, 0.819607843137255, 0.439215686274510}
\definecolor{alt_col183}{rgb}{0.239215686274510, 0.545098039215686, 0.701960784313725}
\definecolor{alt_col184}{rgb}{0.533333333333333, 0.298039215686275, 0.878431372549020}
\definecolor{alt_col185}{rgb}{1.00000000000000, 0.341176470588235, 0.780392156862745}
\definecolor{alt_col186}{rgb}{0.819607843137255, 0.568627450980392, 0.278431372549020}
\definecolor{alt_col187}{rgb}{0.866666666666667, 1.00000000000000, 0.341176470588235}
\definecolor{alt_col188}{rgb}{0.000000000000000, 0.701960784313725, 0.372549019607843}
\definecolor{alt_col189}{rgb}{0.000000000000000, 0.533333333333333, 1.00000000000000}
\definecolor{alt_col190}{rgb}{0.690196078431373, 0.341176470588235, 1.00000000000000}
\definecolor{alt_col191}{rgb}{0.819607843137255, 0.278431372549020, 0.639215686274510}
\definecolor{alt_col192}{rgb}{0.698039215686274, 0.486274509803922, 0.239215686274510}
\definecolor{alt_col193}{rgb}{0.764705882352941, 0.878431372549020, 0.298039215686275}
\definecolor{alt_col194}{rgb}{0.321568627450980, 0.941176470588235, 0.650980392156863}

\newcommand{\altcirc}[1]{\tikz\draw[alt_col#1,fill=alt_col#1] (0,0) circle (.75ex);}
\usepackage{colortbl}
\def\a{\alpha}

\def\la{\lambda}
\def\lam{\lambda}

\def\CC{{\mathbb C}}
\allowdisplaybreaks

\def\C{\mathcal{C}}
\def\A{\mathcal{A}}

\def\Z{\mathbb{Z}}
\def\N{\mathbb{N}}

\def\C{\mathbb{C}}
\def\sl4{\mathfrak{sl}_4(\mathbb{C})}
\def\a{\alpha}
\def\w{\varpi}
\def\al{\alpha}
\def\fg{\mathfrak{g}}

\def\la{\lambda}

\def\jmid{J}
\makeatletter
\newtheorem*{rep@theorem}{\rep@title}
\newcommand{\newreptheorem}[2]{%
\newenvironment{rep#1}[1]{%
 \def\rep@title{#2 \ref{##1}}%
 \begin{rep@theorem}}%
 {\end{rep@theorem}}}

 \newtheorem*{genericthm*}{\thistheoremname}
\newenvironment{namedthm*}[1]
  {\renewcommand{\thistheoremname}{#1}%
   \begin{genericthm*}}
  {\end{genericthm*}}
\makeatother

\usepackage{array}

\newcommand{\zcase}[4]{
\begin{enumerate}
\item If $#2, #4\geq #3$, then
    \begin{align*}
        #1= \wp_q(\xi)&=\sum_{i=#2+#4-#3}^{#2+#3+#4} (L+1)(L+(t\bmod2)+1)q^i
    \end{align*}
    where
    \begin{align*}
    L&=\min\left\{\floor{\frac{t}{2}}, #3-\ceil{\frac{t}{2}}\right\} \text{ and } t=#2+#3+#4-i.
    \end{align*}
\item If $#2\geq#3\geq#4$, then
    \begin{align*} 
        #1=\wp_q(\xi)&=\sum_{i=#2}^{#2+#3+#4} c_iq^i
    \end{align*} 
    where
    \begin{equation*}\resizebox{.9\hsize}{!}{$       c_i= 
            \begin{cases}
                \dfrac{(L+1)(2#4-2F_2+L+2)}{2} & \mbox{if } J < 0 \\ \\
                \dfrac{(L+1)(L+2) - F_2(F_2+1) - #4(#4-1)}{2}+F_2 #4 + #4L- F_2 L + (#4-F_2)(t\bmod 2) & \mbox{if } 0\leq J\leq L \\ \\
                (L+1)(L+(t\bmod 2)+1) &\mbox{if } J > L
            \end{cases}
    $}\end{equation*}
    with $t=#2+#3+#4-i$,
    $F_2=\min\left\{\floor{\frac{t}{2}},#4\right\}, J=#4-F_2-(t\bmod 2), F_1 =\max\{t-#3,0\}$, and
    $L = F_2 - F_1$.
\item If $#4\geq#3\geq#2$, then
    \begin{align*}
        #1= \wp_q(v)&=\sum_{i=#4}^{#2+#3+#4} c_iq^i
    \end{align*}
    where
    \begin{equation*}\resizebox{.9\hsize}{!}{$
        c_i = 
        \begin{cases}
        \dfrac{(L+1)(2#2-2F_2+L+2)}{2} & \text{if } J < 0 \\ \\
        \dfrac{(L+1)(L+2) - F_2(F_2+1) - #2(#2-1)}{2}
         + F_2#2+#2L-F_2L+(#2-F_2)(t\bmod 2)
        & \text{if } 0\leq J\leq L \\ \\
        (L+1)(L+(t\bmod 2)+1) & \text{if } J > L
        \end{cases}$}
    \end{equation*}
    with $t=#2+#3+#4-i$,
    $F_2=\min\left\{\floor{\frac{t}{2}},#2\right\}, J=#2-F_2-(t\bmod2), F_1=\max\{t-#3,0\}$, and $L = F_2 - F_1$.
\item If $#3\geq#2\geq#4$, then
    \begin{align*}
        #1= \wp_q(\xi) &= \sum_{i=#3}^{#2+#3+#4} c_iq^i\ \
        \text{ and }\ c_i=S_1 - S_2 + F_2 - F_1 + 1
    \end{align*}
    where
    \begin{equation*}
       \resizebox{.9\hsize}{!}{$S_1=\begin{cases}
           (t-F_2)(F_2+1) & \text{if } t-#2 < 0 \\\\
           \dfrac{(t-#2)(t-#2+1)+F_1(F_1-1) - 2F_2(F_2+1) + 2#2(t-#2-F_1+1) + 2t(F_2-t+#2)}{2} & \text{if } 0\leq t-#2\leq F_2
        \end{cases}$}
    \end{equation*}
    \begin{equation*}        \resizebox{.9\hsize}{!}{$S_2=\begin{cases}
            0 & \text{if } t-#4 < 0 \\\\
         (t-#4)(t-#4-F_1+1) - \dfrac{(t-#4)(t-#4+1)-F_1(F_1-1)}{2}& \text{if } 0\leq t-#4\leq F_2     \\\\
           (t-#4)(F_2-F_1+1) - \dfrac{F_2(F_2+1)-F_1(F_1-1)}{2}
            & \text{if } t-#4 > F_2
       \end{cases}$}
    \end{equation*}
    with $t = #2+#3+#4-i, F_1=\max\{0,t-\min\{#2+#4,#3\}\},$ and $F_2 = \min\left\{\floor{\frac{t}{2}},#4\right\}.$
\item If $#3\geq#4\geq#2$, then
    \begin{align*}
         #1 &= \wp_q(\xi) = \sum_{i=#3}^{#2+#3+#4} c_iq^i \ \text{and} \ c_i=S_1-S_2 + F_2 - F_1 + 1
    \end{align*}
    where
    \begin{equation*}
    \resizebox{.9\hsize}{!}{$
       S_1 = 
       \begin{cases}
           (t-F_2)(F_2+1) & \text{if } t-#4 < 0 \\\\
           \dfrac{(t-#4)(t-#4+1)+F_1(F_1-1) - 2F_2(F_2+1) + 2#4(t-#4-F_1+1) + 2t(F_2-t+#4)}{2} & \text{if } 0\leq t-#4\leq F_2
        \end{cases}$}
    \end{equation*}
    \begin{equation*}\resizebox{.9\hsize}{!}{$S_2 = 
        \begin{cases}
            0 & \text{if } t-#2 < 0 \\\\
            (t-#2)(t-#2-F_1+1) - \dfrac{(t-#2)(t-#2+1)-F_1(F_1-1)}{2} 
            & \text{if } 0\leq t-#2\leq F_2 \\\\
            (t-#2)(F_2-F_1+1) - \dfrac{F_2(F_2+1)-F_1(F_1-1)}{2}
            & \text{if } t-#2 > F_2
       \end{cases}$}
    \end{equation*}
    with $t= #2+#3+#4-i, F_1= \max\{0,t-\min\{#2+#4,#3\}\}, \text{ and } F_2 = \min\left\{\floor{\frac{t}{2}},#2\right\}.$
\end{enumerate}
}

\newcommand{\floor}[1]{\left\lfloor #1 \right\rfloor}
\newcommand{\ceil}[1]{\left\lceil #1 \right\rceil}
\newcommand{\partn}[6]{#1\al_1+#2\al_2+#3\al_3+#4(\al_1+\al_2)+#5(\al_2+\al_3)+#6(\al_1+\al_2+\al_3)}
\newcommand{\mpartn}[3]{#1\cdot \one + #2\cdot\onep + #3\cdot2}
\newcommand{\one}{\textcolor{red}{1_{r}}}
\newcommand{\onep}{\textcolor{blue}{1_{b}}}
\usepackage{relsize}

\makeatletter
\newtheorem*{rep@conjecture}{\rep@title}
\newcommand{\newrepconjecture}[2]{%
\newenvironment{rep#1}[1]{%
 \def\rep@title{#2 \ref{##1}}%
 \begin{rep@conjecture}}%
 {\end{rep@conjecture}}}
\makeatother

\makeatletter
\newcommand{\addresseshere}{%
  \enddoc@text\let\enddoc@text\relax
}
\makeatother

\theoremstyle{definition}

\newtheorem{definition}{Definition}
\newtheorem{problem}{Problem}

\newtheorem{theorem}{Theorem}
\newreptheorem{theorem}{Theorem}
\newtheorem{proposition}{Proposition}

\newtheorem{example}{Example}
\newtheorem{lemma}{Lemma}
\newtheorem{corollary}{Corollary}
 
\newrepconjecture{conjecture}{Conjecture}

\theoremstyle{plain} 
\newcommand{\thistheoremname}{}
\newtheorem{genericthm}[theorem]{\thistheoremname}

\title{On Kostant's weight $q$-multiplicity formula for $\mathfrak{sl}_{4}(\mathbb{C})$
}

\author{Rebecca E. Garcia}
\address{Sam Houston State University, Department of Mathematics, United States}
\email{\textcolor{blue}{\href{mailto:math\_reg@shsu.edu}{math\_reg@shsu.edu}}}

\author{Pamela E. Harris}
\address{Department of Mathematics and Statistics, Williams College, United States}
\email{\textcolor{blue}{\href{mailto:peh2@williams.edu}{peh2@williams.edu}}}

\author{Marissa Loving}
\address{Department of Mathematics, University of Illinois at Urbana-Champaign, United States}
\email{\textcolor{blue}{\href{mailto:mloving2@illinois.edu}{mloving2@illinois.edu}}}

\author{Lucy Martinez}
\address{Stockton University, United States}
\email{\textcolor{blue}{\href{mailto:marti310@go.stockton.edu}{marti310@go.stockton.edu}}}

\author{David Melendez}
\address{Department of Mathematics, University of Central Florida, United States}
\email{\textcolor{blue}{\href{mailto:davidmelendez@knights.ucf.edu}{davidmelendez@knights.ucf.edu}}}

\author{Joseph Rennie}
\address{Department of Mathematics, University of Illinois at Urbana-Champaign, United States}
\email{\textcolor{blue}{\href{mailto:rennie2@illinois.edu}{rennie2@illinois.edu}}}

\author{Gordon Rojas Kirby}
\address{UC Santa Barbara, Department of Mathematics, United States}
\email{\textcolor{blue}{\href{mailto:gkirby@math.ucsb.edu}{gkirby@math.ucsb.edu}}}

\author{Daniel Tinoco}
\address{Department of Mathematics, San Francisco State University, United States}
\email{\textcolor{blue}{\href{mailto:dtinoco@mail.sfsu.edu}{dtinoco@mail.sfsu.edu}}}

\date{\today}

\begin{document}
\begin{abstract}
    The $q$-analog of Kostant's weight multiplicity formula is an alternating sum over a finite group, known as the Weyl group, whose terms involve the $q$-analog of Kostant's partition function. This formula, when evaluated at $q=1$, gives the multiplicity of a weight in a highest weight representation of a simple Lie algebra. 
    In this paper, we consider the Lie algebra $\mathfrak{sl}_4(\mathbb{C})$ and give closed formulas for the $q$-analog of Kostant's weight multiplicity. This formula depends on the following two sets of results. First, we present closed formulas for the $q$-analog of Kostant's partition function by counting restricted colored integer partitions. These formulas, when evaluated at $q=1$, recover results of De Loera and Sturmfels. 
    Second, we describe and enumerate the Weyl alternation sets, which consist of the elements of the Weyl group that contribute nontrivially to Kostant's weight multiplicity formula. From this, we introduce Weyl alternation diagrams on the root lattice of $\mathfrak{sl}_4(\mathbb{C})$, which are associated to the Weyl alternation sets. This work answers a question posed in 2019 by Harris, Loving, Ramirez, Rennie, Rojas Kirby, Torres Davila, and Ulysse.  
\end{abstract}

\maketitle

\section{Introduction}
Let $\mathfrak g$ be a simple Lie algebra of rank $r$ and $\mathfrak h$ be a Cartan subalgebra of $\mathfrak g$. We let $\Phi$ denote the set of roots corresponding to $(\mathfrak g, \mathfrak h)$, $\Phi^+ \subset \Phi$ denote a set of positive roots, and $\Delta  \subset \Phi^+$ denote a set of simple roots. Throughout, we let $\a_1,\a_2,\ldots,\a_r$ denote the simple roots, $\w_1,\w_2,\ldots,\w_r$ denote the fundamental weights, $P(\mathfrak{g}):=\{c_1\w_1+\cdots+c_r\w_r:c_1,\ldots,c_r\in\mathbb{Z}\}$ denote the set of integral weights, and $P_+(\mathfrak{g}):=\{c_1\w_1+\cdots+c_r\w_r:c_1,\ldots,c_r\in\mathbb{N}\}$ denote the set of dominant integral weights. The theorem of the highest weight asserts that every dominant integral weight $\lambda\in P_+(\mathfrak{g})$ is the highest weight of an irreducible finite-dimensional representation of $\mathfrak{g}$, which we denote by $L(\lambda)$. For general references on Lie theory, see \cite{GW,varadarajan}.

In this paper, we consider the $q$-analog of Kostant's weight multiplicity formula defined by Lusztig, which is defined in \cite{lus} as follows:
\begin{equation}
m_q( \lambda, \mu) = \sum_{ \sigma \in W} (-1)^{ \ell (\sigma)}   \wp_q( \sigma ( \lambda + \rho ) - ( \mu + \rho)).
\label{eq:q-analog} 
\end{equation}
Here $\rho$ is equal to half the sum of the positive roots, $W$ is the Weyl group associated to $(\mathfrak g, \mathfrak h)$ 
(which is generated by reflections orthogonal to the simple roots),
$\ell(\sigma)$ denotes the length of $\sigma\in W$ (which is the minimum nonnegative integer $k$ such that $\sigma$ is a product of $k$ reflections), and $\wp_q$ denotes the $q$-analog of Kostant's partition function defined on $\xi\in \mathfrak h^*$ by $\wp_q(\xi) = \sum c_iq^i$ with $c_i$ representing the number of ways to write the weight $\xi$ as a sum of $i$ positive roots. In this way, 
$\wp_q(\xi)|_{q=1}=\wp(\xi)$ is Kostant's partition function, which counts 
number of ways of expressing the weight $\xi$ as a nonnegative integral sum of positive roots. Thus, $m_{q}(\lambda,\mu)|_{q=1}$
gives the multiplicity of a weight $\mu$ in a highest weight representation $L(\lambda)$ of $\mathfrak{g}$. For a detailed account of weight multiplicity computations, we point the reader to \cite{PHcomputingweightmultiplicities}.

A main challenge in using \eqref{eq:q-analog} for computations is that formulas for the partition function, $\wp_q$, do not exist in much generality. 
Moreover, for a Lie algebra of rank $r$, the number of terms appearing in this sum is factorial in the rank. 
Thus, many have worked to determine closed formulas for the partition function and its $q$-analog, including low rank examples \cite{deloera,HarrisLauber,mars1,mars2} and for specific families of inputs \cite{PHThesis,harris2011,HIO}. Other work was motivated by the observation that in practice, many terms appearing in computations involving Kostant's weight multiplicity formula are zero \cite{Cochet}.
Hence, it is of interest to determine the Weyl group elements whose associated term contributes nontrivially to the sum. More precisely, the Weyl alternation set is
\[\mathcal{A}(\lambda,\mu):=\{\sigma \in W: \wp (\sigma(\lambda + \rho) - (\mu + \rho)) > 0\}.\]

Weyl alternation sets have been studied in the following cases: 
    $\lambda$ and $\mu$ are pairs of weights such that $m(\lambda,\mu)=1$, see  \cite{HRSS};
    $\lambda$ is the highest root of a Lie algebra and $\mu$ is either zero or a positive root, see \cite{PHThesis,HIS,HIW}; and
$\lambda$ is the sum of the simple roots of a classical Lie algebra and $\mu$ is either zero or a positive root, see \cite{CHI}.
In the last two cases, the cardinality of the Weyl alternation set follows a recurrence relation with constant coefficients. For the Lie algebras of type $A$ and $C$ these recurrence relations define the Fibonacci numbers or (multiples) of the Lucas numbers. 

Some interesting geometric behavior is also exhibited by these Weyl alternation sets. For example, Harris, Lescinsky, and Mabie provided some fundamental weight lattice patterns, called ``Weyl alternation diagrams'', describing the Weyl alternation sets for the Lie algebra $\mathfrak{sl}_{3}(\mathbb{C})$ \cite{HLM}. This work was extended to all rank two Lie algebras in \cite{HLRRRKTDU}.  
Figure \ref{fig:G2pics} provides visualizations for the Weyl alternation sets $\A(\lambda,\mu)$, for a fixed weight $\mu$, of the exceptional Lie algebra of type $G_2$ (as presented in \cite{HLRRRKTDU}). In these figures, each color-coded region represents a set of weights $\lambda$ which have the same Weyl alternation set. Note that the interior uncolored region is called the \textit{empty region}, since these are weights $\lambda$ for which $\A(\lambda,\mu)=\emptyset$, and these regions are completely determined by~$\mu$. 

\begin{figure}[h]
    \includegraphics[width=\textwidth]{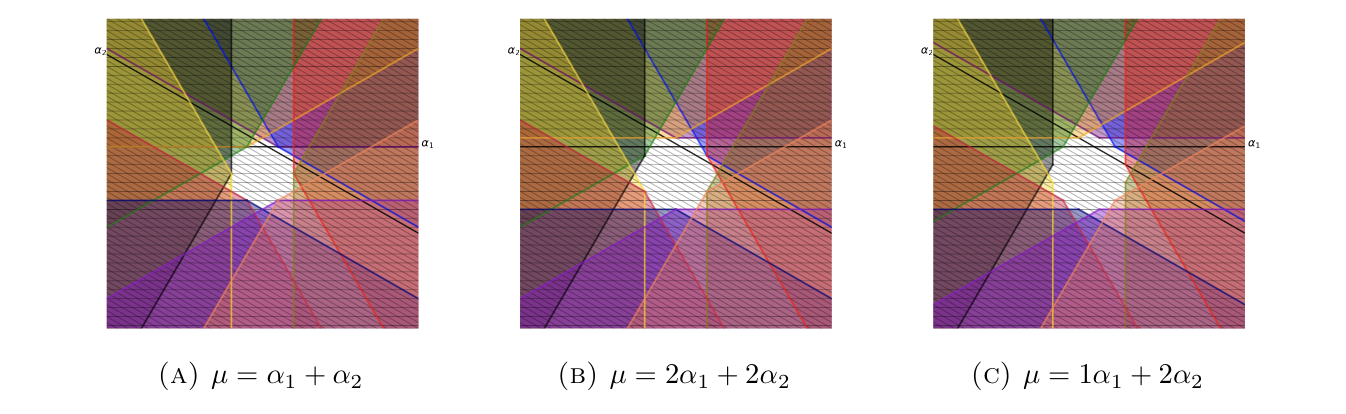}
    \vspace{-7mm}
    \caption{Weyl alternation diagrams of $\fg_2$ with a specified $\mu$.}
    \label{fig:G2pics}
\end{figure}

In the present work, we  consider the Lie algebra $\sl4$ and we
\begin{enumerate}
    \item 
    give closed formulas for the $q$-analog of Kostant's partition function for any weight $\xi$,\label{qpart}
        \item compute the Weyl alternation sets for every pair of integral weights $(\lambda,\mu)$, \label{weyl alt sets} 
        \item illustrate the Weyl alternation diagrams when $\mu=0$ and $\lambda=a_1\a_1+a_2\a_2+a_3\a_3$ with $a_1,a_2,a_3\in\N$, as well as describe the empty region for a variety of weights $\mu$,
    \item 
    use \eqref{qpart} and \eqref{weyl alt sets} to give a closed formula for the $q$-analog of Kostant's weight multiplicity formula for $\sl4$, and
    \item provide code for all of the above mentioned results, which can be found in the GitHub repository at \textcolor{blue}{\href{https://github.com/melendezd/Weight-Multiplicities}{https://github.com/melendezd/Weight-Multiplicities}}.
\end{enumerate}

\subsection{Statements of main results and some specializations}
Throughout, we let $\a_1$, $\a_2$, and $\a_3$ denote the simple roots, and $\w_1$, $\w_2$, and $\w_3$ denote the fundamental weights of $\sl4$.
Harris, Rahmoeller, Schneider, and Simpson gave the following formula for the $q$-analog of Kostant's partition function for the Lie algebra $\mathfrak{sl}_4(\mathbb{C})$. 

\begin{proposition}[Proposition 5.1 in \cite{HRSS}]\label{thm:sum} 
If $m,n,k\in\N$, then
    \[
    \wp_q(m\alpha_1+n\alpha_2+k\alpha_3)=\sum_{f=0}^{\min(m,n,k)}\; \sum_{d=0}^{\min(m-f,n-f)}\; \sum_{e=0}^{\min(n-d-f,k-f)} q^{m+n+k-2f-d-e}.
    \]
\end{proposition}

The formula in Proposition \ref{thm:sum} can be readily implemented in a computer program. However, as $m,n,$ and $k$ grow, the runtime for this algorithm grows rapidly. This motivates our first result. 

\begin{theorem}\label{thm:main1}
Let $m,n,k\in \mathbb{N}:=\{0,1,2,\ldots\}$ and $\xi=m\a_1+n\a_2+k\a_3$.
\begin{enumerate}[leftmargin=.6cm,itemsep=2pt]
  \item\label{thm:main1:part1} If $m,k \geq n$, then
  $\wp_q(\xi)=\sum_{i=m+k-n}^{m+n+k} (L+1)(L+(t\bmod2)+1)q^i$ where $L=\min\left\{\floor{\frac{t}{2}}, n-\ceil{\frac{t}{2}}\right\}$ and  $t=m+n+k-i$.
  
 \item\label{thm:main1:part2} If $m\geq n \geq k$, then
  $\wp_q(\xi)=\sum_{i=m}^{m+n+k} c_iq^i$ where 
  \begin{equation*}
        c_i = 
        \begin{cases}
        \frac{(L+1)(2k-2F_2+L+2)}{2} &\mbox{if}\ J < 0 \\
        \frac{(L+1)(L+2)-F_2(F_2+1)-k(k-1)}{2} 
        +F_2 k+kL- F_2 L+(k-F_2)(t\bmod2)
        & \mbox{if}\ 0\leq J\leq L \\
        (L+1)(L+(t\bmod2)+1) &\mbox{if}\ J > L
        \end{cases}
    \end{equation*}
    with $t=m+n+k-i$, 
    $F_2=\min\left\{\floor{\frac{t}{2}},k\right\}$,
    $J=k-F_2-(t\bmod2)$,
    $F_1 =\max\{t-n,0\}$, and
    $L = F_2 - F_1$.
\item\label{thm:main1:part3} If $k\geq n \geq m$, then
  $\wp_q(\xi)=\sum_{i=k}^{m+n+k} c_iq^i$ where 
  \begin{equation*}
        c_i = 
        \begin{cases}
        \frac{(L+1)(2m-2F_2+L+2)}{2} & \mbox{if}\ J < 0 \\
        \frac{(L+1)(L+2)-F_2(F_2+1)-m(m-1)}{2} 
        +F_2m+mL-F_2L+(m-F_2)(t\bmod2)
        &\mbox{if}\ 0\leq J\leq L \\
        (L+1)(L+(t\bmod2)+1) &\mbox{if}\ J > L
        \end{cases}
    \end{equation*}
    with $t=m+n+k-i$, 
    $F_2=\min\left\{\floor{\frac{t}{2}},m\right\}$,
    $J=m-F_2-(t\bmod2)$,
    $F_1=\max\{t-n,0\}$, and
    $L = F_2 - F_1$.
 \item\label{thm:main1:part4} If $n \geq m \geq k$, then
  $\wp_q(\xi) = \sum_{i=n}^{m+n+k} c_iq^i$ where $c_i=S_1-S_2 + F_2 - F_1 + 1$ and    \begin{align*}
        S_1 &= 
      \begin{cases}
          (t-F_2)(F_2+1) &\mbox{if}\ t-m < 0 \\
          \frac{(t-m)(t-m+1)+F_1(F_1-1)- 2F_2(F_2+1) + 2m(t-m-F_1+1) + 2t(F_2-t+m)}{2}  &\mbox{if}\ 0\leq t-m\leq F_2
        \end{cases} \\
          S_2 &= 
        \begin{cases}
         (t-k)(t-k-F_1+1) - \frac{(t-k)(t-k+1)-F_1(F_1-1)}{2}&\mbox{if}\ 0\leq t-k\leq F_2     \\
          (t-k)(F_2-F_1+1) - \frac{F_2(F_2+1)-F_1(F_1-1)}{2}
            &\mbox{if}\ t-k > F_2\\
            0 &\mbox{if}\ t-k < 0 
      \end{cases}
         \end{align*}
         with $t = m+n+k-i$, $F_1=\max\{0,t-\min\{m+k,n\}\}$, and $F_2 = \min\left\{\floor{\frac{t}{2}},k\right\}$.
\item\label{thm:main1:part5} If $n\geq k\geq m$, then
  $\wp_q(\xi) = \sum_{i=n}^{m+n+k} c_iq^i$ where $c_i=S_1-S_2 + F_2 - F_1 + 1$ and 
    \begin{align*}
        S_1 &= 
      \begin{cases}
          (t-F_2)(F_2+1) &\mbox{if}\ t-k < 0 \\
          \frac{(t-k)(t-k+1)+F_1(F_1-1)- 2F_2(F_2+1)
          + 2k(t-k-F_1+1) + 2t(F_2-t+k)}{2}  &\mbox{if}\ 0\leq t-k\leq F_2 
        \end{cases} \\
          S_2 &= 
        \begin{cases}
            (t-m)(t-m-F_1+1) - \frac{(t-m)(t-m+1)-F_1(F_1-1)}{2}
            &\mbox{if}\ 0\leq t-m\leq F_2 \\(t-m)(F_2-F_1+1) - \frac{F_2(F_2+1)-F_1(F_1-1)}{2}
            &\mbox{if}\ t-m > F_2\\
            0 &\mbox{if}\ t-m < 0
      \end{cases}
    \end{align*}
with $t = m+n+k-i$,
        $F_1 = \max\{0,t-\min\{m+k,n\}\}$, and
        $F_2 = \min\left\{\floor{\frac{t}{2}},m\right\}$.
\end{enumerate}
\end{theorem}

Our proof of Theorem \ref{thm:main1} counts restricted integer partitions
with parts of multiple colors, thereby giving interesting connections between these types of partitions and vector partitions. Moreover, as mentioned above, we note that the runtime for computing 
$\wp_q(m\al_1+n\al_2+k\al_3)$ for $60\leq m,n,k < 70$ was $359$ms on average using an algorithm based on Proposition \ref{thm:sum}, while the same computations took less than $2$ms on average using the closed formulas in Theorem \ref{thm:main1}. 
In general, the runtime of an algorithm to compute each coefficient of $\wp_q(\xi)$ based on Proposition \ref{thm:sum} grows at least linearly with $\wp(\xi)$, while the runtime of an algorithm using Theorem \ref{thm:main1} remains constant. 
We also remark that evaluating the formulas in Theorem \ref{thm:main1} at $q =1$ recovers formulas of De Loera and Sturmfels \cite[Section 5]{deloera}, which we present below.

\begin{corollary}\label{corollary:deloeraresults}
Let $m,n,k\in \mathbb{N}$ and $\xi=m\a_1+n\a_2+k\a_3$.
   \begin{enumerate}[leftmargin=.7cm]
       \item\label{corollary:deloeraresults:part1} If
        $m,k\geq n$, then $\wp(\xi)=\frac{(n+1)(n+2)(n+3)}{6}$.
         \item\label{corollary:deloeraresults:part2} If $m\geq n\geq k$, then $\wp(\xi)=\frac{(k+1)(k+2)(k+3(n-k)+3)}{6}$.
         \item\label{corollary:deloeraresults:part3} If $k\geq n\geq m$, then $\wp(\xi)=\frac{(m+1)(m+2)(m+3(n-m)+3)}{6}$.
         \item\label{corollary:deloeraresults:part4} If $n\geq m\geq k$ and $n\geq m+k$, then $\wp(\xi)=\frac{(k+1)(k+2)(3m-k+3)}{6}$.
         \item\label{corollary:deloeraresults:part5} If $n\geq k\geq m$ and $n\geq m+k$, then $\wp(\xi)=\frac{(m+1)(m+2)(3k-m+3)}{6}$.
         \item\label{corollary:deloeraresults:part6} If $n\geq m\geq k$ and $m+k\geq n$, then 
         {\footnotesize{
         \[
        \wp(\xi) = 
        \frac{-2 k^3 - (m - n)^2 (3 + m - n) + 3 k^2 (n-1) +
   2 (3 + 2 m + n) + k (5 - 3 m^2 - 3 ( n-2) n + m (3 + 6 n))}{6}.
         \]
         }}
         
        \item\label{corollary:deloeraresults:part7} If $n\geq k\geq m$ and $m+k\geq n$, then 
        {\footnotesize{
        \begin{align*}
        \wp(\xi) = 
        \frac{-2 m^3 - (k - n)^2 (3 + k - n) + 3 m^2 (n-1) +
   2 (3 + 2 k + n) + m (5 - 3 k^2 - 3 ( n-2) n + k (3 + 6 n))}{6}.
        \end{align*}
        }}
   \end{enumerate}
\end{corollary}
We note that the formula in \cite[Item 4 (above Theorem 6.1 on page 14)]{deloera} is missing a factor of $1/6$, which we correct in Corollary \ref{corollary:deloeraresults} parts \eqref{corollary:deloeraresults:part4} and \eqref{corollary:deloeraresults:part5}. We also remark that in Corollary \ref{corollary:deloeraresults} there are 7 cases, while there are only 5 in Theorem~\ref{thm:main1}. This is because part \eqref{thm:main1:part5} in Theorem~\ref{thm:main1} encompasses cases \eqref{corollary:deloeraresults:part5} and \eqref{corollary:deloeraresults:part7} in Corollary \ref{corollary:deloeraresults}, and part \eqref{thm:main1:part4} in Theorem \ref{thm:main1} encompasses cases \eqref{corollary:deloeraresults:part4} and \eqref{corollary:deloeraresults:part6} in Corollary \ref{corollary:deloeraresults}.

Our second main result, Theorem~\ref{thm:195 alt sets} in Section \ref{sec:Weyl alt sets and diagrams},  describes the Weyl alternation sets $\A(\lambda, \mu)$ for any pair of integral weights $\lambda$ and $\mu$. This result establishes that there are a total of 195 distinct Weyl alternation sets and that the maximum cardinality among all Weyl alternation sets is $6$. Given the length of Theorem \ref{thm:195 alt sets} we defer its statement until Section \ref{sec:Weyl alt sets and diagrams}.  To illustrate Theorem \ref{thm:195 alt sets}, we consider $\lambda$ in the nonnegative octant of the root lattice of $\mathfrak{sl}_{4}(\mathbb{C})$ and obtain the following result.

\begin{corollary}\label{cor:13 alt sets}
If $\mu=0$ and $\lambda=(2x-y)\w_1 + (2y-x-z)\w_2 + (2z-y)\w_3$ with $x,y,z,2x-y,2y-x-z,2z-y \in \N$, then $\lambda$ is in the nonnegative integral root lattice $\N\a_1\oplus\N\a_2\oplus\N\a_3$ and
\begin{enumerate}[leftmargin=.7cm]
\item  $\A_1:=\A(\lambda,0)=\{1\}$  if  $\lam=0$,
\item  $\A_2:=\A(\lambda,0)=\{1,s_2\}$  if $x+z-y-1\in\N$,  $-1<x-y<2 $,  and $-1<z-y<2$,
\item $\A_3:=\A(\lambda,0)=\{1,s_2,s_3\}$  if $x+z-y-1,y-z-1\in\N$ ,  $x-y-2\notin\N $, and  $-1 < x-y < 2$,
\item   $\A_4:=\A(\lambda,0)=\{1,s_1,s_2\}$  if $y-x-1,x+z-y-1\in\N$,  $z-x-2\notin\N $,  and  $-1<z-y<2$,
\item $\A_5:=\A(\lambda,0)=\{1,s_2,s_3,s_2s_3\}$  if $x-z-2,x+z-y-1,y-z-1\in\N$,  and $-1<x-y<2$,
\item $\A_6:=\A(\lambda,0)=\{1,s_1,s_3,s_3s_1\}$  if $y-x-1,y-z-1\in\N$,   $z+z-y-1\notin\N$, and $-2<x-z<2 $,
\item $\A_7:=\A(\lambda,0)=\{1,s_1,s_2,s_2s_1\}$  if $y-x-1,z-x-2,x+z-y-1\in\N$, and  $-1<z-y<2$,
\item $\A_8:=\A(\lambda,0)=\{1,s_1,s_2,s_3,s_3s_1\}$  if $x+z-y-1,y-x-1,y-z-1\in\N$,  and $-2 < x-z < 2$, 
\item $\A_9:=\A(\lambda,0)=\{1,s_1,s_2,s_3,s_2s_3,s_3s_1\}$  if $x+z-y-1,y-x-1,x-z-2,y-z-1\in\N$, 
\item $\A_{10}:=\A(\lambda,0)=\{1,s_1,s_2,s_3,s_2s_1,s_3s_1\}$  if $x+z-y-1,y-x-1,z-x-2,y-z-1\in\N$, 
\item $\A_{11}:=\A(\lambda,0)=\{1,s_2,s_3,s_2s_3,s_3s_2,s_2s_3s_2\}$  if $x+z-y-1,x-z-2,x-y-2,y-z-1\in\N$,
\item $\A_{12}:=\A(\lambda,0)=\{1,s_1,s_2,s_1s_2,s_2s_1,s_1s_2s_1\}$  if $x+z-y-1,z-y-2,y-z-1,z-x-1\in\N$,
\item $\A(\lambda,0)=\emptyset$, otherwise. 
\end{enumerate}
\end{corollary}

In Figure~\ref{fig:diagram for 13 alt sets}, we illustrate the 12 nonempty Weyl alternation sets 
given by Corollary \ref{cor:13 alt sets}. In these figures, the axes are the nonnegative real span of the simple roots $\a_1$ (in red), $\a_2$ (in green), and $\a_3$ (in blue). For each $1\leq i\leq 12$, the colored vertices within each particular subfigure depict the a set 
of weights whose Weyl alternation set is indicated in the caption. \\

\begin{figure}[h]
    \centering
\includegraphics[width=\textwidth]{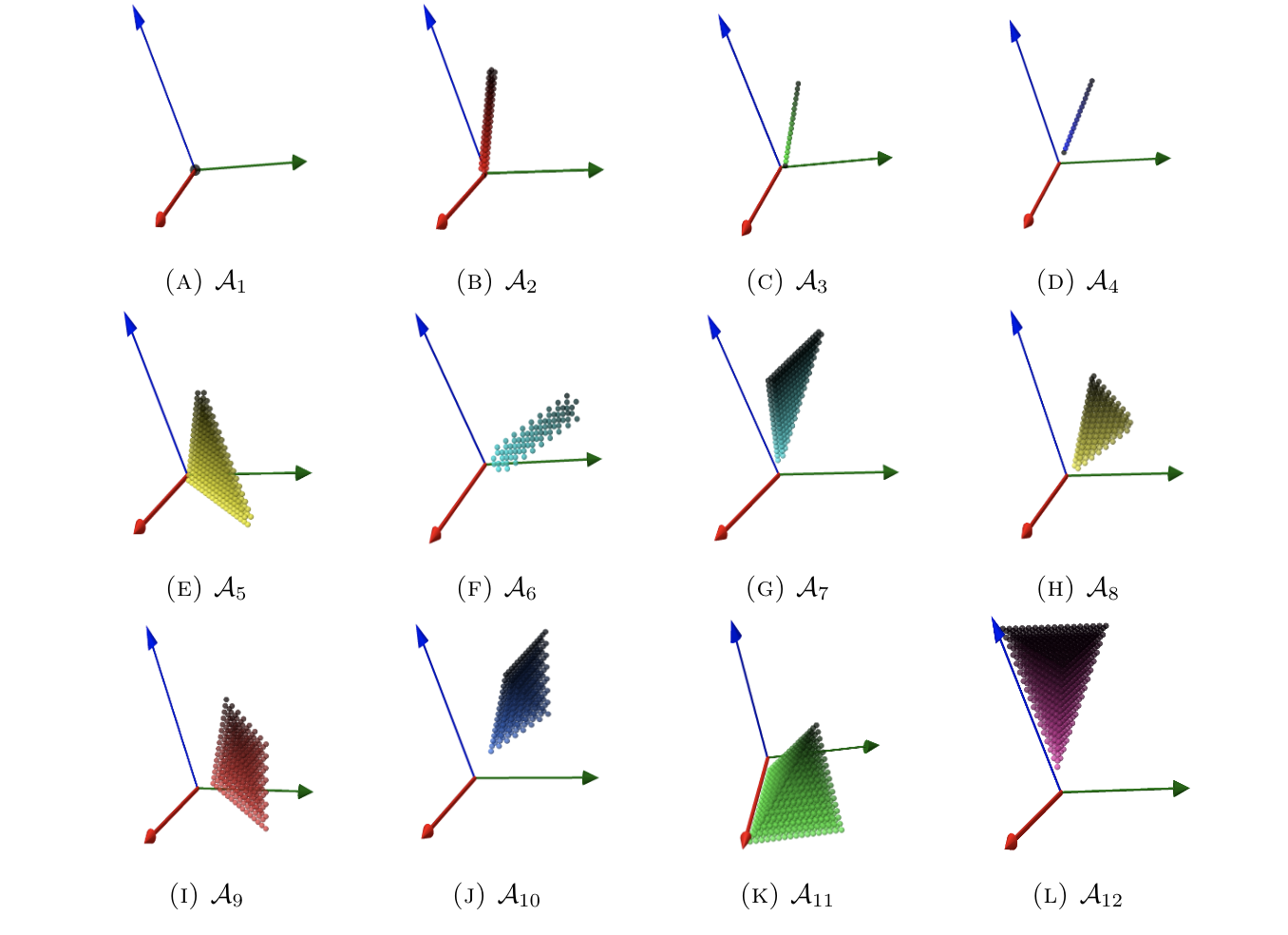}
\vspace{-7mm}
    \caption{The Weyl alternation diagrams of Corollary \ref{cor:13 alt sets}.}
    \label{fig:diagram for 13 alt sets}
\end{figure}
\vspace{-4mm}

In Figure \ref{fig:emptyregionexamples}, we illustrate some of the possible behavior of the central empty region for the Lie algebra $\sl4$.
For a fixed $\mu=c_1\a_1+c_2\a_2+c_3\a_3$, with $c_1,c_2,c_3\in\N$,
the colored vertices in Figure \ref{fig:emptyregionexamples} denote the weights $\lambda=m\w_1+n\w_2+k\w_2$, with $m,n,k\in\Z$, for which $\A(\lambda,\mu)=\emptyset$. Such regions are completely described by Theorem \ref{thm:195 alt sets}. Thus, Theorem \ref{thm:195 alt sets} 
answers \cite[Question 5.1]{HLRRRKTDU}, by giving the Weyl alternation diagrams of $\sl4$ along with a description of how the empty region diagrams change as the weight $\mu$ changes.

\begin{figure}[h]
    \centering
        \includegraphics[width=\textwidth]{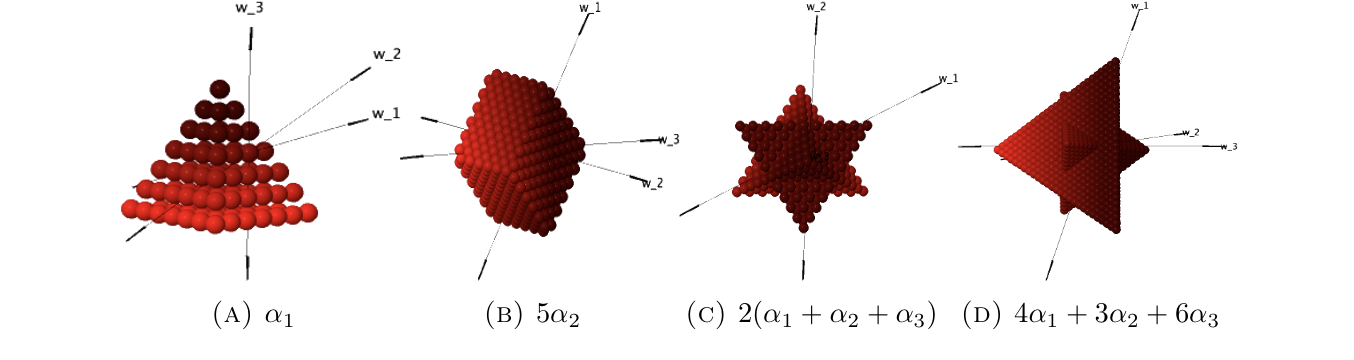}
    \caption{Various empty regions associated to $\sl4$ with a fixed weight $\mu$.}
    \label{fig:emptyregionexamples}
\end{figure}

Our last main result presents a closed formula for the $q$-analog of Kostant's weight multiplicity formula for the Lie algebra $\sl4$ for dominant integral weight $\lambda$ and $\mu$. 
This restriction is motivated by the fact that the 
weights which lie on the
nonnegative octant of the fundamental weight lattice are in correspondence with the finite-dimensional irreducible  representations of $\mathfrak{sl}_4(\mathbb{C})$. 
Theorem \ref{thm:qKWMF} utilizes the closed formulas for the $q$-analog of Kostant's partition function given by Theorem \ref{thm:main1} and the Weyl alternation sets $\A(\lambda,\mu)$ given by Theorem \ref{thm:195 alt sets}.

\begin{reptheorem}{thm:qKWMF}
Let $\lam=m\w_1+n\w_2+k\w_3$ and
 $\mu=c_1\w_1+c_2\w_2+c_3\w_3$ with $m,n,k,c_1,c_2,c_3\in\N$. If
$ x=\frac{3m+2n+k-3c_1-2c_2-c_3}{4}$,
$y=\frac{m+2n+k-c_1-2c_2-c_3}{2}$, and
$z=\frac{m+2n+3k-c_1-2c_2-3c_3}{4}$,
then

{\footnotesize{
\[
m_q(\lambda,\mu)=\begin{cases}
Z_1-Z_{11}-Z_3+Z_5+Z_{10}-Z_7 & P_1,Q_1,R_1,Q_6,R_4,Q_5,R_3 \in \N \ \text{and}\  P_4, P_3 \notin \N \\
Z_1-Z_6-Z_3+Z_4+Z_8-Z_2 & P_1,Q_1,R_1,P_4,Q_6,P_3,Q_4 \in \N \ \text{and} \ R_4, R_3 \notin \N\\
Z_1-Z_6-Z_{11}-Z_3+Z_8+Z_9 & P_1,Q_1,R_1,P_4,R_4,Q_6,Q_4 \in \N \ \text{and} \ P_3, Q_5, R_3 \notin \N\\
Z_1-Z_6-Z_{11}-Z_3+Z_9+Z_5 & P_1,Q_1,R_1,P_4,Q_6,R_4, Q_5 \in \N \ \text{and} \ P_3,Q_4,R_3 \notin \N\\
Z_1-Z_6-Z_{11}-Z_3+Z_9 & P_1,Q_1,R_1,P_4,Q_6,R_4 \in \N\ \text{and} \ P_3,Q_4, Q_5, R_3 \notin \N\\
Z_1-Z_{11}-Z_3+Z_5 & P_1,Q_1,R_1,Q_6,R_4,Q_5 \in \N \ \text{and} \ P_4,P_3,R_3 \notin \N \\
Z_1-Z_6-Z_3+Z_9 & P_1,Q_1,R_1,P_4,R_4 \in \N \ \text{and} \ Q_6,Q_4,Q_5 \notin \N\\
Z_1-Z_6-Z_{11}+Z_8 & P_1,Q_1,R_1,P_4,Q_6,Q_4 \in \N \ \text{and} \ R_4,P_3,R_3 \notin \N\\
Z_1-Z_{11}-Z_3 & P_1,Q_1,R_1,Q_6,R_4 \in \N \ \text{and} \ P_4,P_3, Q_5, R_3 \notin \N \\
Z_1-Z_6-Z_{11} & P_1,Q_1,R_1,P_4,Q_6 \in \N \ \text{and} \ R_4,P_3,Q_4,R_3 \notin \N\\
Z_1-Z_3 & P_1,Q_1,R_1,R_4 \in \N \ \text{and} \ Q_6,P_4,Q_5 \notin \N \ \text{and} \ (P_3 \notin \N \ \text{or} \ Q_4 \notin \N )\\
Z_1-Z_{11} & P_1,Q_1,R_1,Q_6 \in \N \ \text{and} \ P_4,R_4,R_3,P_3 \notin \N\\
Z_1 - Z_6 & P_1,Q_1,R_1, P_4 \in \N \ \text{and} \ Q_6,R_4,Q_4 \notin \N \ \text{and} \ (Q_5 \notin \N \ \text{or} \ R_3 \notin \N)\\
Z_1 & P_1,Q_1, R_1 \in \N \ \text{and} \ R_4,P_4,Q_6 \notin \N \ \text{and} \ (P_3 \notin \N \ \text{or} \ Q_4 \notin \N) \\& \text{and} \ (Q_5 \notin \N \ \text{or} \ R_3 \notin \N )\\
0 & \textnormal{otherwise,}
\end{cases}
\]
}}
where 
\begin{align*}
Z_{1}=&\wp_q(1(\lam+\rho)-(\rho+\mu))=\wp_q(P_1\al_1 + Q_1\al_2 + R_1\al_3)\\
Z_{2}=&\wp_q(s_1s_2s_1 (\lam+\rho)-(\rho+\mu))=\wp_q(P_3\al_1 + Q_4\al_2 + R_1\al_3)\\
Z_{3}=&\wp_q(s_3(\lam+\rho)-(\rho+\mu))=\wp_q(P_1\al_1 + Q_1\al_2 + R_4\al_3)\\
Z_{4}=&\wp_q(s_1s_2(\lam+\rho)-(\rho+\mu))=\wp_q(P_3\al_1 + Q_6\al_2 + R_1\al_3)\\
Z_{5}=&\wp_q(s_2s_3(\lam+\rho)-(\rho+\mu))=\wp_q(P_1\al_1 + Q_5\al_2 + R_4\al_3)\\
Z_{6}=&\wp_q(s_1(\lam+\rho)-(\rho+\mu))=\wp_q(P_4\al_1 + Q_1\al_2 + R_1\al_3)\\
Z_{7}=&\wp_q(s_2s_3s_2(\lam+\rho)-(\rho+\mu))=\wp_q(P_1\al_1 + Q_5\al_2 + R_3\al_3)\\
Z_{8}=&\wp_q(s_2s_1(\lam+\rho)-(\rho+\mu))=\wp_q(P_4\al_1 + Q_4\al_2 + R_1\al_3)\\
Z_{9}=&\wp_q(s_3s_1 (\lam+\rho)-(\rho+\mu))=\wp_q(P_4\al_1 + Q_1\al_2 + R_4\al_3)\\
Z_{10}=&\wp_q(s_3s_2(\lam+\rho)-(\rho+\mu))=\wp_q(P_1\al_1 + Q_6\al_2 + R_3\al_3)\\
Z_{11}=&\wp_q(s_2(\lam+\rho)-(\rho+\mu))=\wp_q(P_1\al_1 + Q_6\al_2 + R_1\al_3)
\end{align*}
and 
$P_1=x$, $P_3=-c_1-c_2-y+z-2$, $P_4=-c_1-x+y-1$, $Q_1=y$, $Q_4=-c_1-c_2-x+z-2$,
$Q_5=-c_2-c_3+x-z-2$,
$Q_6=-c_2+x-y+z-1$,
$R_1=z$,
$R_3=-c_2-c_3+x-y-2$, and 
$R_4=-c_3+y-z-1$.
\end{reptheorem}

For each $1\leq i\leq 11$, the formula for $Z_i$ given in Theorem \ref{thm:qKWMF} is obtained by an application of Theorem \ref{thm:main1}, see Appendix \ref{appendix:ztable} for the details of this. By evaluating $q=1$ in Theorem~\ref{thm:qKWMF} we give a formula for the multiplicity of the weight $\mu$ in the highest weight representation $L(\lambda)$ of~$\sl4$, with highest dominant integral weight $\lambda$.

\subsection*{Outline of the Paper} The remainder of the manuscript is organized as follows. Section \ref{sec:background} provides all of the necessary background on the Lie algebra $\sl4$ in order to make our approach precise. Section \ref{sec:qpartition} establishes Theorem \ref{thm:main1} 
using a combinatorial approach related to integer paritions. Section \ref{sec:Weyl alt sets and diagrams} presents all of the results associated to the Weyl alternation sets and Weyl alternation diagrams. Section \ref{sec:qmult} establishes Theorem \ref{thm:qKWMF}. We end with Section \ref{sec:future} where we provide a few directions for future research. In particular, we ask about further connections between Kostant's partition function and restricted colored integer partitions.

\section{Background}\label{sec:background}
In this section we provide definitions needed to make our approach precise throughout the following sections. We follow the notation used in \cite{GW}.

The Lie algebra $\sl4$ consists of traceless $4\times 4$ complex-valued matrices. We will consider the Cartan subalgebra, $\mathfrak{h}$ of $\sl4$, consisting of the diagonal matrices of $\sl4$. For $i=1,2,3$ let $\a_i=e_i-e_{i+1}$, where $e_1,e_2,e_3,e_4$ denote the standard basis vectors of $\mathbb{R}^4$. Then the set $\Delta=\{\a_1,\a_2,\a_3\}$ is a set of simple roots and $\Phi^+=\{\a_1,\a_2,\a_3,\a_1+\a_2,\a_2+\a_3,\a_1+\a_2+\a_3\}$ is the corresponding set of positive roots of $\sl4$.  Figure \ref{fig:rootsystem} gives a geometry-preserving projection of the root system of $\mathfrak{sl}_4(\mathbb{C})$ into $\mathbb{R}^3$. The fundamental weights of $\sl4$ are
\begin{align}\label{eq:fundamental weights}
    \w_1 &= \frac{1}{4}(3\al_1+2\al_2+\al_3),\qquad
    \w_2 = \frac{1}{2}(\al_1+2\al_2+\al_3),\qquad
    \w_3 = \frac{1}{4}(\al_1+2\al_2+3\al_3).
\end{align}
Hence, $\rho=\frac{1}{2}\sum_{\alpha\in\Phi^+}\alpha=\w_1+\w_2+\w_3=\frac{1}{2}(3\al_1 + 4\al_2 + 3\al_3)$.
\begin{figure}[h]
\centering
\includegraphics[width=1.5in]{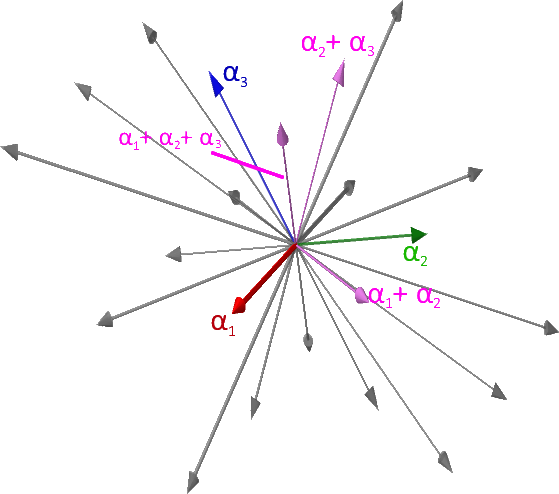}
\caption{Root system of $\mathfrak{sl}_4(\C)$.}
\label{fig:rootsystem}
\end{figure}

The Weyl group of $\sl4$ is isomorphic to the symmetric group $\mathfrak{S}_4$ and its generators are $s_1$, $s_2$, and $s_3$ which act as reflections across hyperplanes orthogonal to the simple roots $\a_1,\a_2,\a_3$, respectively. In particular for $1\leq i\leq 3$ we have that 
\begin{align*}
    s_i(\a_j)=\begin{cases}
    -\a_j&\mbox{if $i=j$}\\
    \a_i+\a_{j}&\mbox{if $|i-j|=1$}\\
    \a_j&\mbox{if $|i-j|>1$}
    \end{cases}
\qquad \mbox{and}\qquad
    s_i(\w_j)=\begin{cases}
    \w_j&\mbox{if $i\neq j$}\\
    \w_j-\a_{j}&\mbox{if $i=j$}.
    \end{cases}
\end{align*}

We conclude this section with the following result which will be important when computing Weyl alternation sets in Section \ref{sec:Weyl alt sets and diagrams}.
\begin{lemma}\label{lem:integral}
Let $\lambda=m\w_1+n\w_2+k\w_3$, with $m,n,k\in \mathbb{Z}$. Then $\lambda=c_1{\a_1}+c_2\a_2+c_3\a_3$, for some $c_1,c_2,c_3\in\Z$ if and only if $m=2x-y, n=2y-z-x$ and $k=2z-y$ where $x, y, z \in \Z.$
\end{lemma}
\begin{proof}
First, we prove the forward direction. Let $\lambda=m\w_1+n\w_2+k\w_3$, with $m,n,k\in \mathbb{Z}.$ 
By \eqref{eq:fundamental weights} we write $\lambda$ as
\begin{align*}
     \lambda 
    &= \frac{m}{4}(3\al_1+2\al_2+\al_3) + \frac{n}{2}(\al_1+2\al_2+\al_3) + \frac{k}{4}(\al_1+2\al_2+3\al_3)\\
    &= \Big(\frac{3m+2n+k}{4}\Big)\al_1+\Big(\frac{2m+4n+2k}{4}\Big)\al_2+ \Big(\frac{m+2n+3k}{4}\Big)\al_3.
\end{align*}
In order for $\lambda=c_1{\a_1}+c_2\a_2+c_3\a_3$ for some $c_1,c_2,c_3\in\Z$, it must be that
\begin{align}
    3m+2n+k &= 4x\label{eq:f1}\\
    2m+4n+2k &= 4y\label{eq:f2}\\
    m+2n+3k &= 4z\label{eq:f3}
\end{align}
for some $x, y, z \in \Z.$ Then, from \eqref{eq:f1}, we know $k=4x-3m-2n$, and so substituting into \eqref{eq:f2} and simplifying yields $m = 2x-y$.
Similarly, solving for $m$ in \eqref{eq:f3}  shows $m=4z-2n-3k$. Substituting this into equation \eqref{eq:f2} and simplifying yields $k=2z-y$. 
Lastly, substituting $m=2x-y$ and $k=2z-y$ into \eqref{eq:f1} establishes  that $n=2y-x-z$, as desired.

To prove the backwards direction, let $\lambda=m\omega_1+n\omega_2+k\omega_3$ and assume
\begin{align}
    m &=2x-y\label{eq:g1}\\
    n &=2y-z-x\label{eq:g2}\\
    k &=2z-y\label{eq:g3}
\end{align}
where $x, y, z \in \Z$,
By \eqref{eq:fundamental weights} we can write 
$\lambda$ in terms of the simple roots as
\begin{align}
    \lambda &= \Big(\frac{3m+2n+k}{4}\Big)\al_1+\Big(\frac{2m+4n+2k}{4}\Big)\al_2+ \Big(\frac{m+2n+3k}{4}\Big)\al_3.\label{eq:h1}
\end{align}
Substituting \eqref{eq:g1}, \eqref{eq:g2}, and \eqref{eq:g3} into \eqref{eq:h1}, gives $\lambda=x\al_1+y\al_2+z\al_3$, which completes the proof.
\end{proof}

\section{Formulas for the \texorpdfstring{$q$}{q}-analog of Kostant's partition function}\label{sec:qpartition}
In this section we establish Theorem \ref{thm:main1}, which
generalizes the formulas of De Loera and Sturmfels for the Kostant's partition function of $\mathfrak{sl}_4(\mathbb{C})$ to its $q$-analog \cite{deloera}.
Our general strategy is to reduce the problem of counting partitions of a weight with a fixed number of parts to the problem of counting restricted partitions of an integer with parts drawn from a multiset of natural numbers. We begin with the following definition.

\begin{definition}\label{def:colored partitions}
    Let $t\in\N$, and let $\mathcal{M}$ be a multiset with elements from $\N$. Then an $\mathcal{M}$-partition of $t$ is a sequence of natural numbers $(a_m)_{m\in\mathcal{M}}$ such that $ \sum_{m\in\mathcal{M}} m \cdot a_m=t $.
\end{definition}

When working with elements of a multiset we associate a ``color'' to elements with multiplicity greater than one so that we may distinguish them. For more on colored partitions, restricted colored partitions, and their properties, see \cite{CL,CLY, Keith,Kim}. The following example illustrates Definition \ref{def:colored partitions} in the case when $\mathcal{M}=\{1,1,2\}$.

\begin{example}For ease of distinguishing between the multiple $1$'s we let $\mathcal{M}$ be the multiset $\mathcal{M}=\{\one,\onep, 2\}$. 
    Then the distinct $\mathcal{M}$-partitions of $4$ are $\one + \one + 2$, $\one + \onep + 2$, $\onep + \onep + 2$, $\one + \one + \one + \one$, $\onep + \one + \one + \one$, $\onep + \onep + \one + \one$, $\onep + \onep + \onep + \one$, $\onep + \onep + \onep + \onep$, and $2 + 2$.
\end{example}

Next, we establish a bijection between vector partitions with a fixed number of parts and restricted colored integer partitions, a result which will be central to the remainder of this paper. 

\begin{proposition}
    Let $m,n,k,i\in\N$, and let $\xi=m\al_1+n\al_2+k\al_3$. Then the number of ways to express $\xi$ as a sum of exactly $i$ positive roots is equal to the number of $\{\one,\onep,2\}$-partitions $\mpartn def$ of $t:=m+n+k-i$ that satisfy
    \begin{enumerate}[label=(\roman*),ref=\roman*]
        \item \label{prop5.1} $d+f\leq m$,
        \item \label{prop5.2} $d+e+f\leq n$, and
        \item \label{prop5.3}$e+f\leq k$.
    \end{enumerate}
    \label{bijection}
\end{proposition}
\begin{proof}
    Let $X$ denote the set of partitions of $\xi$ using positive roots, where $(a,b,c,d,e,f)$ corresponds to the partition $\xi=a\al_1 + b\al_2 + c\al_3 + d(\al_1+\al_2) + e(\al_2+\al_3) + f(\al_1+\al_2+\al_3)$. Let $Y$ denote the set of $\{\one,\onep,2\}$-partitions of $t$ satisfying \eqref{prop5.1}-\eqref{prop5.3}, where the tuple $(d,e,f)$ corresponds to the partition $t=d\cdot\one + e\cdot\onep + f\cdot2$. 
    We define a function $F:X\to Y$ which is given by $(a,b,c,d,e,f) \mapsto (d,e,f)$.
    
    To prove that $F$ is well-defined, let $p=(a,b,c,d,e,f)$ be a partition of $\xi$. Then the image of this partition under $F$ is the partition $\mpartn def$. First, note that we can expand and combine terms to find that
        $\partn abcdef \\=\ (a+d+f)\al_1+(b+d+e+f)\al_2+(c+e+f)\al_3$.
    
    Comparing the coefficients of $\xi=m\al_1+n\al_2+k\al_3$ with the coefficients above, it follows that
    \[a = m-d-f,\qquad
        b = n-d-e-f, \qquad
        c = k-e-f,\]
    and so our partition is actually $(m-d-f,n-d-e-f,k-e-f,d,e,f)$. The number of parts in $p$ is $a+b+c+d+e+f=m+n+k-d-e-2f$. 
    Hence, as defined above, we have that $t=d+e+2f$, and so $F(p)=(d,e,f)$ is indeed a $\{\one,\onep,2\}$-partition of $t$.
    
    Since $p\in X$, we have by definition that each of its components are nonnegative integers. Thus, $a, b, c \geq 0$ imply
    \[
        d+f\leq m, \qquad
        d+e+f\leq n \\\qquad
        e+f\leq k.
    \]
    Thus, $(d,e,f)$ satisfies \eqref{prop5.1}-\eqref{prop5.3} and is also a $\{\one, \onep, 2\}$ partition of $t$ and so is in $Y$.
    
    Next to show that $F$ is a bijection, consider $G:Y\to X$ given by \[(d,e,f) \mapsto (m-d-f,n-d-e-f,k-e-f,d,e,f).\] A similar argument to the one above will show that $G$ is well defined and it is straightforward to check that $G$ is the inverse of $F$. It follows that $F$ is a bijection, and so $|X|=|Y|$, as desired.
\end{proof}

In the following subsections, we 
give closed formulas for $\wp_q(m\al_1+n\al_2+k\al_3)$ in terms of $m$, $n$, and $k$ based on the following conditions:
\begin{enumerate}
    \item\label{cond:mkn} $m,k\geq n$,
    \item\label{cond:mnk} $m\geq n\geq k$,
    \item\label{cond:knm} $k\geq n\geq m$,
    \item\label{cond:nmk} $n \geq m\geq k$, and
    \item\label{cond:nkm} $n\geq k\geq m$.
\end{enumerate}
As is standard, whenever a lower index of a sum is larger than the upper index, the sum is zero.

\indent In Subsections \ref{subsec:case1}-\ref{subsec:case3}, we apply Proposition \ref{bijection} to determine closed formulas for $\wp_q$ in cases \eqref{cond:mkn}, \eqref{cond:mnk}, and \eqref{cond:nmk}.
In each case, we determine the minimum number of parts that a partition of $\xi=m\al_1+n\al_2+k\al_3$ may have, giving us the lowest and highest powers of $q$ in $\wp_q(\xi)$.
In general, note that the partition of the vector $m\al_1+n\al_2+k\al_3$ with the maximum number of parts is the partition $(m,n,k,0,0,0)$, which has $m+n+k$ parts.\\
\indent In Subsection \ref{subsec:mirroredcases}, we then exploit a symmetry in the root system for $\sl4$ to find closed formulas for cases \eqref{cond:knm} and \eqref{cond:nkm} using the closed formulas from cases \eqref{cond:mnk} and \eqref{cond:nmk}, respectively.

\subsection{Proof of Theorem \texorpdfstring{\ref{thm:main1}}{1}, Part \texorpdfstring{\eqref{thm:main1:part1}}{1}}
\label{subsec:case1}
\begin{lemma}  \label{ncase1}
    Let $m,n,k,t\in\N$ be such that $m,k\geq n$. The number of $\{\one,\onep,2\}$-partitions of $t$ of the form $d\cdot \one +e\cdot \onep +f\cdot 2$ with $d,e,f\in\N$ that satisfy 
    \begin{enumerate}[label=(\roman*),ref=\roman*]
        \item\label{lemma1:ineq1} $d+f \leq m$,
        \item\label{lemma1:ineq2} $d+e+f\leq n$, and 
        \item\label{lemma1:ineq3} $e+f \leq k$
    \end{enumerate}
    is equal to
    \begin{equation}\label{lemma1:sum}
        \sum_{f=F_1}^{F_2} \sum_{d=D_1}^{D_2} 1,
    \end{equation} where $D_1 = 0, D_2 = t - 2f, F_1 = \max \{0, t - n\},$ and $F_2 = \left \lfloor \frac{t}{2} \right \rfloor.$
\end{lemma}

\begin{proof}
We wish to show that if $d,e,f\in\N$ are such that $d+e+2f=t$, then 
\eqref{lemma1:ineq1}-\eqref{lemma1:ineq3} hold if, and only if $D_1 \leq d \leq D_2$ and $F_1 \leq f \leq F_2$.

We begin by proving the converse direction. 
Assume $D_1 \leq d \leq D_2$ and $F_1 \leq f \leq F_2$. 

For the inequality in \eqref{lemma1:ineq2}, note that since $f\geq F_1\geq t-n$, we get $t-n \leq f$. Then, by substituting $t=d+e+2f$, we get the desired inequality, $d+e+f\leq n$. Using this result, the fact that $m,k\geq n$, and $e,d\geq 0$ the other two inequalities follow. To obtain the inequality in \eqref{lemma1:ineq1}, we have $d \leq n - e - f \leq n - f$, so $d+f \leq n \leq m$. To obtain the inequality in \eqref{lemma1:ineq3}, observe that $d+e+f\leq n$ implies $e \leq n - d -f \leq n - f$. Thus, $e+f \leq n \leq k$.

We now prove the forward direction. Assume 
\eqref{lemma1:ineq1}-\eqref{lemma1:ineq3} hold.

We first show that $D_1 \leq d \leq D_2$. Observe that $d\in \N$ and so $D_{1} = 0\leq d.$ Now, from $t=d+e+2f,$ we have that $d = t - e - 2f \leq t - 2f = D_{2}$. Therefore, $D_1 \leq d \leq D_2$.

We next show that $F_1 \leq f \leq F_2$.  Since $d+e+f \leq n$ and $t=d+e+2f$, we have $t-f = d+e+2f-f = d+e+f \leq n$, which implies that $t - n \leq f$. Furthermore, $f \in \N$ means that $0 \leq f$. Therefore, $F_1 = \max \{ 0, t - n \} \leq f$. Now, by definition, we have that $2f = t-d-e \leq t$. Since $f \in \N$, we have that $f \leq \floor{\frac{t}{2}} = F_2$. Hence, $F_1 \leq f \leq F_2.$

We have proved that if $d,e,f\in\N$ such that $d+e+2f=t,$ then \eqref{lemma1:ineq1}-\eqref{lemma1:ineq3} hold if, and only if $D_1 \leq d \leq D_2$ and $F_1 \leq f \leq F_2$. Since $e$ is determined by $d$ and $f,$ it then follows that the number of $\{\one,\onep,2\}$-partitions of 
$t$ satisfying \eqref{lemma1:ineq1}-\eqref{lemma1:ineq3} is the sum in  \eqref{lemma1:sum}. 
\end{proof}
We now establish part \eqref{thm:main1:part1} of Theorem \ref{thm:main1}, which we restate below for ease of reference.
\begin{proposition}\label{prop:3}
    If $m,n,k\in\N$
    satisfy $m,k\geq n$, then
    \[
        \wp_q(m\al_1+n\al_2+k\al_3) = \sum_{i=m+k-n}^{m+n+k} (L+1)(L+(t\bmod2)+1)q^i,
    \]
    where $L= \min\left\{\floor{\frac{t}{2}},n-\ceil{\frac{t}{2}}\right\}$
    and $t=m+n+k-i$.
    \label{case1}
\end{proposition}
\begin{proof}
    \indent First, note that the maximum number of parts among all partitions of $\xi$ is $m+n+k$.
    To find the minimum number of parts among all partitions of $\xi$, we note that in a partition of $\xi$, we can use $\al_1+\al_2+\al_3$ at most $n$ times since $m,k\geq n$.
    Afterwards, we cannot use $\al_1+\al_2$ nor $\al_2+\al_3$ since $n(\al_1+\al_2+\al_3)$ already contains the term $n\al_2$.
    Thus, we obtain a partition
    \begin{equation*}
        \xi = (m-n)\al_1 + (k-n)\al_3 + n(\al_1+\al_2+\al_3),
    \end{equation*}
    with the minimum number of parts: $(m-n)+(k-n)+n=m+k-n$.
    
    By Proposition \ref{bijection}, we know that for $i\in\N$, the number of vector partitions of
    $\xi$ with $i$ parts is equal to the number of $\{\one,\onep,2\}$-partitions
    of $t=m+n+k-i$ that satisfy \eqref{lemma1:ineq1}-\eqref{lemma1:ineq3}.
    Lemma \ref{ncase1} then tells us that the number of such partitions of $t$ is
    \begin{equation*}
        \sum_{f=\max\{0,t-n\}}^{\floor{\frac{t}{2}}}
        \sum_{d=0}^{t-2f} 1.
    \end{equation*}
    Thus, the number of partitions of $\xi$ using exactly $i$ positive roots is
    \begin{align*}
        \sum_{f=\max\{0,t-n\}}^{\floor{\frac{t}{2}}}
        \sum_{d=0}^{t-2f} 1
        &= \sum_{f=\max\{0,t-n\}}^{\floor{\frac{t}{2}}} (t-2f+1).
    \end{align*}
    Reindexing using $j=\floor{\frac{t}{2}}-f$ and letting 
    $L=\floor{\frac{t}{2}} - \max\{0,t-n\}$, this sum becomes
    \begin{align*}
        \sum_{j=0}^{L} \left[t-2\left(\floor{\frac{t}{2}} - j \right)+1\right] 
        &= \sum_{j=0}^L (2j + (t\bmod2) + 1) 
        = (L+1)(L+(t\bmod2)+1).
    \end{align*}
    Finally, we simplify $L$: 
    \begin{align*}
        L &= \floor{\frac{t}{2}} - \max\{0,t-n\}= \min\left\{\floor{\frac{t}{2}}, n-\left(t -\floor{\frac{t}{2}}\right)\right\}=\min\left\{\floor{\frac{t}{2}},n - \ceil{\frac{t}{2}}\right\}.\qedhere
    \end{align*}
\end{proof}

This completes the proof part \eqref{thm:main1:part1} Theorem \ref{thm:main1}, where $m,k\geq n$.
We now move on to part \eqref{thm:main1:part2} of Theorem \ref{thm:main1}, where $m\geq n\geq k$.
\subsection{Proof of Theorem \texorpdfstring{\ref{thm:main1}}{1}, Part \texorpdfstring{\eqref{thm:main1:part2}}{2}}
\label{subsec:case2}

\begin{lemma}
    Let $m,n,k,t\in\N$ be such that $m\geq n\geq k$.
    Then the number of $\{\one,\onep,2\}$-partitions of $t$ of the form $d\cdot \one + e\cdot\onep + f\cdot2$ with $d,e,f\in\N$ that satisfy
    the conditions
    \begin{enumerate}[label=(\roman*),ref=\roman*]
        \item\label{lemma2:ineq1} $d+f\leq m$,
        \item\label{lemma2:ineq2} $d+e+f\leq n$, and  
        \item\label{lemma2:ineq3} $e+f\leq k$
    \end{enumerate}
    is equal to
    \begin{equation}\label{lemma2:sum}
        \sum_{f=F_1}^{F_2}\; \sum_{e=E_1}^{E_2} 1,
    \end{equation}where $ E_1 = 0$, $E_2 = \min\{k-f,t-2f\}$,  $F_1 = \max \{  0, t-n\}$, and $F_2 =\min\left\{\floor{\frac{t}{2}},k\right\}$.
    \label{ncase2}
\end{lemma}

\begin{proof}
    We wish to show that if $t=d+e+2f$, then
    $E_1\leq e\leq E_2$ and $F_1\leq f\leq F_2$ if, and only if $d$, $e$, and $f$ satisfy \eqref{lemma2:ineq1}, \eqref{lemma2:ineq2}, and  \eqref{lemma2:ineq3}.

  In the forward direction, we assume $E_1\leq e\leq E_2$ and $F_1\leq f\leq F_2$. Note that by definition, $e+f \leq E_2+f \leq (k-f)+f = k$, and so \eqref{lemma2:ineq3} is satisfied. Next, observe that since $d=t-e-2f$, we have $d+e+f \leq (t-e-2f) + e + f = t-f \leq t-F_1 \leq t - (t - n) = n $, so \eqref{lemma2:ineq2} is satisfied. Finally, with \eqref{lemma2:ineq2}, we see that $d+f \leq d+e+f \leq n \leq m$. Hence, we get the inequality in \eqref{lemma2:ineq1}.
    
   Conversely, assume the inequalities in \eqref{lemma2:ineq1}, \eqref{lemma2:ineq2}, and \eqref{lemma2:ineq3} hold. Then $f\geq 0$, since $f\in \N$. Furthermore, $t-n = d+e+2f - n = d+e+f-n+f \leq n-n+f = f$. Therefore, $F_1 \leq f$. 
   
   Next, we note that since $t=d+e+2f$, then $f = \frac{t-d-e}{2} \leq \frac{t}{2}$, from which it follows that $f \leq \floor{\frac{t}{2}}$. Furthermore, we observe from \eqref{lemma2:ineq3} that $e+f \leq k$, yielding $f \leq k-e \leq k$. Hence, $f \leq F_2$. Altogether, this shows that $F_1 \leq f \leq F_2$.
   
   Continuing on, we observe that $e \in \N$ implies $E_1 = 0 \leq e.$ By assuming \eqref{lemma2:ineq3}, we note that this implies that $e \leq k - f $. Additionally, since $t = d+e+2f$, we have that $e = t - d - 2f \leq t - 2f$, which gives $e \leq E_2$. Altogether, this shows $E_1 \leq e \leq E_2$.
   
To summarize, we have proved that $\mpartn def$ is a partition of $t$ satisfying  \eqref{lemma2:ineq1}, \eqref{lemma2:ineq2} and  \eqref{lemma2:ineq3} if, and only if
$F_1\leq f\leq F_2$ and $D_1\leq d\leq D_2$.
Since $d$ is determined by $e$ and $f$, it then follows that
the number of such partitions of $t$ is the sum in \eqref{lemma2:sum}.
\end{proof}
We now establish part \eqref{thm:main1:part2} of Theorem \ref{thm:main1}, which we restate below for ease of reference.
\begin{proposition}
    If $m,n,k\in\N$ satisfy $m\geq n\geq k$, then
    $    \wp_q(m\al_1 + n \al_2 + k \al_3)=\sum_{i=m}^{m+n+k} c_iq^i,
    $
    where
    \begin{equation*}\resizebox{.95\hsize}{!}{$
        c_i = 
        \begin{cases}
        \dfrac{(L+1)(2k-2F_2+L+2)}{2} & \mbox{if }J < 0 \\\\
        \dfrac{(L+1)(L+2) - F_2(F_2+1) - k(k-1)}{2}
        +F_2 k+kL- F_2 L+(k-F_2)(t\bmod2)
        & \mbox{if }0\leq J\leq L \\\\
        (L+1)(L+(t\bmod2)+1) & \mbox{if }J > L
        \end{cases}$}
    \end{equation*}
    with $t=m+n+k-i$, 
    $F_2=\min\left\{\floor{\frac{t}{2}},k\right\}$,
    $J=k-F_2-(t\bmod2)$,
    $F_1 =\max\{t-n,0\}$, and 
    $L = F_2 - F_1$.
    \label{case2}
\end{proposition}
\begin{proof}
    \indent First, we note that the maximum number of parts in a partition of $\xi$ is $m+n+k$. To find the minimum number of parts among the partitions of $\xi$, first observe that we can use $(\a_1+\a_2+\a_3)$ a maximum of $k$ times since $k\leq m,n$. Assuming our partition has $k$ $(\al_1+\al_2+\al_3)$'s, we can use a maximum of $n-k$ $(\al_1+\al_2)$'s, from which point we complete the partition with $(m-n)$ $\al_1$'s. Thus, we obtain a partition
    \begin{equation*}
        \xi = (m-n)\al_1 + (n-k)(\al_1+\al_2) + k(\al_1+\al_2+\al_3),
    \end{equation*}
    with the minimum number of parts: $(m-n)+(n-k)+k = m$.
    
    By Proposition \ref{bijection} we have that $c_i$ equals the number of $\{\one,\onep,2\}$-partitions of $t$ satisfying
    \begin{enumerate}[label=(\roman*),ref=\roman*]
        \item $d+f \leq m$,
        \item $d+e+f \leq n$, and
        \item $e+f \leq k$.
    \end{enumerate}
    Lemma \ref{ncase2} then tells us that the number of such partitions is 
    \begin{align*}
        \sum_{f=F_1}^{F_2}
        \sum_{e=0}^{\min\{k-f,t-2f\}} 1 
        &= \sum_{f=F_1}^{F_2}
        [\min\{k-f,t-2f\}+1].
    \end{align*}
    
    Reindexing by $j=F_2-f$, our sum then becomes
    \begin{equation*}
        \sum_{j=0}^L [\min\{k-(F_2-j), t-2(F_2-j)\}+1],
    \end{equation*}
    which we claim is equal to 
    \begin{equation}\label{prop9:sum1}
        \sum_{j=0}^{L} [\min\{k-(F_2-j), 2j + (t\bmod2)\}+1],
    \end{equation}
    where $L=F_2-F_1$. 
    
    To justify this, recall that $F_2=\min\left\{\floor{\frac{t}{2}},k\right\}$.
    If $k\leq\floor{\frac{t}{2}}$, then $F_2=k$, in which case
    we have \[t-f \geq t - k \geq t - \floor{\frac{t}{2}} \geq \floor{\frac{t}{2}} \geq k.\] Thus, $t-2f\geq k-f$.
   
    When $F_2=k$, we also have $j = F_2 - f = k-f$, and so
    \begin{align*}
        \min\{k-F_2+j, 2j+t\bmod{2}\} &= \min\{k-k+k-f,2(k-f)+(t\bmod{2})\} \\
        &= \min\{k-f, 2(k-f)+(t\bmod{2})\}.
    \end{align*}
    Note here that $k-f \leq 2(k-f) + (t\bmod{2})$. Therefore, $\min\{k-f,2(k-f)+(t\bmod{2})\}=k-f$. Thus, when $k\leq \floor{\frac{t}{2}}$, we have \[\min\{k-f,t-2f\} = k-f = k-(F_2-j) = \min\{k-(F_2-j), 2j + (t\bmod2)\}.\]
    
    On the other hand, when $k\geq\floor{\frac{t}{2}}$,
    we have $F_2=\floor{\frac{t}{2}}$, and so \[  t-2f = t - 2(F_2-j) = t - 2F_2 + 2j = t - 2\floor{\frac{t}{2}} + 2j = 2j + (t\bmod 2). \] 
    Thus, in general, we have $\min\{k-f,t-2f\} = \min\{k-(F_2-j), 2j + (t\bmod2)\}$, and so our sum is indeed as in \eqref{prop9:sum1}.
    
    Observe that in \eqref{prop9:sum1}, we have a sum of a minimum of two expressions linear in our index $j$. Let $\jmid$ be the $j$ value for which these functions meet. To find $\jmid$, we equate the expressions $2\jmid+(t\bmod2)=k-F_2+\jmid$. Solving for $\jmid$, we find that $\jmid = k-F_2-(t\bmod2)$.
    
    Since $2j+(t\bmod2)$ increases more quickly with respect to $j$ than $k-F_2+j$ does, it follows that $2j+(t\bmod2)\leq k-F_2-(t\bmod2)$ when $j\leq\jmid$ and $2j+(t\bmod2)> k-F_2-(t\bmod2)$ when $j>J$. We then evaluate the sum in \eqref{prop9:sum1} in three different cases depending on $\jmid$.
    
    \smallskip
    
    \noindent\textit{Case 1 ($\jmid<0$):}
    If $\jmid<0$, then we have $j>\jmid$ for all $j\in\{0,1,\dots,L\}$,
    in which case $k-F_2+j < 2j+(t\bmod2)$.
    Thus, we get \[\sum_{j=0}^L \left[\min{(2j + (t\bmod 2), k-F_2+j)}+1\right] = \sum_{j=0}^L (k-F_2+j+1) = \dfrac{(L+1)(2k-2F_2+L+2)}{2},\] after rearranging and simplifying the terms in the second sum. 
    
    \smallskip
    
    \noindent \textit{Case 2 ($0\leq\jmid\leq L$):}
    If $0\leq\jmid\leq L$, then we must separate the sum \eqref{prop9:sum1} into two sums, one indexed by $j\leq\jmid$, and the other indexed by $j>\jmid$. Doing this and further simplifying, we obtain 
    \begin{align*}
    \sum_{j=0}^L [\min\{2j + (t\bmod 2), k -F_2+j\} +1]
    =&\sum_{j=0}^{\jmid}(2j + (t\bmod2)) + \sum_{\jmid+1}^L (k-F_2+j) + \sum_{j=0}^L 1&\\
    &\hspace{-2.7in}= (t\bmod2)(\jmid+1)
    + (k-F_2)(L-(\jmid+1)+1)
    + \frac{L(L+1)}{2} + \frac{\jmid(\jmid+1)}{2} + (L+1),
    \end{align*}
    where we obtain the last equality by rearranging and simplifying the second sum. Now, by substituting in the value for $\jmid = k-F_2- (t\bmod2)$, combining terms and using the fact that $(t \bmod 2)^2 = t \bmod 2$, we then obtain the closed formula 
    \[ \frac{(L+1)(L+2) - F_2(F_2+1) - k(k-1)}{2}+F_2k+kL-F_2L+(k-F_2)(t\bmod2). \]

\smallskip
\noindent \textit{Case 3 ($\jmid>L$):}
 Since $j\leq L<\jmid$, we have that $2j+(t\bmod2)<k-F_2+j$, and so
    \begin{align*}
        \sum_{j=0}^L \left[\min{(2j + (t\bmod 2), k-F_2+j)}+1\right]
        &= \sum_{j=0}^L [2j+(t\bmod2) +1] 
        = (L+1)(L+(t\bmod2)+1),
    \end{align*}
after rearranging and simplifying the second sum.

Therefore, in all three of these cases, we have the desired result.
\end{proof}

\subsection{Proof of Theorem \texorpdfstring{\ref{thm:main1}}{1}, Part \texorpdfstring{\eqref{thm:main1:part4}}{4}}
\label{subsec:case3}

\begin{lemma}
    Let $m,n,k,t \in\N$ be such that $n\geq m\geq k$.
    Then the number of $\{\one,\onep,2\}$-partitions 
    $d\cdot\one + e\cdot\onep + f\cdot2$ of $t$ that satisfy the conditions
    \begin{enumerate}[label=(\roman*),ref=\roman*]
        \item\label{lemma3:ineq1} $d+f\leq m$,
        \item\label{lemma3:ineq2} $d+e+f\leq n$, and
        \item\label{lemma3:ineq3} $e+f\leq k$
    \end{enumerate}
    is equal to 
    \begin{equation}\label{lemma3:sum}
       \sum_{f=F_1}^{F_2}\; \sum_{d=D_1}^{D_2} 1
    \end{equation}
    where $D_1 = \max\{0,t-k-f\}, D_2 = \min\{t-2f, m-f\},  F_1 = \max\{0,t-\min\{m+k,n\}\}, \mbox{ and } F_2 = \min\left\{\floor{\frac{t}{2}},k\right\}$.
    \label{ncase3}
\end{lemma}
\begin{proof}
    We wish to show that when $d+e+2f=t$,
    the conditions \eqref{lemma3:ineq1}, \eqref{lemma3:ineq2}, and \eqref{lemma3:ineq3} are satisfied
    if and only if $F_1\leq f \leq F_2$ and $D_1\leq d\leq D_2$. 
  
    First, let $m,n,k,t,d,e,f \in \N $ and assume $d+e+2f=t$, $F_1\leq f\leq F_2$, and $D_1\leq d\leq D_2$.\\ 
    For condition \eqref{lemma3:ineq3}, observe that $e=t-2f-d$ and $D_1 \leq d$. This implies \[ e + f \leq t - 2f - d + f \leq t - D_1 - f \leq k. \]
    For \eqref{lemma3:ineq2}, first assume $m+k\leq n$. Then, using the fact that $d \leq D_2 \leq m - f$ and that $e \leq k$ (from \eqref{lemma3:ineq3}), we have 
    \[ d+e+f \leq D_2 +e+f \leq m+e \leq m+k \leq n. \]
    On the other hand, if we assume $m+k\geq n$, then 
    \begin{align*}
        F_1 &= \max\{0,t-\min\{m+k,n\}\}= \max\{0,t-n\}.
    \end{align*}
    In this case, since $t = d+e+2f$ and $t-n \leq F_1 \leq f$, we have \[ d+e+f = t - f \leq t - (t-n) = n. \] Thus, in either case, \eqref{lemma3:ineq2} is satisfied.
    Now for condition \eqref{lemma3:ineq1}, note that because $d \leq D_2$ we have \[ d+f \leq D_2 +f \leq m. \]
    Therefore, if $F_1\leq f \leq F_2$ and $D_1\leq d \leq D_2,$ then conditions \eqref{lemma1:ineq1}, \eqref{lemma3:ineq2}, and \eqref{lemma3:ineq3} hold as desired.
    
    Conversely, assume the inequalities \eqref{lemma3:ineq1}, \eqref{lemma3:ineq2}, and \eqref{lemma3:ineq3} hold. We first show that $F_1\leq f \leq F_2$. Assume $d,e,f \in \N.$ Then, since $f \in \N$, we have $f\geq0$. In the case where $n\leq m+k$, we have $F_1=\max\{t-n,0\}$. Together with condition \eqref{lemma3:ineq2} and since $t = d + e + 2f$, we see that $$f \leq n-d-e = n - d - t + d + 2f = n- t + 2f.$$ Thus, $t-n\leq f$. Then $f \geq \max\{t-n,0\} = F_1$ and so  $F_1\leq f$.
    Next, we note that $t=d+e+2f$ implies $2f = t-d-e\leq t$. Thus, $f\leq \frac{t}{2}$, and so $f\leq \floor{\frac{t}{2}}$ since $f\in\N$. Observe also that condition \eqref{lemma3:ineq3} implies that $f\leq k-e\leq k$. Therefore, since $f\leq k$ and $f\leq \floor{\frac{t}{2}}$, we have that $f\leq \min\left\{\floor{\frac{t}{2}},k\right\}=F_2$. 
    Hence, we have shown $F_1 \leq f \leq F_2$, where $F_1=\max\{0,t-\min\{m+k,n\}\}$ and $F_2= \min\left\{\floor{\frac{t}{2}},k\right\}$.
    
    We now wish to show that $D_1\leq d\leq D_2.$ Since $d \in \N,$ we have $0\leq d$. Condition \eqref{lemma3:ineq3} together with $t = d + e + 2f$ implies $k \geq e + f = (t - 2f - d) + f$.  Therefore, $t-k-f \leq d$. Combining this, we see that $ d \geq 0$ and $ d \geq t-k-f$, whereby $d \geq \max\{0, t-k-f \}= D_1$. 
    Next, rewriting condition \eqref{lemma3:ineq1}, we have $d\leq m-f$. In addition, $d \leq d+e = (t - 2f - e) + e = t - 2f$, since $t = d+e+2f$. Therefore, $d \leq \min \{ t-2f, m-f \} = D_2$. Combining both results, we see that $D_1 \leq d \leq D_2$. We have proved that $F_1 \leq f \leq F_2$ and $D_1 \leq d \leq D_2$ as desired. 
    
    To summarize, we have shown that $d\cdot\one+e\cdot\onep+f\cdot2 = t$ is a partition of $t$ satisfying conditions \eqref{lemma3:ineq1}, \eqref{lemma3:ineq2}, and \eqref{lemma3:ineq3} if, and only if $F_1\leq f\leq F_2$ and $D_1\leq d\leq D_2$.
    Since $e$ is determined by $d$ and $f$, it then follows that the number of such partitions of $t$ is given in the sum in \eqref{lemma3:sum}.
\end{proof}

We now establish part \eqref{thm:main1:part4} of Theorem \ref{thm:main1}, which we restate below for ease of reference.

\begin{proposition}
   If $m,n,k\in\N$ satisfy
    $n\geq m\geq k$, then
        $\wp_q(m\al_1+n\al_2+k\al_3) = \sum_{i=n}^{m+n+k} c_iq^i$,
   where $c_i=S_1-S_2 + F_2 - F_1 + 1$,
   \begin{align*}
        S_1 &= 
      \begin{cases}
          (t-F_2)(F_2+1) &\mbox{if}\ t-m < 0 \\
          \frac{(t-m)(t-m+1)+F_1(F_1-1)- 2F_2(F_2+1) + 2m(t-m-F_1+1) + 2t(F_2-t+m)}{2}  &\mbox{if}\ 0\leq t-m\leq F_2,
        \end{cases} \\
          S_2 &= 
        \begin{cases}
         (t-k)(t-k-F_1+1) - \frac{(t-k)(t-k+1)-F_1(F_1-1)}{2}&\mbox{if}\ 0\leq t-k\leq F_2     \\
          (t-k)(F_2-F_1+1) - \frac{F_2(F_2+1)-F_1(F_1-1)}{2}
            &\mbox{if}\ t-k > F_2\\
            0 &\mbox{if}\ t-k < 0
      \end{cases}
         \end{align*}
         with $t = m+n+k-i, F_1 = \max\{0,t-\min\{m+k,n\}\}, \text{ and } F_2 = \min\left\{\floor{\frac{t}{2}},k\right\}$.
    \label{case3}
\end{proposition}

\begin{proof}
    \indent We begin by noting that the maximum number of parts in a partition of $\xi$ is $m+n+k$.
    
    To find the minimum number of parts among the partitions of $\xi$, first observe that we can use $(\a_1+\a_2+\a_3)$ a maximum of $k$ times, since $k\leq m,n$. We can then use $(\al_1+\al_2)$ a maximum of $m-k$  times. Finally, we complete a partition of $\xi$ with $(n-m)$ $\al_1$'s, obtaining a partition
    \begin{equation*}
        \xi = (n-m)\al_1 + (m-k)(\al_1+\al_2) + k(\al_1+\al_2+\al_3)
    \end{equation*}
    with the minimum number of parts: $(n-m)+(m-k)+k=n$.
    
    Proposition \ref{bijection} and Lemma \ref{ncase3} give us the following closed formula for $c_i$, where $t=m+n+k-i$ and
    \begin{align*}
        c_i &= \sum_{f=F_1}^{F_2}
        \sum_{d=\max\{0,t-k-f\}}^{\min\{t-2f,m-f\}} 1 
         &=F_2 - F_1 + 1+\sum_{f=F_1}^{F_2} \min\{t-2f,m-f\} - \sum_{f=F_1}^{F_2}\max\{0,t-k-f\}.
    \end{align*}
    
    Let 
    \begin{equation*}
        S_1 = \sum_{f=F_1}^{F_2} \min\{t-2f,m-f\},\quad
        S_2 = \sum_{f=F_1}^{F_2}\max\{0,t-k-f\}.
    \end{equation*}
    
    First, we consider $S_1$, starting with the case where $t-m<0$. Note that if $m+k\leq n$, then 
    $
        F_1 = \max\{0,t-\min\{m+k,n\}\} 
        = \max\{0,t-(m+k)\} 
        = 0,
    $
    since $t\leq m+k$.
    Furthermore, when $n\leq m+k$, we have
    $
        F_1 = \max\{0,t-\min\{m+k,n\}\}
        = \max\{0,t-n)\} 
        = 0,
    $
    since in this case $t\leq m\leq n$. Thus, $F_1=0$. Note that $t-2f < m-f$ if and only if $f > t-m$. It then follows that when $t-m<0$, we have
    \begin{align*}
        S_1=\sum_{f=0}^{F_2} (t-2f) 
        = t(F_2+1) - \frac{2F_2(F_2+1)}{2}
        =(t-F_2)(F_2+1).
    \end{align*}
   
    Next, when $0\leq t-m\leq F_2$, we have
    
    \begin{align*}
        {S_1=\sum_{f=F_1}^{t-m} (m-f) + \sum_{f=t-m+1}^{F_2} (t-2f)} 
         &= \frac{(t-m)(t-m+1)}{2} +m(t-m-F_1+1)+ \frac{F_1(F_1-1)}{2}\\ &\qquad\qquad\qquad +t(F_2-t+m) -F_2(F_2+1).
    \end{align*}
    
Finally, recall that $F_2=\min\{\floor{\frac{t}{2}},k\}$, and note that
    $t\leq m+k$, and so $t-m \leq k$. Additionally, note that
    $
        2m = m + m \geq m + k \geq t,
    $
    and so $m \geq \ceil{\frac{t}{2}}$, giving us $t - m \leq \floor{\frac{t}{2}}$. Therefore, we have $t-m\leq \min\left\{\floor{\frac{t}{2}},k\right\} = F_2$. Thus, we do not need to consider the case $t-m > F_2$. 
    
    Next we consider $S_2$. Note that $0 > t-k-f$ when $t-k < f$, and $t-k-f > 0$ when $t-k > f$. It then follows that when $t-k<0$, we have $S_2=\sum_{f=F_1}^{F_2} 0 =0$. In addition, when $0\leq t-k\leq F_2$, we have
    \begin{align*}
        S_2=\sum_{f=F_1}^{t-k} t- k - f  + \sum_{f=F_1}^{F_2} 0 
        = (t-k)(t-k-F_1+1) -\frac{(t-k)(t-k+1)}{2} +\frac{F_1(F_1-1)}{2}.
    \end{align*}
    Finally, when $t-k > F_2$, we have    
    \begin{align*}
        S_2=\sum_{f=F_1}^{F_2} (t-k-f) 
        &= (t-k)(F_2-F_1+1) - \dfrac{F_2(F_2+1) - F_1(F_1-1)}{2}.
    \end{align*}
    Thus the result holds.
\end{proof}

\subsection{Remaining Cases}
\label{subsec:mirroredcases}
In this section, we exploit the symmetry of the root system to give closed formulas for the $q$-analog of Kostant's partition function in the rest of
the cases provided in~\cite{deloera}.
\begin{proposition}\label{mirror}
    If $m,n,k\in\N$, then 
    $        \wp_q(m\al_1+n\al_2+k\al_3) =
        \wp_q(k\al_1+n\al_2+m\al_3).
    $
\end{proposition}
\begin{proof}
    Suppose we have a partition of $m\al_1+n\al_2+k\al_3$ with $i$ parts:
    \begin{equation*}
        a\al_1+b\al_2+c\al_3+d(\al_1+\al_2)+e(\al_2+\al_3)+f(\al_1+\al_2+\al_3)
        = m\al_1+n\al_2+k\al_3.
    \end{equation*}
   \indent We identify these partitions of $m\al_1+n\al_2+k\al_3$
    with tuples $(a,b,c,d,e,f)\in\N^6$ such that
    \begin{enumerate}
        \item $a+d+f=m$,
        \item $b+d+e+f=n$,
        \item $c+e+f=k$, and
        \item $a+b+c+d+e+f=i$.
    \end{enumerate}
    Call this set of $6$-tuples $\mathcal{P}_i(m,n,k)$.
    We then define a function 
    $f_{mnk}:\mathcal{P}_i(m,n,k)\to\mathcal{P}_i(k,n,m)$
    given by $(a,b,c,d,e,f) \mapsto (c,b,a,e,d,f).$ First we wish to show that $f$ is well-defined. Let $(a,b,c,d,e,f)\in\mathcal{P}_i(m,n,k)$. Then note that
    \begin{enumerate}
        \item $c+e+f=k$,
        \item $b+e+d+f=n$,
        \item $a+d+f=m$, and
        \item $a+b+c+d+e+f=i$.
    \end{enumerate}
    Thus, we have that $(c,b,a,e,d,f)\in\mathcal{P}_i(k,n,m)$, and so $f_{mnk}:\mathcal{P}_i(m,n,k)\to\mathcal{P}_i(k,n,m)$ is well-defined for all $m,n,k\in\N$.
    
    Note, then, that for all $m,n,k\in\N$, we have
    \begin{align*}
        f_{knm}(f_{mnk}(a,b,c,d,e,f)) &= f_{knm}(c,b,a,e,d,f) = (a,b,c,d,e,f).
    \end{align*}
    It follows that $f_{mnk}\circ f_{knm}$ is the identity as well, so $f_{knm}=f_{mnk}^{-1}$.
    Thus, $f_{mnk}$ is a bijection, and so 
    $|\mathcal{P}_i(m,n,k)| = |\mathcal{P}_i(k,n,m)|$.
    Since $m,n,k,$ and $i$ were arbitrary, and
    $
        \wp_q(m\al_1+n\al_2+k\al_3)=\sum_{i\in\N} |\mathcal{P}_i(m,n,k)|q^i,
    $
    then we have 
    $
        \wp_q(m\al_1+n\al_2+k\al_3) = 
        \wp_q(k\al_1+n\al_2+m\al_3)
    $,
    as desired.
\end{proof}

Proposition \ref{mirror} immediately gives us the closed formulas in Theorem \ref{thm:main1} parts \eqref{thm:main1:part3} and \eqref{thm:main1:part5}. We now state and prove these two results explicitly.

\begin{corollary}
If $m,n,k\in\N$ satisfy $k\geq n\geq m$, then
 $\wp_q(m\al_2+n\al_2+k\al_3)=\sum_{i=k}^{m+n+k} c_iq^i$,
    where
    \begin{equation*}\resizebox{0.95\hsize}{!}{$
        c_i = 
        \begin{cases}
        \dfrac{(L+1)(2m-2F_2+L+2)}{2} & \mbox{if } J < 0 \\
        \dfrac{(L+1)(L+2) - F_2(F_2+1) - m(m-1)}{2}
        +F_2m+mL-F_2L+(m-F_2)(t\bmod2)
        & \mbox{if } 0\leq J\leq L \\
        (L+1)(L+(t\bmod2)+1) & \mbox{if } J > L
        \end{cases}$}
    \end{equation*}
    with $t=m+n+k-i$, 
    $F_2=\min\left\{\floor{\frac{t}{2}},m\right\}$,
    $J=m-F_2-(t\bmod2),$
    $F_1=\max\{t-n,0\}$, and 
    $L = F_2 - F_1$.
    \label{coro:case4}
\end{corollary}
\begin{proof}
    Proposition $\ref{mirror}$ tells us that the number of partitions
    of $m\al_1+n\al_2+k\al_3$ with $i$ parts equals the number of
    partitions of $k\al_1+n\al_2+m\al_3$ with $i$ parts.
    Since $k\geq n\geq m$, Proposition $\ref{case2}$ then gives 
    us the closed formula for the number of such partitions, as shown 
    in the statement above.
\end{proof}
    
\begin{corollary}
   If $m,n,k\in\N$ satisfy
    $n\geq k\geq m$, then 
    $
        \wp_q(m\al_1+n\al_2+k\al_3) = \sum_{i=n}^{m+n+k} c_iq^i
    $,
   where $c_i=S_1-S_2 + F_2 - F_1 + 1$,
    \begin{align*}
        S_1 &= 
      \begin{cases}
          (t-F_2)(F_2+1) &\mbox{if}\ t-k < 0 \\
          \frac{(t-k)(t-k+1)+F_1(F_1-1)- 2F_2(F_2+1)
          + 2k(t-k-F_1+1) + 2t(F_2-t+k)}{2}  &\mbox{if}\ 0\leq t-k\leq F_2 ,
        \end{cases}
        \\  S_2 &= 
        \begin{cases}
            (t-m)(t-m-F_1+1) - \frac{(t-m)(t-m+1)-F_1(F_1-1)}{2}
            &\mbox{if}\ 0\leq t-m\leq F_2 \\(t-m)(F_2-F_1+1) - \frac{F_2(F_2+1)-F_1(F_1-1)}{2}
            &\mbox{if}\ t-m > F_2\\
            0 &\mbox{if}\ t-m < 0,
      \end{cases}
    \end{align*}
with $t = m+n+k-i$,
        $F_1 = \max\{0,t-\min\{m+k,n\}\}$, and
        $F_2 = \min\left\{\floor{\frac{t}{2}},m\right\}$.
    \label{coro:cor2}
\end{corollary}
\begin{proof}
    Proposition $\ref{mirror}$ tells us that the number of partitions
    of $m\al_1+n\al_2+k\al_3$ with $i$ parts equals the number of
    partitions of $k\al_1+n\al_2+m\al_3$ with $i$ parts.
    Since $n\geq k\geq m$, Proposition $\ref{case3}$ then gives 
    us the closed formula for the number of such partitions, as shown 
    in the statement above.
\end{proof}

This establishes all parts of Theorem \ref{thm:main1}.
Next, in Section \ref{sec:Weyl alt sets and diagrams}, we compute the Weyl alternation sets for $\sl4$. Then, in Section \ref{sec:qmult}, we utilize these results along with Theorem \ref{thm:main1} to provide a closed formula for the $q$-analog of Kostant's weight multiplicity formula for the Lie algebra $\sl4$.

\section{Weyl alternation sets and diagrams  of \texorpdfstring{$\mathfrak{sl}_4(\mathbb{C})$}{sl4(C)}}\label{sec:Weyl alt sets and diagrams}

We begin by recalling that the theorem of the highest weight asserts that every  nonnegative integral linear combination of the fundamental weights (i.e. every dominant weight) of a semisimple Lie algebra $\mathfrak{g}$ is the highest weight of an irreducible finite-dimensional representation of $\mathfrak{g}$ \cite[Chapter 3]{GW}. Thus, the nonnegative octant of the fundamental weight lattice is the only portion of this lattice that is significant in the representation theory of Lie algebras. However, our work on Weyl alternation sets considers weights on all of the  fundamental weight lattice since interesting symmetries arise in the geometry of the associated Weyl alternation diagrams. 

\subsection{Weyl alternation sets}\label{subsec:Weyl alt sets}
Throughout we let  $\lambda=m\w_1 + n\w_2 + k\w_3$ with $m,n,k \in \Z$ and $\mu = c_1\w_1 + c_2\w_2 + c_3\w_3 $ with $c_1,c_2,c_3 \in \Z$. In order to compute the Weyl alternation sets $\A(\lambda,\mu)$ we need to determine when  $\wp( \sigma(\la + \rho) - (\mu + \rho))>0$ for a fixed $\sigma\in W$.  This is equivalent to determining when $\sigma(\la + \rho) - (\mu + \rho)$ can be expressed as a nonnegative integral sum of simple roots. 

As an example, consider the case when $\sigma=1$, the identity element of $W$. In this case, note
\begin{align}\label{sigma_1}
    1(\la + \rho) - (\rho + \mu) 
    = &\left(\frac{3m+2n+k-3c_1-2c_2-c_3}{4}\right) \al_1 
    + \left(\frac{m+2n+k-c_1-2c_2-c_3}{2}\right) \al_2 \\
    &\qquad\qquad\qquad\qquad+ \left(\frac{m+2n+3k-c_1-2c_2-3c_3}{4}\right) \al_3. \nonumber
\end{align}
Hence, $1\in\A(\lambda,\mu)$ if and only if the coefficients of $\a_1$, $\a_2$, and $\a_3$ in \eqref{sigma_1} are nonnegative integers. We first deal with the fact that these coefficients must be integers, a condition which we henceforth refer to as the \textit{integrality condition}. We later focus on the restriction that these coefficients are nonnegative, which we henceforth refer to as the \textit{nonnegativity condition}. Let $b_1,b_2,b_3$ represent the coefficients of $\al_1,\al_2,\al_3$ in \eqref{sigma_1}, respectively.
In order for $b_1$, $b_2$, and $b_3$ to be integers there must exist $x,y,z\in\N$ such that
\begin{align*}
3m+2n+k-3c_1-2c_2-c_3 &= 4x,\\    m+2n+k-c_1-2c_2-c_3 &= 2y, \\
m+2n+3k-c_1-2c_2-3c_3 &= 4z.
\end{align*}
Since $c_1$, $c_2$, and $c_3$ are fixed, we obtain the following system of linear 
equations in the variables $m$, $n$, and $k$:
\begin{align*}
    3m + 2n + k &= 4x + 3c_1 + 2c_2 + c_3 \\
    m + 2n + k &= 2y + c_1 + 2c_2 + c_3 \\ 
    m + 2n + 3k &= 4z + c_1 + 2c_2 + 3c_3
\end{align*}
with solution
\begin{align}
    m &= 2x -y + c_1 \label{sigma_1_sub} \\
    n&= -x +2y -z +c_2 \nonumber \\
    k&= -y +2z + c_3. \nonumber 
\end{align}
Substituting \eqref{sigma_1_sub} into \eqref{sigma_1} yields
\begin{align}
    1(\la-\rho)-(\rho+\mu) = x\al_1+y\al_2+z\al_3. \label{sigma_1_sub_1}
\end{align}
Thus the set of all lattice points in (\ref{sigma_1}) as $m,n,k\in\Z$ vary equals the set of all lattice points in (\ref{sigma_1_sub_1}) as $x,y,z\in\Z$ vary. 

Next, we consider a similar process for  the case $\sigma=s_1$ and find that
\begin{align}
    s_1(\la + \rho) - (\rho + \mu) &= \label{sigma_s1}
    \left(\frac{-m+2n+k -3c_1 - 2c_2 - c_3}{4} - 1 \right)\al_1 + 
    \left( \frac{m + 2n + k - c_1 - 2c_2 - c_3}{2} \right)\al_2 \\
    &\qquad\qquad\qquad\qquad + \nonumber
    \left( \frac{m + 2n + 3k - c_1 - 2c_2 - 3c_3}{4} \right) \al_3
\end{align}
has corresponding system
\begin{align*}
    -m + 2n + k &= 4x' + 3c_1 + 2c_2 + c_3 \\
    m + 2n + k &= 2y' + c_1 + 2c_2 + c_3 \\
    m + 2n + 3k &= 4z' + c_1 + 2c_2 + 3c_3
\end{align*}
with $x',y',z'\in\Z$, which has  solution
\begin{align}
    m &= -2x' + y' - c_1 \label{sigma_s1_sub}\\
    n &= x' + y' - z' + c_1 + c_2 \nonumber\\
    k &= -y' + 2z' + c_3. \nonumber
\end{align}
Substituting \eqref{sigma_s1_sub} into \eqref{sigma_s1} yields
\begin{align}\label{final s_1}
    s_1(\la + \rho) - (\rho + \mu)&=(x'-1)\a_1+y'\a_2+z'\a_3.
\end{align}

Repeating this process in the case $\sigma=s_2$, we find that
\begin{align}
    s_2(\la+\rho) - (\rho+\mu) &= \label{sigma_s2}
    \left( \frac{3m+2n+k-3c_1-2c_2-c_1}{4} \right) \al_1 + \left( \frac{m+k-c_1-2c_2-c_3}{2} - 1 \right) \al_2\\
    &\qquad\qquad\qquad\qquad+ \left( \frac{m+2n+3k - c_1-2c_2-3c_3}{4} \right) \al_3 \nonumber
\end{align}
has corresponding system
\begin{align*}
    3m+2n+k &= 4x''+3c_1+2c_2+c_3 \\
    m+k &= 2y'' + c_1 + 2c_2 + c_3 \\
    m+2n+3k &= 4z'' + c_1 + 2c_2 + 3c_3
\end{align*}
with $x'',y'',z''\in\Z$, which has solution 
\begin{align}
    m &= x'' + y'' - z'' + c_1 + c_2 \label{sigma_s2_sub}\\
    n &= x'' - 2y'' + z'' - c_2 \nonumber\\
    k &= -x'' + y'' + z'' + c_2 + c_3. \nonumber
\end{align}
Substituting \eqref{sigma_s2_sub} into \eqref{sigma_s2} yields
\begin{align}\label{final s_2}
    s_2(\la + \rho) - (\rho + \mu)&=x''\a_1+(y''-1)\a_2+z''\a_3.
\end{align}

Repeating the process one final time in the case $\sigma = s_3$, we find that
\begin{align}
    s_3(\la + \rho) - (\rho + \mu) &= \label{sigma_s3}
    \left( \frac{3m+2n+k-3c_1-2c_2-c_3}{4} \right)\al_1 + \left( \frac{m + 2n + k - c_1 - 2c_2 - c_3}{2} \right) \al_2 \\ \nonumber
    &\qquad\qquad\qquad\qquad+ \left( \frac{m + 2n - k - c_1 - 2c_2 - 3c_3}{4} - 1 \right) \al_3
\end{align}
has corresponding system
\begin{align*}
    3m + 2n + k &= 4x''' + 3c_1 + 2c_2 + c_3 \\
    m + 2n + k &= 2y''' + c_1 + 2c_2 + c_3 \\
    m + 2n - k &= 4z''' + c_1 + 2c_2 + 3c_3
\end{align*}
with $x''',y''',z'''\in\Z$, which has solution 
\begin{align}
    m &= 2x''' - y''' + c_1 \label{sigma_s3_sub}\\
    n &= -x''' + y''' + z''' + c_2 + c_3 \nonumber\\
    k &= y''' - 2z''' - c_3. \nonumber
\end{align}
Substituting \eqref{sigma_s3_sub} into \eqref{sigma_s3} yields
\begin{align}\label{final s_3}
    s_3(\la + \rho) - (\rho + \mu)&=x'''\a_1+y'''\a_2+(z'''-1)\a_3.
\end{align}

Each of the substitutions made allowed us to express $\sigma(\lambda+\rho)-\rho-\mu$ as an integral sum of simple roots, provided $\sigma=1,s_1,s_2,$ or $s_3$. 
However, it turns out that substituting \eqref{sigma_1_sub} into \eqref{sigma_s1}, \eqref{sigma_s2}, and \eqref{sigma_s3} yields
\begin{align}
        s_1(\la+\rho)-(\rho+\mu)& = (-x+y-c_1-1)\al_1+y\al_2+z\al_3
        \label{s1}\\
        s_2(\la+\rho)-(\rho+\mu) &= x\al_1+(-c_2+x-y+z-1)\al_2 + z\al_3
        \label{s2}\\
        s_3(\la+\rho)-(\rho+\mu) &= x\al_1+y\al_2+(-c_3+y-z-1)\al_3.
        \label{s3}
    \end{align}
Thus, this single substitution ensures that the coefficients of  $\al_1,\al_2,\al_3$ in the above equations are integers whenever $x,y,z\in\Z$. In light of this it is natural to ask if these substitutions always preserve the set of lattice points in $\sigma(\la+\rho)-(\rho+\mu)$. As it turns out, the answer is yes. We prove this next and henceforth we use the substitution for $\sigma=1$ as given in \eqref{sigma_1_sub}. 

\begin{proposition}\label{prop:same lattices}
    Let $m,n,k,c_1,c_2,c_3\in\Z$.
    The lattices obtained by letting $x,y,z\in\Z$ vary in 
    \eqref{s1}, \eqref{s2}, and \eqref{s3} equal the integer lattices 
    obtained by letting $m,n,k\in\Z$ vary
    in \eqref{final s_1}, \eqref{final s_2}, and \eqref{final s_3}, 
    respectively.
    \label{prop:div}
\end{proposition}
\begin{proof}
    The set of lattice points of the form in \eqref{sigma_s1} equals the set of points $(x'-1)\al_1+y'\al_2+z'\al_3$ for $x',y',z'\in\Z$. Now, notice if $x' = -x+y-c_1$, $y'=y$, and $z'=z$, then $(x'-1)\al_1 + y'\al_2 + z'\al_3 = (-x+y-c_1-1)\al_1+y\al_2+z\al_3.$ Hence, lattice points of the form in \eqref{s1} can be written in the form given in \eqref{final s_1}. Conversely, if $x = -x'+y-c_1$, $y=y'$, and $z=z'$, then $ (-x+y-c_1-1)\al_1 + y\al_2 + z\al_3 = (x'-1)\al_1+y'\al_2+z'\al_3.$ Thus, \eqref{final s_1} and \eqref{s1} give the same lattice.
    
    The set of lattice points of the form in \eqref{sigma_s2} equals the set of points $x''\al_1+(y''-1)\al_2+z''\al_3$ for $x'',y'',z''\in\Z$. If $x''=x$, $y'' = -c_2+x-y+z$, and $z''=z$, then $x''\al_1 + (y''-1)\al_2 + z''\al_3 = x\al_1 + (-c_2+x-y+z-1)\al_2 + z\al_3.$ Hence, lattice points of the form in \eqref{s2} can be written in the form \eqref{final s_2}. Conversely, if $x=x''$, $y = -c_2+x-y''+z$, and $z = z''$, then $x\al_1 + (-c_2+x-y+z-1)\al_2 + z\al_3 = x''\al_1 + (y''-1)\al_2 + z''\al_3.$ Thus,  \eqref{final s_2} and \eqref{s2} give the same lattice.

    The set of lattice points of the form in \eqref{sigma_s3} equals the set of points $x'''\al_1+y'''\al_2+(z'''-1)\al_3$ for $x''',y''',z'''\in\Z$. If $x'''=x$, $y'''=y$, and $z''' = -c_1+y-z$, then $x'''\al_1 + y'''\al_2 + (z'''-1)\al_3 = x\al_1 + y\al_2 + (-c_3+y-z-1)\al_3.$ Hence, lattice points of the form \eqref{s3} can be written in the form \eqref{final s_3}. Conversely, if $x=x'''$, $y=y'''$, and $z = -c_1+y-z'''$, then $x\al_1 + y\al_2 + (-c_3+y-z-1)\al_3 = x'''\al_1 + y'''\al_2 + (z'''-1)\al_3.$ Thus, \eqref{s3} and \eqref{final s_3} give the same lattice.
\end{proof}

Proposition \ref{prop:same lattices} and the substitution from Lemma \ref{lem:integral} allow us reduce our work to describing the Weyl alternation sets in the case where $\lambda$ is in the root lattice and $\mu$ is in the nonnegative octant of the root lattice. 

\begin{theorem}\label{thm:195 alt sets} 
Let $\lambda= (2x-y+c_1)\w_1 + (2y-x-z+c_2)\w_2 + (2z-y+c_3)\w_3$ and fix  $\mu=c_1\w_1 + c_2\w_2 + c_3\w_3$ with $x,y,z\in\Z$ and $c_1,c_2,c_3 \in \N$. Then there are 195 distinct Weyl alternation sets $\A(\lambda,\mu)$ and these are listed explicitly in Appendix \ref{appendix:AS}.
\end{theorem}

\begin{proof}
Let $\lambda= (2x-y+c_1)\w_1 + (2y-x-z+c_2)\w_2 + (2z-y+c_3)\w_3$ for some $x,y,z\in\Z$ and a fixed $\mu=c_1\w_1 + c_2\w_2 + c_3\w_3$ with $c_1,c_2,c_3 \in \N$. Direct computations show that $\sigma(\lambda+\rho)-\rho-\mu$ can be written as a sum of simple roots for each $\sigma\in W$ as is given in Table \ref{tab:tablewithelementsFULLCOEFFS}. From these computations and the definition of a Weyl alternation set, it follows that \begin{align*}
    1 \in \A(\lambda,\mu) \iff & x\geq 0,  y\geq 0, \ \text{and} \  z\geq 0\\
    s_1\in \A(\lambda,\mu) \iff& (-c_1-x+y-1)\geq 0,y\geq 0, \ \text{and} \ z\geq 0\\
    s_2\in \A(\lambda,\mu) \iff& x\geq 0,(-c_2 + x - y + z - 1)\geq 0, \ \text{and} \ z\geq 0\\
    s_3\in \A(\lambda,\mu) \iff& x\geq 0,y\geq 0, \ \text{and} \ (-c_3 + y - z - 1)\geq 0\\
    s_1s_2\in \A(\lambda,\mu) \iff& (-c_1-c_2-y+z-2)\geq 0,(-c_2 + x - y + z - 1)\geq 0, \ \text{and} \ z\geq 0\\
    s_2s_1\in \A(\lambda,\mu) \iff& (-c_1-x+y-1)\geq 0, (-c_1 - c_2 - x + z - 2)\geq 0, \ \text{and} \ z\geq 0\\
    s_2s_3\in \A(\lambda,\mu) \iff& x\geq 0,(-c_2 - c_3 +x - z - 2 )\geq 0, \ \text{and} \ (-c_3 + y - z - 1)\geq 0\\
    s_3s_1\in \A(\lambda,\mu) \iff& (-c_1-x+y-1)\geq 0,y\geq 0, \ \text{and} \ (-c_3 + y - z - 1)\geq 0\\
    s_3s_2\in \A(\lambda,\mu) \iff& x\geq 0,(-c_2 + x - y + z - 1)\geq 0, \ \text{and} \ ( -c_2 - c_3 + x - y - 2)\geq 0\\
    s_1s_2s_1\in \A(\lambda,\mu) \iff& (-c_1-c_2-y+z-2)\geq 0,(-c_1 - c_2 - x + z - 2)\geq 0, \ \text{and} \ z\geq 0\\
    s_1s_2s_3\in \A(\lambda,\mu) \iff& (-c_1 - c_2 - c_3 - z - 3)\geq 0,(-c_2 - c_3 +x - z - 2)\geq 0, \ \text{and} \ \\ &( -c_3 + y - z - 1)\geq 0\\
    s_2s_3s_1\in \A(\lambda,\mu) \iff& (-c_1-x+y-1)\geq 0,(-c_1-c_2-c_3-x+y-z-3)\geq 0, \ \text{and} \ \\ & (-c_3 + y - z - 1)\geq 0\\
    s_2s_3s_2\in \A(\lambda,\mu) \iff& x\geq 0,(-c_2 - c_3 +x - z - 2)\geq 0, \ \text{and} \ (-c_2 - c_3 + x - y - 2)\geq 0\\
    s_3s_1s_2\in \A(\lambda,\mu) \iff& (-c_1-c_2-y+z-2)\geq 0,(-c_2 + x - y + z - 1)\geq 0, \ \text{and} \ \\ & ( -c_2 - c_3 + x - y - 2)\geq 0\\
    s_3s_2s_1\in \A(\lambda,\mu) \iff& (-c_1-x+y-1)\geq 0,(-c_1 - c_2 - x + z - 2 )\geq 0, \ \text{and} \ \\ &(-c_1 - c_2 - c_3 - x - 3 )\geq 0\\
    s_1s_2s_3s_1\in \A(\lambda,\mu) \iff& (-c_1 - c_2 - c_3 - z - 3)\geq 0,(-c_1-c_2-c_3-x+y-z-3)\geq 0, \ \text{and} \ \\ &(-c_3 + y - z - 1)\geq 0\\
    s_1s_2s_3s_2\in \A(\lambda,\mu) \iff& (-c_1 - c_2 - c_3 - z - 3)\geq 0,(-c_2 - c_3 +x - z - 2 )\geq 0, \ \text{and} \ \\ &(-c_2 - c_3 + x - y - 2)\geq 0\\
    s_2s_3s_1s_2\in \A(\lambda,\mu) \iff& (-c_1-c_2-y+z-2)\geq 0,(-c_1 - 2c_2 - c_3 - y - 4)\geq 0, \ \text{and} \ \\ &(-c_2 - c_3 + x - y - 2)\geq 0\\
    s_2s_3s_2s_1\in \A(\lambda,\mu) \iff& (-c_1-x+y-1)\geq 0,(-c_1-c_2-c_3-x+y-z-3)\geq 0, \ \text{and} \ \\ &(-c_1 - c_2 - c_3 - x - 3)\geq 0\\
    s_3s_1s_2s_1\in \A(\lambda,\mu) \iff& (-c_1-c_2-y+z-2)\geq 0,(-c_1 - c_2 - x + z - 2)\geq 0, \ \text{and} \ \\ &(-c_1 - c_2 - c_3 - x - 3)\geq 0\\
    s_1s_2s_3s_1s_2\in \A(\lambda,\mu) \iff& (-c_1 - c_2 - c_3 - z - 3)\geq 0,(-c_1 - 2c_2 - c_3 - y - 4)\geq 0, \ \text{and} \ \\ &(-c_2 - c_3 + x - y - 2)\geq 0\\
    s_1s_2s_3s_2s_1\in \A(\lambda,\mu) \iff & (-c_1 - c_2 - c_3 - z - 3)\geq 0,  (-c_1-c_2-c_3-x+y-z-3)\geq 0, \ \text{and} \ \\ &  (-c_1 - c_2 - c_3 - x - 3 )\geq 0\\
    s_2s_3s_1s_2s_1\in \A(\lambda,\mu) \iff& (-c_1-c_2-y+z-2)\geq 0,(-c_1 - 2c_2 - c_3 - y - 4)\geq 0, \ \text{and} \ \\ & (-c_1 - c_2 - c_3 - x - 3)\geq 0\\
    s_1s_2s_3s_1s_2s_1\in \A(\lambda,\mu) \iff& ( -c_1 - c_2 - c_3 - z - 3)\geq 0, (-c_1 - 2c_2 - c_3 - y - 4)\geq 0, \ \text{and} \ \\ &(-c_1 - c_2 - c_3 - x - 3)\geq 0.
\end{align*}
Intersecting these solutions sets on the lattice $\Z{\al_1} \oplus \Z{\al_2}\oplus \Z{\al_3}$ produces the desired results.
\begin{table}[h]
\begin{tabular}{|l|l||l|l||l|l|}\hline
$\sigma\in W$& $\sigma(\lambda+\rho)-\rho-\mu$&$\sigma\in W$& $\sigma(\lambda+\rho)-\rho-\mu$&$\sigma\in W$& $\sigma(\lambda+\rho)-\rho-\mu$\\\hline\hline
1 &
$(P_1,Q_1,R_1)$&
$s_1$ &
$(P_4,Q_1,R_1)$&
$s_2$ &
$(P_1,Q_6,R_1)$\\\hline
$s_3$ &
$(P_1,Q_1,R_4)$&
$s_1s_2$ &
$(P_3,Q_6,R_1)$&
$s_2s_1$ &
$(P_4,Q_4,R_1)$\\\hline
$s_2s_3$ &
$(P_1,Q_5,R_4)$&
$s_3s_1$&
$(P_4,Q_1,R_4)$&
$s_3s_2$ &
$(P_1,Q_6,R_3)$\\\hline
$s_1s_2s_1$&
$(P_3,Q_4,R_1)$&
$s_1s_2s_3$ &
$(P_2,Q_5,R_4)$&
$s_2s_3s_1$&
$(P_4,Q_3,R_4)$\\\hline
$s_2s_3s_2$ &
$(P_1,Q_5,R_3)$&
$s_3s_1s_2$ &
$(P_3,Q_6,R_3)$&
$s_3s_2s_1$ &
$(P_4,Q_4,R_2)$\\\hline
$s_1s_2s_3s_1$&
$(P_2,Q_3,R_4)$&
$s_1s_2s_3s_2$ &
$(P_2,Q_5,R_3)$&
$s_2s_3s_1s_2$&
$(P_3,Q_2,R_3)$\\\hline
$s_2s_3s_2s_1$ &
$(P_4,Q_3,R_2)$&
$s_3s_1s_2s_1$ &
$(P_3,Q_4,R_2)$&
$s_1s_2s_3s_1s_2$ &
$(P_2,Q_2,R_3)$\\\hline
$s_1s_2s_3s_2s_1$&
$(P_2,Q_3,R_2)$&
$s_2s_3s_1s_2s_1$ &
$(P_3,Q_2,R_2)$&
$s_1s_2s_3s_1s_2s_1$ &
$(P_2,Q_2,R_2)$\\\hline
\end{tabular}
\begin{tabular}{lll}
$P_1=x$&
$P_2=-c_1-c_2-c_3-z-3$&
$P_3=-c_1-c_2-y+z-2$\\
$P_4=-c_1-x+y-1$&
$Q_1=y$&
$Q_2=-c_1-2c_2-c_3-y-4$\\
$Q_3=-c_1-c_2-c_3-x+y-z-3$&
$Q_4=-c_1-c_2-x+z-2$&
$Q_5=-c_2-c_3+x-z-2$\\
$Q_6=-c_2+x-y+z-1$&
$R_1=z$&
$R_2=-c_1-c_2-c_3-x-3$\\
$R_3=-c_2-c_3+x-y-2$&
$R_4=-c_3+y-z-1$&\\
\end{tabular}
\caption{Notation for $\sigma(\lambda+\rho)-\rho-\mu$ as a sum of simple roots for each $\sigma\in W$, where $(X,Y,Z):=X\a_1+Y\a_2+Z\a_3$.}
\label{tab:tablewithelementsFULLCOEFFS}
\end{table}

\end{proof}
Theorem \ref{thm:195 alt sets} states that there are 195 distinct Weyl alternation sets, which can be attained as $\lambda$ varies within the root lattice. In the next section, we give examples which realize each of these sets when $\mu=0$.

\begin{table}[htp]
\centering
\resizebox{6.75in}{!}{
\includegraphics[width=\textwidth]{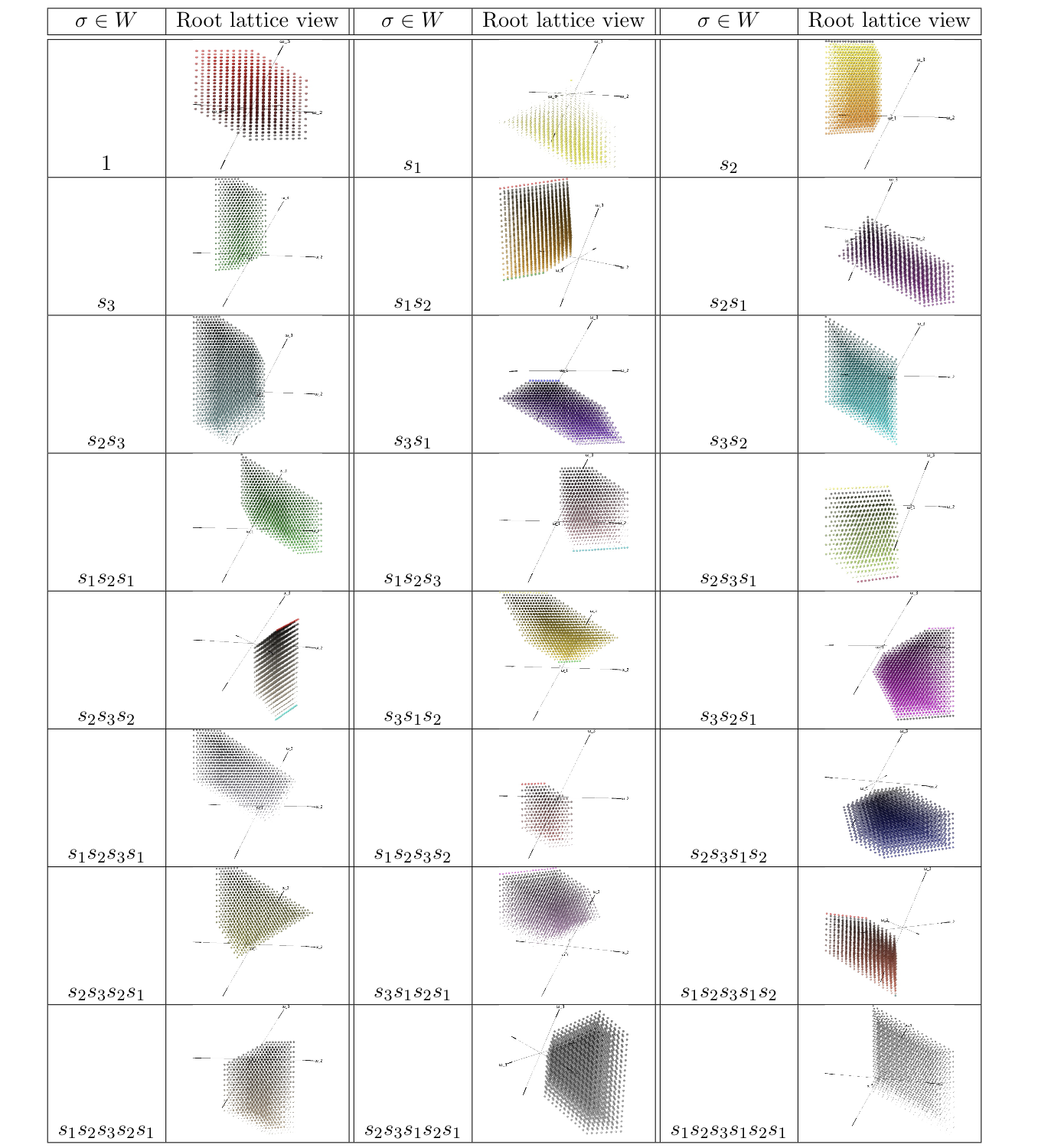}
}
\caption{Solution regions for inequalities determined by each Weyl group element in the case that $\mu=0$.}\label{tab:weyl solutions}
\end{table}

\subsection{Weyl alternation diagrams}
The proof of Theorem \ref{thm:195 alt sets} provides inequalities which describe when certain elements of the Weyl group are in the Weyl alternation set $\A(\lambda,\mu)$. Each Weyl group element provides 3 inequalities, whose solution set cuts out a cone of the lattice $\mathbb{Z}\w_1\oplus\mathbb{Z}\w_2\oplus\mathbb{Z}\w_3$. Table \ref{tab:weyl solutions} illustrates  these regions in the root lattice for each element in $W$ and when $\mu=0$.

We note that as $\mu$ varies within the nonnegative octant of the root lattice, the diagrams in Table \ref{tab:weyl solutions} translate but their geometry remains the same. Moreover, for any fixed $\mu$, by intersecting the associated solution regions we can create the associated $\mu$ Weyl alternation diagram. 

As is evident, 
Weyl alternation diagrams 
are best analyzed when one has the ability to rotate an image and view it from a multitude of angles, we provide code which allows a user to input a weight $\mu$ and whose output is the set of Weyl alternation diagrams associated to each of the Weyl alternation sets determined by $\mu$. The code, along with instructions on how to use it, are found in the GitHub repository at \textcolor{blue}{\href{https://github.com/melendezd/Weight-Multiplicities}{https://github.com/melendezd/Weight-Multiplicities}}.
However, here we present a set of 2-dimensional diagrams containing the same information for the case when $\mu=0$. 
\begin{example}
To illustrate the $\mu=0$ Weyl alternation diagrams, we fix an integer $z_0$ (satisfying $-5\leq z_0\leq 5$), and give a Weyl alternation diagram on the integral lattice spanned by $w_1$ and $w_2$ at height $z_0$. That is, we consider the set of weights
$\lambda=x\w_1+y\w_2+z_0\w_3$ with $x,y\in Z$ satisfying $-25\leq x,y\leq 25$, and we color all lattice points with the same color when they share a Weyl alternation set $\A(\lambda,0)$. This produces the diagrams in Figure~\ref{fig:mu=0 weyl diagrams}.
\begin{figure}[h]
    \centering
\includegraphics[width=\textwidth]{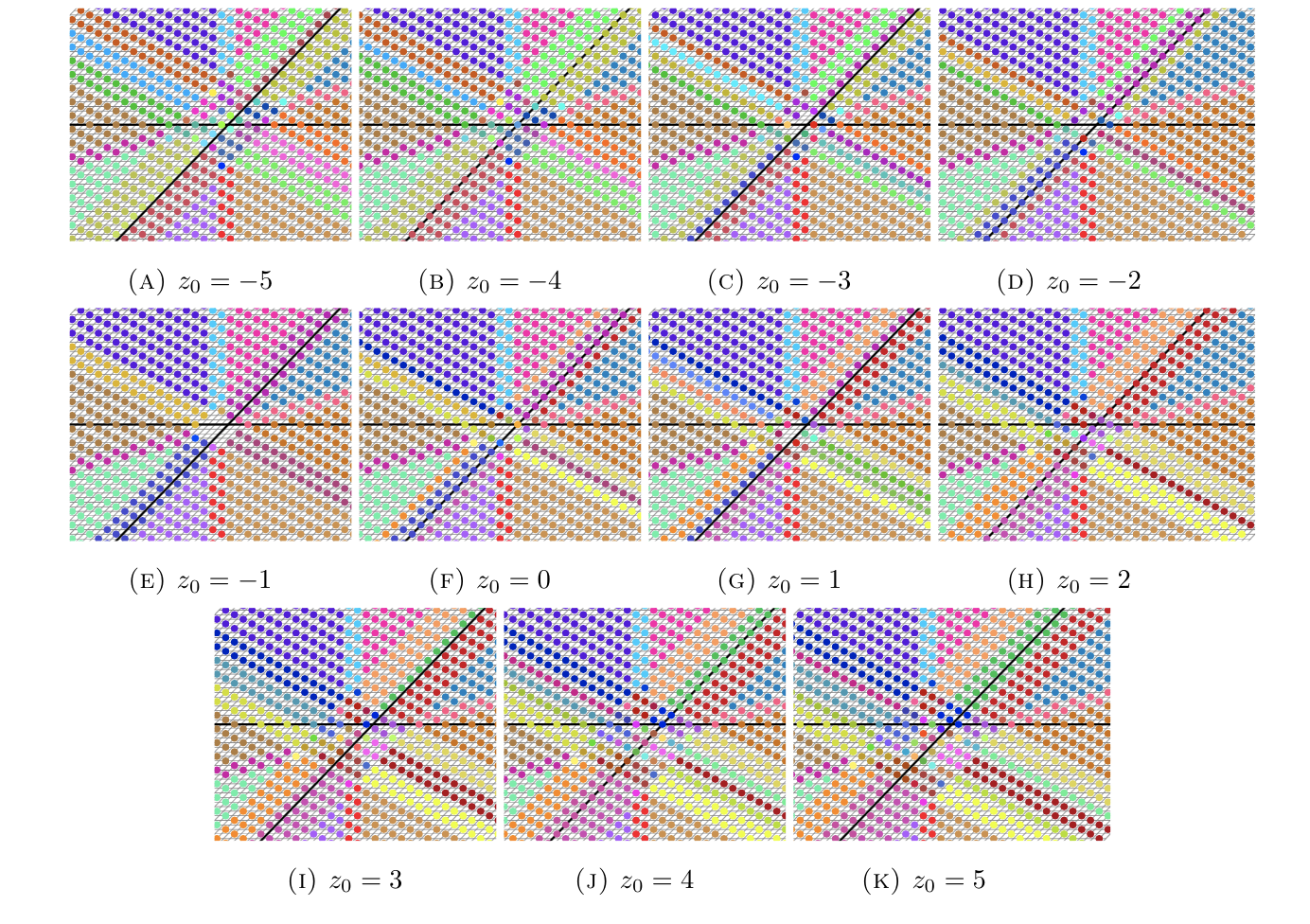}
    \caption{The $\mu=0$ Weyl alternation diagrams for $\sl4$. The horizontal axis is the real span of $\w_1$, while the other is the real span of $\w_2$. The key for these diagram is given in Appendix \ref{appendix:AS}.}
\label{fig:mu=0 weyl diagrams}
\end{figure}
\end{example}

Another set of diagrams associated to Weyl alternation sets are known as empty regions. We recall that for a fixed $\mu$ in the root lattice, the empty region consists of the set of weights $\lambda$ on the root lattice for which $\A(\lambda,\mu)=\emptyset$. In the following examples, we fix $\mu$ in the root lattice, color red every weight $\lambda$ in the root lattice for which $\A(\lambda,\mu)=0$, and give a geometric description and visualization of the empty region. Note that at times we increase the size of the vertices to better illustrate the behavior of the empty region.
\vspace{.1in}
{ 
\begin{enumerate}[label=Case \arabic*]
    \item\hspace{-1.5mm}:  $\mu=n\al_1$.  This empty region is in the shape of a tetrahedron. Additionally, the tetrahedron increases in size as the coefficient of $\alpha_1$ increases.
    
    \begin{center}
\includegraphics[width=5in]{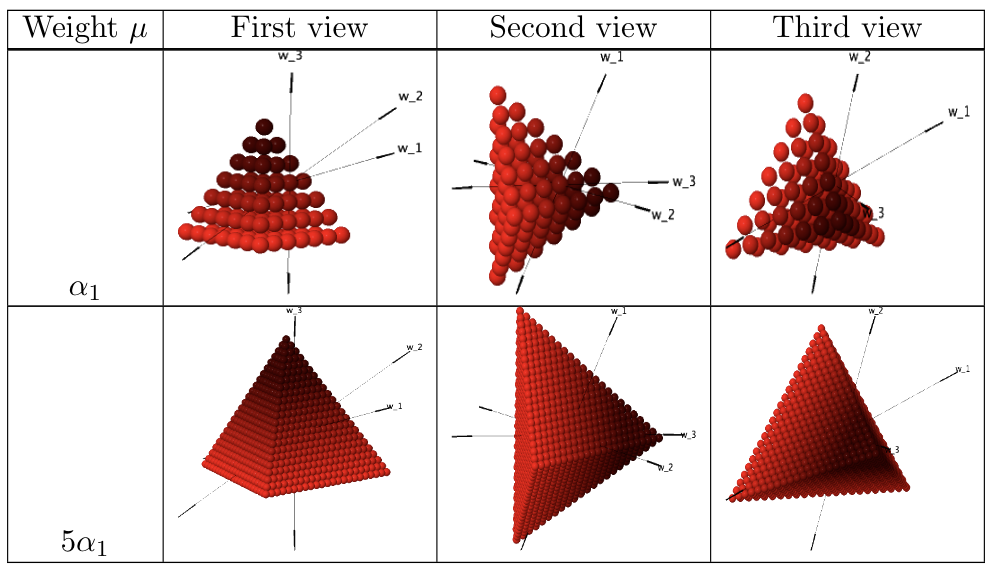}
    \end{center}
    
    \item\hspace{-1.5mm}: $\mu=n\al_2$. This empty region takes the form of a cube. As before, the cube increases in size as the coefficient of $\al_2$ increases.

        \begin{center}
\includegraphics[width=5in]{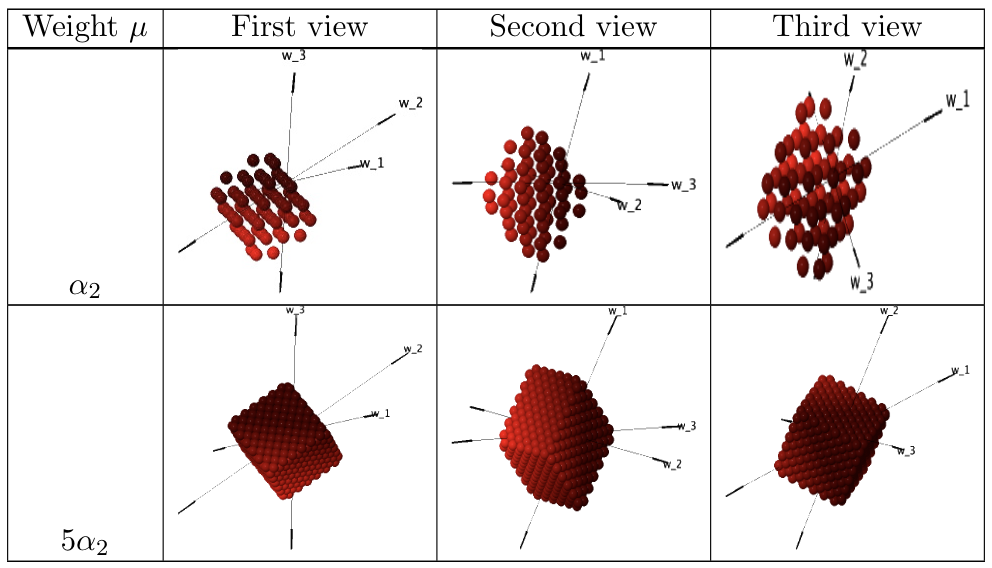}
    \end{center}

    \item\hspace{-1.5mm}: $\mu=n\al_3$. This empty region forms a tetrahedron but with a different orientation than when $\mu=n\a_1$. As before, the tetrahedron increases in size as the coefficient of $\a_3$ increases.
    
        \begin{center}
\includegraphics[width=5in]{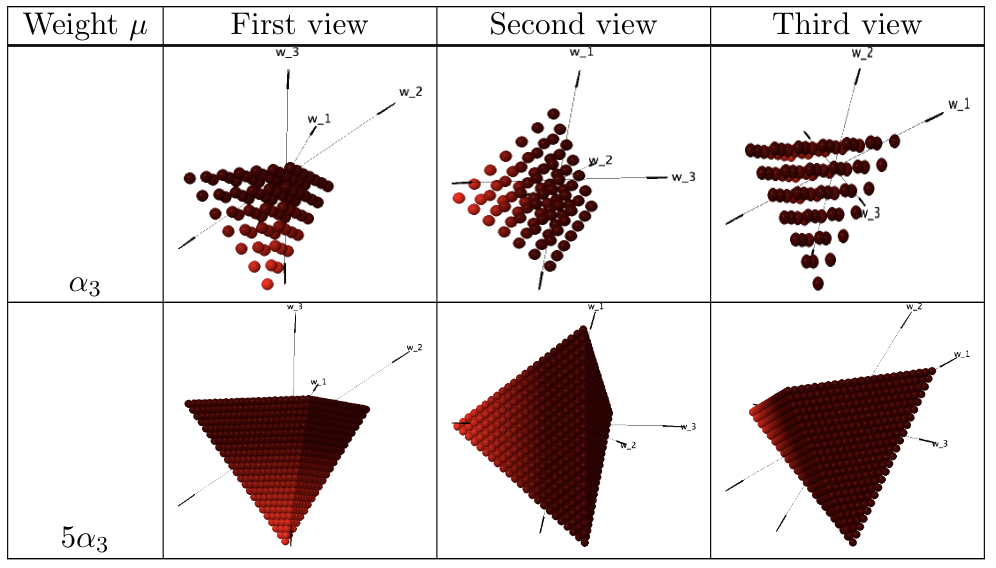}
    \end{center}
    
    \item\hspace{-1.5mm}:\label{geoempty:case4} $\mu=n\al_1+n\al_2$. Once again the empty region is a tetrahedron, but we observe that each face shows a set of points near their center. As $n$ increases, so does size of the region.

        \begin{center}
\includegraphics[width=5in]{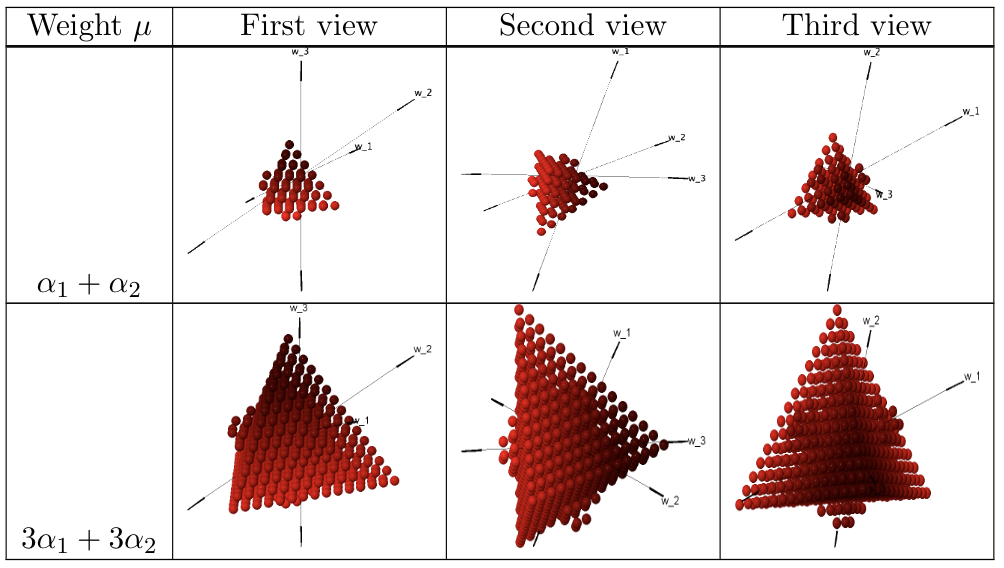}
    \end{center}

\item\hspace{-1.5mm}: $\mu=n\al_2+n\al_3$. The empty region is a tetrahedron with a few points coming out of it. However, the orientation and the size is different from \ref{geoempty:case4}. As the coefficients for $\al_2,\al_3$ increase, the shape gets bigger.
\vspace{.1in}

        \begin{center}
\includegraphics[width=5in]{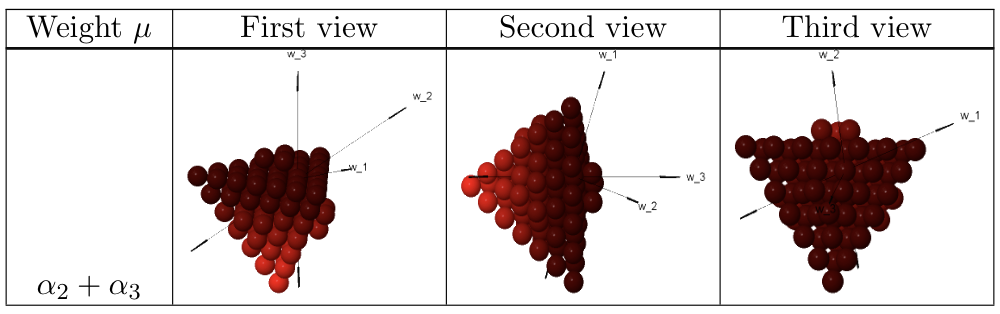}\\
\includegraphics[width=5in]{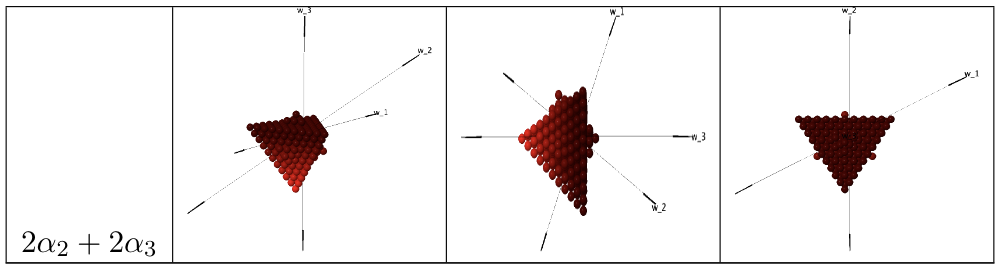}
    \end{center}
    
    \item\hspace{-1.5mm}: $\mu=n\al_1+n\al_3$. The empty region is a stellated octahedron, being formed by two tetrahedrons. As before, the size of the region grows as $n$ grows.
    
        \begin{center}
\includegraphics[width=5in]{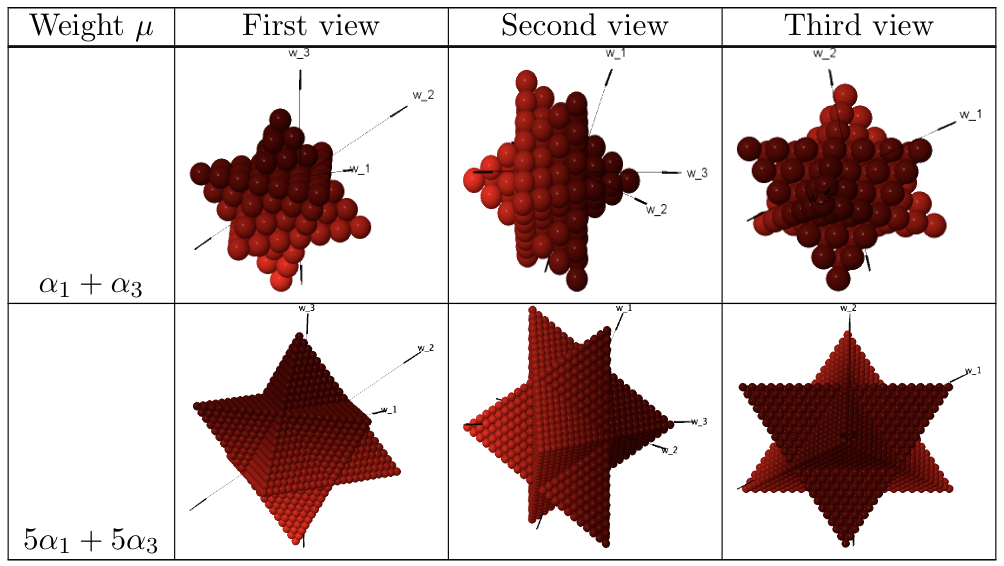}
    \end{center}

    \item\hspace{-1.5mm}: $\mu=n\al_1+n\al_2+n\al_3$. The empty region is a stellated octahedron, being formed by two tetrahedrons. Note that as $n$ grows, so does the size of the region.

        \begin{center}
\includegraphics[width=5in]{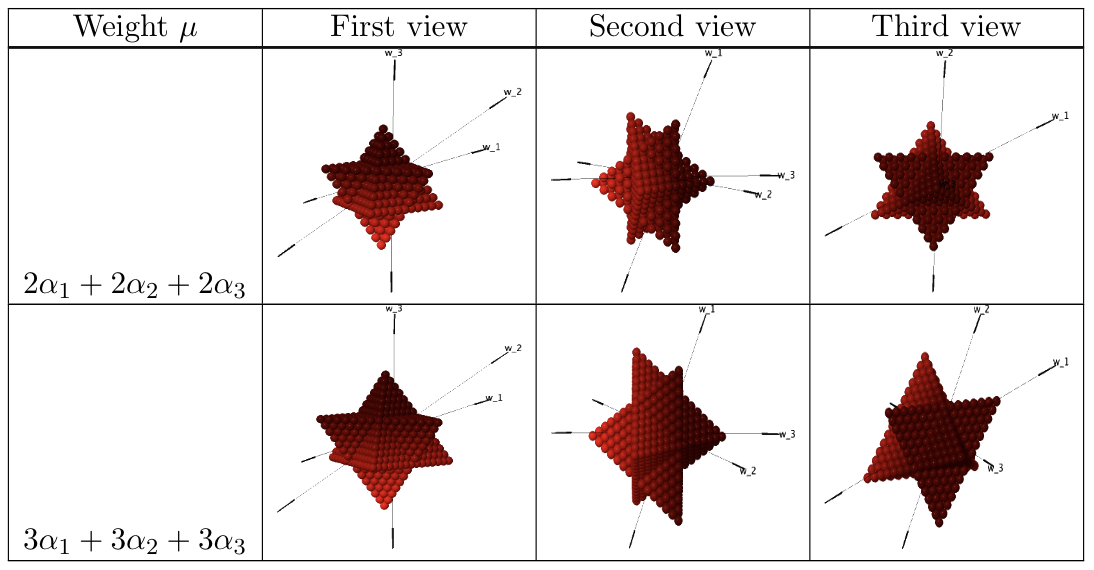}
    \end{center}
        
    \item\hspace{-1.5mm}:
    $\mu=m\a_1+n\a_2+k\a_3$.
    Here, note that we let $m, n,$ and $k$ be of different parity. When $m, k$ are odd and $n$ is even, the empty region is a tetrahedron. However, if $n$ is odd and $m, k$ are even, we get a stellated octahedron, with one tetrahedron smaller than the other. These shapes seem to be consistently determined by the parity of $m, n, k$ as described. Again, as the  coefficients $m,n,k$ grow, so does the size of the region.

        \begin{center}
\includegraphics[width=5in]{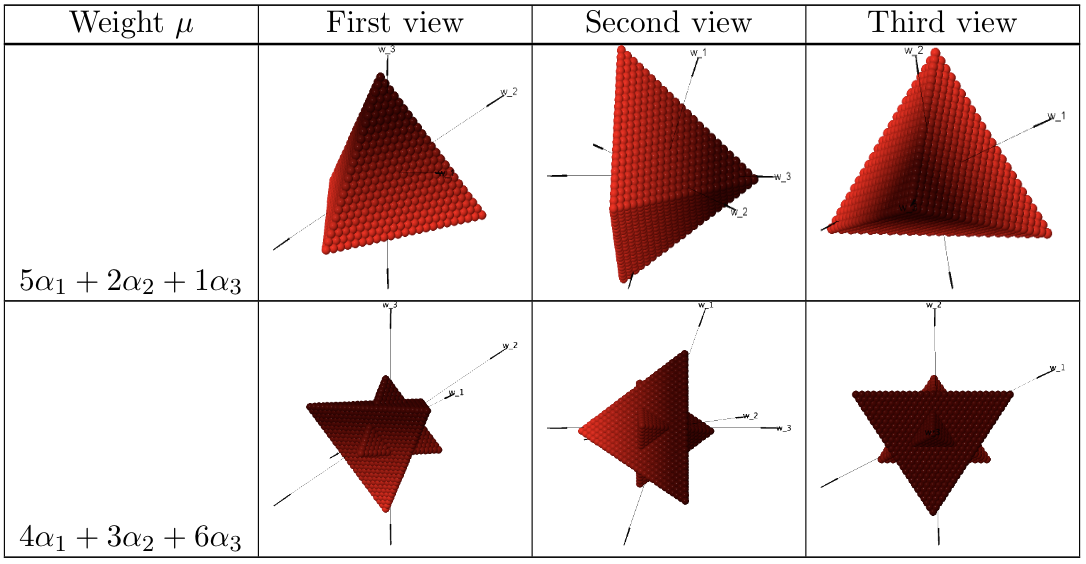}
    \end{center}
    
\end{enumerate}
}

\section{The  \texorpdfstring{$q$}{q}-analog of Kostant's weight multiplicity formula }\label{sec:qmult}
In this section we use Theorem \ref{thm:main1} and Theorem \ref{thm:195 alt sets} to give a formula for the $q$-multiplicity of a weight $\mu$ in $L(\lambda)$, a highest weight irreducible representation of $\sl4$ with highest weight $\lambda$. We begin by letting $\lambda=m\w_1+n\w_2+k\w_3$ and $\mu=c_1\w_1+c_2\w_2+c_3\w_3$, with $m,n,k,c_1,c_2,c_3\in\mathbb{N}$, and we set the notation for the rest of the section as given in Table \ref{tab:tablewithelementsFULLCOEFFS}.

\begin{theorem}\label{thm:qKWMF}
Let $\lam=m\w_1+n\w_2+k\w_3$ and
 $\mu=c_1\w_1+c_2\w_2+c_3\w_3$, with $m,n,k,c_1,c_2,c_3\in\N$. If
$ x=\frac{3m+2n+k-3c_1-2c_2-c_3}{4}$,
$y=\frac{m+2n+k-c_1-2c_2-c_3}{2}$, and
$z=\frac{m+2n+3k-c_1-2c_2-3c_3}{4}$,
then

{\footnotesize
\[
m_q(\lambda,\mu)=\begin{cases}
Z_1-Z_{11}-Z_3+Z_5+Z_{10}-Z_7 & P_1,Q_1,R_1,Q_6,R_4,Q_5,R_3 \in \N \ \text{and}\  P_4, P_3 \notin \N \\
Z_1-Z_6-Z_3+Z_4+Z_8-Z_2 & P_1,Q_1,R_1,P_4,Q_6,P_3,Q_4 \in \N \ \text{and} \ R_4, R_3 \notin \N\\
Z_1-Z_6-Z_{11}-Z_3+Z_8+Z_9 & P_1,Q_1,R_1,P_4,R_4,Q_6,Q_4 \in \N \ \text{and} \ P_3, Q_5, R_3 \notin \N\\
Z_1-Z_6-Z_{11}-Z_3+Z_9+Z_5 & P_1,Q_1,R_1,P_4,Q_6,R_4, Q_5 \in \N \ \text{and} \ P_3,Q_4,R_3 \notin \N\\
Z_1-Z_6-Z_{11}-Z_3+Z_9 & P_1,Q_1,R_1,P_4,Q_6,R_4 \in \N\ \text{and} \ P_3,Q_4, Q_5, R_3 \notin \N\\
Z_1-Z_{11}-Z_3+Z_5 & P_1,Q_1,R_1,Q_6,R_4,Q_5 \in \N \ \text{and} \ P_4,P_3,R_3 \notin \N \\
Z_1-Z_6-Z_3+Z_9 & P_1,Q_1,R_1,P_4,R_4 \in \N \ \text{and} \ Q_6,Q_4,Q_5 \notin \N\\
Z_1-Z_6-Z_{11}+Z_8 & P_1,Q_1,R_1,P_4,Q_6,Q_4 \in \N \ \text{and} \ R_4,P_3,R_3 \notin \N\\
Z_1-Z_{11}-Z_3 & P_1,Q_1,R_1,Q_6,R_4 \in \N \ \text{and} \ P_4,P_3, Q_5, R_3 \notin \N \\
Z_1-Z_6-Z_{11} & P_1,Q_1,R_1,P_4,Q_6 \in \N \ \text{and} \ R_4,P_3,Q_4,R_3 \notin \N\\
Z_1-Z_3 & P_1,Q_1,R_1,R_4 \in \N \ \text{and} \ Q_6,P_4,Q_5 \notin \N \ \text{and} \ (P_3 \notin \N \ \text{or} \ Q_4 \notin \N )\\
Z_1-Z_{11} & P_1,Q_1,R_1,Q_6 \in \N \ \text{and} \ P_4,R_4,R_3,P_3 \notin \N\\
Z_1 - Z_6 & P_1,Q_1,R_1, P_4 \in \N \ \text{and} \ Q_6,R_4,Q_4 \notin \N \ \text{and} \ (Q_5 \notin \N \ \text{or} \ R_3 \notin \N)\\
Z_1 & P_1,Q_1, R_1 \in \N \ \text{and} \ R_4,P_4,Q_6 \notin \N \ \text{and} \ (P_3 \notin \N \ \text{or} \ Q_4 \notin \N) \\& \text{and} \ (Q_5 \notin \N \ \text{or} \ R_3 \notin \N )\\
0 & \textnormal{otherwise,}
\end{cases}
\]
}
where $Z_1,\dots,Z_{11}$ are defined as in Appendix \ref{appendix:ztable}.
\end{theorem}
\begin{proof}
\indent Since $m,n,k,c_1,c_2,c_3\in\N$, we have
$P_2$, $Q_2$, $R_2$ are less than zero by inspection, and a straightforward substitution shows that $Q_3=-2m-2k-2c_1-4c_2-2c_3-3<0$.
Consequently, $\wp(\sigma(\lam+\mu)-(\rho+\mu))=0$ for all 
\[\sigma\in W_0:=\left\{\begin{matrix}s_2s_3s_1,s_3s_2s_1,s_2s_3s_2s_1,s_3s_1s_2s_1,s_2s_3s_1s_2,s_2s_3s_1s_2s_1,\\s_1s_2s_3,s_1s_2s_3s_2,s_1s_2s_3s_1,s_1s_2s_3s_2s_1,s_1s_2s_3s_1s_2,s_1s_2s_3s_1s_2s_1\end{matrix}\right\}.\]
Thus, the only possible Weyl group elements  contributing nontrivially to $m_q(\lambda,\mu)$ are those in $W\setminus W_0$. 
Table \ref{tab:tablewithelementsFULLCOEFFS} presents our notation for writing  $\sigma(\lambda+\rho)-\rho-\mu$ as a nonnegative integral sum of simple roots for all $\sigma\in W$ and in particular for $\sigma\in 
W \setminus W_0$.

Consider the case $\sigma=s_3s_1s_2$.
Note that \[P_3 = \frac{-m-2n+k-3c_1-2c_2-c_3}{4}-2
\qquad\mbox{and}\qquad R_3=\frac{m-2n-k-c_1-2c_2-3c_3}{4}-2.\]
In order for us to have $\wp_q(s_3s_1s_2(\lam+\rho)-(\rho+\mu))> 0$,
we must have in particular that $P_3\geq 0$ and $R_3\geq 0$.
As a result, we have that
\begin{align*}
    k &\geq 8+m+2n+3c_1+2c_2+c_3 \text{ and }
    k \leq -8+m-2n-c_1-2c_2-3c_3.
\end{align*}
It then follows that
\begin{align*}
    8+m+2n+3c_1+2c_2+c_3 \leq
    -8+m-2n-c_1-2c_2-3c_3,
\end{align*}
which implies
$0\geq 16+4n+4c_1+4c_2+4c_3$. 
Since $n,c_1,c_1,c_3\geq 0$, this is a contradiction, and thus $\wp_q(s_3s_1s_2(\lam+\rho)-(\rho+\mu)) = 0$ for all $\lam,\mu$ in the nonnegative octant of the fundamental weight lattice.

\indent Hence, the only remaining elements of the Weyl group for which it is a possibility that $\wp_q(\sigma(\lambda+\rho)-(\rho+\mu))>0$ are those in Table \ref{tab:last 11}.

\begin{table}[htp]
\begin{tabular}{|l|l|c|}
\hline
 $\sigma$ & $\xi=\sigma(\lam+\rho)-(\rho+\mu)$&$\wp_q(\xi)$ \\\hline\hline

1 &
$P_1\al_1 + Q_1\al_2 + R_1\al_3$ &$Z_{1}$ \\\hline

$s_1s_2s_1 $&
$P_3\al_1 + Q_4\al_2 + R_1\al_3$ &$Z_{2}$ \\\hline

$s_3$ &
$P_1\al_1 + Q_1\al_2 + R_4\al_3$ &$Z_{3}$ \\\hline

$s_1s_2$ &
$P_3\al_1 + Q_6\al_2 + R_1\al_3$ &$Z_{4}$ \\\hline

$s_2s_3$ &
$P_1\al_1 + Q_5\al_2 + R_4\al_3$ &$Z_{5}$ \\\hline

$s_1$ &
$P_4\al_1 + Q_1\al_2 + R_1\al_3$ &$Z_{6}$ \\\hline

$s_2s_3s_2$ &
$P_1\al_1 + Q_5\al_2 + R_3\al_3$ &$Z_{7}$ \\\hline

$s_2s_1$ &
$P_4\al_1 + Q_4\al_2 + R_1\al_3$ &$Z_{8}$ \\\hline

$s_3s_1 $&
$P_4\al_1 + Q_1\al_2 + R_4\al_3$ &$Z_{9}$ \\\hline

$s_3s_2$ &
$P_1\al_1 + Q_6\al_2 + R_3\al_3$ &$Z_{10}$ \\\hline

$s_2$ &
$P_1\al_1 + Q_6\al_2 + R_1\al_3$ &$Z_{11}$ \\\hline
\end{tabular}
\qquad
\begin{tabular}{l}
$P_1=x$\\
$P_3=-c_1-c_2-y+z-2$\\
$P_4=-c_1-x+y-1$\\
$Q_1=y$\\
$Q_4=-c_1-c_2-x+z-2$\\
$Q_5=-c_2-c_3+x-z-2$\\
$Q_6=-c_2+x-y+z-1$\\
$R_1=z$\\
$R_3=-c_2-c_3+x-y-2$\\
$R_4=-c_3+y-z-1$\\
\end{tabular}
\caption{Notation for $\sigma(\lambda+\rho)-\rho-\mu$ when $\sigma\in W\setminus (W_0\cup\{s_3s_1s_2\})$.}\label{tab:last 11}
\end{table}
For each of the elements in Table \ref{tab:last 11} we have that 
\begin{align*}
    1\in\mathcal{A}(\lam,\mu) &\Longleftrightarrow P_1,Q_1,R_1\in\N ,\\
    s_1\in\mathcal{A}(\lam,\mu) &\Longleftrightarrow P_4,Q_1,R_1\in\N ,\\
    s_2\in\mathcal{A}(\lam,\mu) &\Longleftrightarrow P_1,Q_6,R_1\in\N ,\\
    s_3\in\mathcal{A}(\lam,\mu) &\Longleftrightarrow P_1,Q_1,R_4\in\N ,\\
    s_1s_2\in\mathcal{A}(\lam,\mu) &\Longleftrightarrow P_3,Q_6,R_1\in\N ,\\
    s_2s_1\in\mathcal{A}(\lam,\mu) &\Longleftrightarrow P_4,Q_4,R_1\in\N ,\\
    s_3s_1\in\mathcal{A}(\lam,\mu) &\Longleftrightarrow P_4,Q_1,R_4\in\N ,\\
    s_2s_3\in\mathcal{A}(\lam,\mu) &\Longleftrightarrow P_1,Q_5,R_4\in\N ,\\
    s_3s_2\in\mathcal{A}(\lam,\mu) &\Longleftrightarrow P_1,Q_6,R_3\in\N ,\\
    s_1s_2s_1\in\mathcal{A}(\lam,\mu) &\Longleftrightarrow P_3,Q_4,R_1\in\N ,\\ 
    s_2s_3s_2\in\mathcal{A}(\lam,\mu) &\Longleftrightarrow P_1,Q_5,R_3\in\N .
\end{align*}
Intersecting these inequalities and applying the integrality conditions from Section \ref{subsec:Weyl alt sets}, we find that our Weyl alternation set for $\lam$ and $\mu$ are 
\begin{equation*}
    \mathcal{A}(\lam,\mu) =
    \begin{cases}
        {\{1\}} & R_1,P_1,Q_1\in\N  \text{ and }   Q_6,R_4,P_4\notin \N \text{ and}\\&(Q_5\notin\N \text{ or } R_3\notin\N) \text{ and } (P_3\notin\N \text{ or } Q_4\notin\N)\\
        {\{1,s_2\}} & R_1,Q_6,P_1,Q_1\in\N \text{ and }  P_3,R_4,R_3,P_4\notin\N\\
        {\{1,s_3\}} & R_1,P_1,R_4,Q_1\in\N \text{ and }  Q_6,Q_5,P_4\notin\N \text{ and }\\&(P_3\notin\N \text{ or } Q_4\notin\N)\\
        {\{1,s_1\}} & R_1,P_1,Q_1,P_4\in\N \text{ and }  Q_6,Q_4,R_4\notin\N \text{ and }\\& (Q_5\notin\N \text{ or } R_3\notin\N)\\
        {\{1,s_3,s_2\}} & {R_1,Q_6,P_1,R_4,Q_1\in\N \text{ and }  P_3,Q_5,R_3,P_4\notin\N}\\
        {\{1,s_2,s_1\}} & {R_1,Q_6,P_1,Q_1,P_4\in\N \text{ and }  P_3,Q_4,R_4,R_3\notin\N}\\
        {\{1,s_3,s_2s_3,s_2\}} & {R_1,Q_5,Q_1,R_4,Q_6,P_1\in\N \text{ and }  P_3,R_3,P_4\notin\N}\\
        {\{s_1,s_3s_1,s_3,1\}} & {R_1,P_1,R_4,Q_1,P_4\in\N \text{ and }  Q_6,Q_5,Q_4\notin\N}\\
        {\{s_1,s_2,1,s_2s_1\}} & {R_1,Q_1,P_4,Q_6,P_1,Q_4\in\N \text{ and }  P_3,R_4,R_3\notin\N}\\
        {\{s_1,s_3s_1,s_3,s_2,1\}} & {R_1,Q_1,P_4,R_4,Q_6,P_1\in\N \text{ and }  P_3,Q_5,Q_4,R_3\notin\N}\\
        {\{1,s_3,s_2s_3,s_2,s_3s_1,s_1\}} & {R_1,Q_5,Q_1,P_4,R_4,Q_6,P_1\in\N \text{ and }  P_3,Q_4,R_3\notin\N}\\
        {\{1,s_3s_1,s_2s_1,s_3,s_2,s_1\}} & {R_1,Q_1,P_4,R_4,Q_6,P_1,Q_4\in\N \text{ and }  P_3,Q_5,R_3\notin\N}\\
        {\{1,s_2s_3,s_3,s_2,s_2s_3s_2,s_3s_2\}} & {R_1,Q_5,Q_1,R_3,R_4,Q_6,P_1\in\N \text{ and }  P_3,P_4\notin\N}\\
        {\{1,s_1s_2s_1,s_2s_1,s_2,s_1,s_1s_2\}} & {R_1,Q_1,P_4,P_3,Q_6,P_1,Q_4\in\N \text{ and }  R_4,R_3\notin\N}.\\
    \end{cases}
\end{equation*}
    The result now follows from applying Theorem \ref{thm:main1} to evaluate the associated $q$-analog of Kostant's partition function for each element in the associated Weyl alternation set. In Appendix
    \ref{appendix:ztable} to describe concretely each of the polynomials $Z_1,\ldots, Z_{11}$.
\end{proof}
We end this section by applying Theorem \ref{thm:qKWMF} to compute a few $q$-weight multiplicities.
\begin{example}
Let $\lam=\a_1+\a_2+\a_3$ and $\mu=0$. Hence $m=1,n=0,k=1$, and $c_1=c_2=c_3=0$, from which 
we obtain $x=y=z=1$, and, hence, $\lambda=\omega_1+\omega_3.$ This means that $P_1=Q_1=R_1=1$, $P_3=-2$, $P_4=-1$, $Q_5=-2$, $Q_6=-2$, $R_4=-1.$ Since, $P_1, Q_1, R_1 \in \N$ and $P_3, P_4, Q_5, Q_6, R_4 \notin \N,$ then by Theorem \ref{thm:qKWMF} we have that
\begin{align*}
    m_q(\w_1+\w_3,0) &= Z_1 = \wp_q(P_1\al_1 + Q_1\al_2 + R_1\al_3) = \wp_q(\al_1+\al_2+\al_3).
\end{align*}
Applying the formula for $Z_1$ in Appendix \ref{appendix:ztable} we find that $t=2, L=0, 1\leq i\leq3.$ Using our closed formula in Theorem \ref{thm:main1}, part \eqref{thm:main1:part1} gives us the final result
\begin{align*}
    m_q(\w_1+\w_3,0)&=q^1+q^2+q^3.
\end{align*}
Evaluating at $q=1$ yields $m(\w_1+\w_3,0)=1+1+1=3$. We recall, that when $\lambda$ is the highest root, $m(\lambda,0)$ is the multiplicity of the zero-weight in the adjoint representation of $\sl4$. This equals the rank of $\sl4$ which is indeed 3.
\end{example}

\begin{example}
Let $\lam=\w_1+2\w_2+3\w_3$ and $\mu=\w_1+\w_2+\w_3$. Hence $m=1,n=2,k=3$, and $c_1=c_2=c_3=1$, from which 
we obtain $x=1,y=2,z=2$, and, hence, 
$P_1 = 1$, $Q_1 = 2$, $R_1 = 2$, $R_4 = -2$, $P_4 = -1$, $Q_6 = -1$, $P_3 = -4$, $Q_5 = -5.$
Since $P_1,Q_1,R_1\in\N$ and $R_4,P_4,Q_6,P_3,Q_5\notin\N$,
then by Theorem \ref{thm:qKWMF} we have that
\begin{align*}
    m_q(\w_1+2\w_2+3\w_3,\w_1+\w_2+\w_3) &= Z_1 = \wp_q(P_1\al_1 + Q_1\al_2 + R_1\al_3) = \wp_q(\al_1+2\al_2+2\al_3).
\end{align*}
Applying the formula for $Z_1$ in Appendix \ref{appendix:ztable} we find that $t=3$, $F_1=1$, $F_2=1$, $L=0$, and $\jmid=-1$.
Then, since $\jmid<0$, we find that the coefficient of $q^2$ is $c_2=1$.
Applying our closed formula in a similar fashion in the case where $i=2,3,4,5$ gives us the final result
\begin{align*}
    m_q(\w_1+2\w_2+3\w_3,\w_1+\w_2+\w_3) 
    &= q^2 + 3q^3 + 2q^4 + q^5.
\end{align*}
Evaluating at $q=1$ yields 
$   m(\w_1+2\w_2+3\w_3,\w_1+\w_2+\w_3) = 1+3+2+1
    = 7.
$
\end{example}

\begin{example}
Let  $\lam=\w_1+3\w_2$ and $\mu=\w_1+\w_2$.
Hence,
$P_1= 1$, $Q_1 = 2$, $R_1 = 1$, $R_4= 0$, $Q_6 = -2$, $P_4 = -1$, $Q_5= -3$, $P_3= -5.$
Note that $P_1,Q_1,R_1,R_4\in\N$ and $Q_6,P_4,Q_5,P_3\notin\N$. 
Thus, by Theorem \ref{thm:qKWMF}, we have that
\begin{align*}
    m_q(\w_1+3\w_2,\w_1+\w_2) &= Z_1 - Z_3 =
    \wp_q(\al_1 + 2\al_2 + \al_3) - \wp_q(\al_1 + 2\al_2).
\end{align*}
Applying the formula for $Z_1$ and $Z_3$ in Appendix \ref{appendix:ztable} we find that
$
    m_q(\w_1+3\w_2,\w_1+\w_2)=q^2+q^3+q^4.
$
Evaluating this polynomial at $q=1$ yields
$
    m(\w_1+3\w_2,\w_1+\w_2)=1+1+1= 3.
$
\end{example}

\section{Future Work}\label{sec:future}
As we showed in Section \ref{sec:qpartition}, for $i\in\N$ the set of partitions of weights as sums of exactly $i$ positive roots of the Lie algebra $\sl4$ are in bijection with certain subsets of $\{\one,\onep,2\}$-partitions of $i$. We believe that is just one occurrence of such a bijection. Hence we pose the following.

\begin{problem}
Characterize sets of restricted colored integer partitions that are in bijection with partitions of weights as sums of a certain number of positive roots of a Lie algebra.
\end{problem} 

An answer to this question would elucidate and strengthen the initial connection we have made between restricted colored integer partitions and the representation theory of Lie algebras. A connection which was asked about by Corteel and Lovejoy in \cite{CL}.

\section*{Acknowledgements}
This research was supported in part by the Alfred P. Sloan Foundation, the Mathematical Sciences Research Institute, and the National Science Foundation. The second author thanks Jesus De Loera for conversations that inspired this research.

\addresseshere

\newpage

\appendix

\section{Theorem \ref{thm:195 alt sets}: Weyl alternation sets of \texorpdfstring{$\sl4$}{sl4(C)}}\label{appendix:AS}

Table \ref{tab:tablewithelements} provides the reduced expression of $\sigma(\lambda+\rho)-\rho-\mu$ after the substitutions for the integrality condition in Lemma \ref{lem:integral} are applied.

\begin{table}[h]
\begin{tabular}{|l|p{1.5in}|l|}\hline
$\sigma\in W$& $\sigma(\lambda+\rho)-\rho-\mu$\\ \hline\hline
1 &
$P_1\a_1+Q_1\a_2+R_1\a_3$\\ \hline
$s_1$ &
$P_4\a_{1}+Q_1\a_{2}+R_1\a_3$\\ \hline
$s_2$ &
$P_1\a_{1}+Q_6\a_{2}+R_1\a_3$\\\hline
$s_3$ &
$P_1\a_{1}+Q_1\a_{2}+R_4\a_{3}$\\\hline
$s_1s_2$ &
$P_3\a_{1}+Q_6\a_{2}+R_1\a_{3}$\\\hline
$s_2s_1$ &
$P_4\a_{1}+Q_4\a_{2}+R_1\a_3$\\\hline
$s_2s_3$ &
$P_1\a_{1}+Q_5\a_{2}+R_4\a_3$\\\hline
$s_3s_1$&
$P_4\a_1+Q_1\a_{2}+R_4\a_3$\\\hline
$s_3s_2$ &
$P_1\a_{1}+Q_6\a_{2}+R_3\a_3$\\\hline
$s_1s_2s_1$&
$P_3\a_{1}+Q_4\a_{2}+R_1\a_{3}$\\\hline
$s_1s_2s_3$ &
$P_2\a_{1}+Q_5\a_{2}+R_4\a_3$\\\hline
$s_2s_3s_1$&
$P_4\a_{1}+Q_3\a_{2}+R_4\a_3$\\\hline
$s_2s_3s_2$ &
$P_1\a_{1}+Q_5\a_{2}+R_3\a_3$\\\hline
$s_3s_1s_2$ &
$P_3\a_{1}+Q_6\a_{2}+R_3\a_{3}$\\\hline
$s_3s_2s_1$ &
$P_4\a_1+Q_4\a_{2}+R_2\a_3$\\\hline
$s_1s_2s_3s_1$&
$P_2\a_1+Q_3\a_{2}+R_4\a_{3}$\\\hline
$s_1s_2s_3s_2$ &
$P_2\a_{1}+Q_5\a_{2}+R_3\a_{3}$\\\hline
$s_2s_3s_1s_2$&
$P_3\a_{1}+Q_2\a_{2}+R_3\a_3$\\\hline
$s_2s_3s_2s_1$ &
$P_4\a_{1}+Q_3\a_{2}+R_2\a_3$\\\hline
$s_3s_1s_2s_1$ &
$P_3\a_{1}+Q_4\a_{2}+R_2\a_{3}$\\\hline
$s_1s_2s_3s_1s_2$ &
$P_2\a_{1}+Q_2\a_{2}+R_3\a_3$\\\hline
$s_1s_2s_3s_2s_1$&
$P_2\a_1+Q_3\a_2+R_2\a_3$\\\hline
$s_2s_3s_1s_2s_1$ &
$P_3\a_{1}+Q_2\a_{2}+R_2\a_{3}$\\\hline
$s_1s_2s_3s_1s_2s_1$ &
$P_2\a_{1}+Q_2\a_{2}+R_2\a_3$\\\hline
\end{tabular}
\begin{tabular}{l}
$P_1=x$\\
$P_2=-c_1-c_2-c_3-z-3$\\
$P_3=-c_1-c_2-y+z-2$\\
$P_4=-c_1-x+y-1$\\
$Q_1=y$\\
$Q_2=-c_1-2c_2-c_3-y-4$\\
$Q_3=-c_1-c_2-c_3-x+y-z-3$\\
$Q_4=-c_1-c_2-x+z-2$\\
$Q_5=-c_2-c_3+x-z-2$\\
$Q_6=-c_2+x-y+z-1$\\
$R_1=z$\\
$R_2=-c_1-c_2-c_3-x-3$\\
$R_3=-c_2-c_3+x-y-2$\\
$R_4=-c_3+y-z-1$\\
\end{tabular}
\caption{Notation for $\sigma(\lambda+\rho)-\rho-\mu$ as a sum of simple roots for each $\sigma\in W$.}
\label{tab:tablewithelements}
\end{table}
Thus, the conditions Table \ref{tab:conditions} describe explicitly when a given element $\sigma\in W$ is in $\A(\lambda,\mu)$. 

\begin{table}[h]
    \centering
    \begin{tabular}{ll}
$K_1:$&$P_1=x\in\N$\\
$K_2:$&$P_2-c_1-c_2-c_3-z-3\in\N$\\
$K_3:$&$P_3=-c_1-c_2-y+z-2\in\N$\\
$K_4:$&$P_4=-c_1-x+y-1\in\N$\\
$K_5:$&$Q_1=y\in\N$\\
$K_6:$&$Q_2=-c_1-2c_2-c_3-y-4\in\N$\\
$K_7:$&$Q_3=-c_1-c_2-c_3-x+y-z-3\in\N$\\
$K_8:$&$Q_4=-c_1-c_2-x+z-2\in\N$\\
$K_9:$&$Q_5=-c_2-c_3+x-z-2\in\N$\\
$K_{10}:$&$Q_6=-c_2+x-y+z-1\in\N$\\
$K_{11}:$&$R_1=z\in\N$\\
$K_{12}:$&$R_2=-c_1-c_2-c_3-x-3\in\N$\\
$K_{13}:$&$R_3=-c_2-c_3+x-y-2\in\N$\\
$K_{14}:$&$R_4=-c_3+y-z-1\in\N$\\
\end{tabular}
    \caption{Conditions used in defining the Weyl alternation sets of $\sl4$.}
    \label{tab:conditions}
\end{table}

Using the conditions as listed in Table \ref{tab:conditions}, and letting $\neg$ denote  negation and $\vee$ denote or, then the Weyl alternation sets are as follows:
\begin{enumerate}[leftmargin=.7cm]
\item \altcirc{1} $\A(\lam,\mu)=\{1\}$ if $K_{11}$, $K_1$, $K_5$,$\neg K_{10}$, $\neg K_4$, $\neg K_{14}$, $\neg K_2$ $\vee$ $\neg K_7$ $\vee$ $\neg K_{12}$, $\neg K_2$ $\vee$ $\neg K_6$ $\vee$ $\neg K_{12}$, $\neg K_3$ $\vee$ $\neg K_6$ $\vee$ $\neg K_{12}$, $\neg K_3$ $\vee$ $\neg K_8$, $\neg K_9$ $\vee$ $\neg K_{13}$, $\neg K_2$ $\vee$ $\neg K_6$ $\vee$ $\neg K_{13}$, $\neg K_3$ $\vee$ $\neg K_6$ $\vee$ $\neg K_{13}$, 
\item \altcirc{2} $\A(\lam,\mu)=\{s_1\}$ if $K_{11}$, $K_5$, $K_4$,$\neg K_1$, $\neg K_{14}$, $\neg K_8$, $\neg K_2$ $\vee$ $\neg K_6$ $\vee$ $\neg K_{12}$, $\neg K_3$ $\vee$ $\neg K_6$ $\vee$ $\neg K_{12}$, $\neg K_3$ $\vee$ $\neg K_{10}$, $\neg K_{12}$ $\vee$ $\neg K_7$, $\neg K_2$ $\vee$ $\neg K_6$ $\vee$ $\neg K_{13}$, $\neg K_3$ $\vee$ $\neg K_6$ $\vee$ $\neg K_{13}$, $\neg K_2$ $\vee$ $\neg K_9$ $\vee$ $\neg K_{13}$, 
\item \altcirc{3} $\A(\lam,\mu)=\{s_2\}$ if $K_{11}$, $K_{10}$, $K_1$,$\neg K_5$, $\neg K_{13}$, $\neg K_3$, $\neg K_2$ $\vee$ $\neg K_7$ $\vee$ $\neg K_{12}$, $\neg K_2$ $\vee$ $\neg K_6$ $\vee$ $\neg K_{12}$, $\neg K_2$ $\vee$ $\neg K_7$ $\vee$ $\neg K_{14}$, $\neg K_9$ $\vee$ $\neg K_{14}$, $\neg K_{12}$ $\vee$ $\neg K_7$ $\vee$ $\neg K_4$, $\neg K_7$ $\vee$ $\neg K_{14}$ $\vee$ $\neg K_4$, $\neg K_8$ $\vee$ $\neg K_4$, 
\item \altcirc{4} $\A(\lam,\mu)=\{s_3\}$ if $K_1$, $K_{14}$, $K_5$,$\neg K_{11}$, $\neg K_9$, $\neg K_4$, $\neg K_2$ $\vee$ $\neg K_6$ $\vee$ $\neg K_{12}$, $\neg K_2$ $\vee$ $\neg K_7$, $\neg K_{10}$ $\vee$ $\neg K_{13}$, $\neg K_3$ $\vee$ $\neg K_6$ $\vee$ $\neg K_{12}$, $\neg K_3$ $\vee$ $\neg K_8$ $\vee$ $\neg K_{12}$, $\neg K_2$ $\vee$ $\neg K_6$ $\vee$ $\neg K_{13}$, $\neg K_3$ $\vee$ $\neg K_6$ $\vee$ $\neg K_{13}$, 
\item \altcirc{5} $\A(\lam,\mu)=\{s_1s_2\}$ if $K_3$, $K_{10}$, $K_{11}$,$\neg K_1$, $\neg K_8$, $\neg K_{13}$, $\neg K_2$ $\vee$ $\neg K_7$ $\vee$ $\neg K_{12}$, $\neg K_2$ $\vee$ $\neg K_7$ $\vee$ $\neg K_{14}$, $\neg K_5$ $\vee$ $\neg K_4$, $\neg K_6$ $\vee$ $\neg K_{12}$, $\neg K_2$ $\vee$ $\neg K_9$ $\vee$ $\neg K_{14}$, $\neg K_{12}$ $\vee$ $\neg K_7$ $\vee$ $\neg K_4$, $\neg K_7$ $\vee$ $\neg K_{14}$ $\vee$ $\neg K_4$, 
\item \altcirc{6} $\A(\lam,\mu)=\{s_2s_1\}$ if $K_{11}$, $K_8$, $K_4$,$\neg K_5$, $\neg K_3$, $\neg K_{12}$, $\neg K_{10}$ $\vee$ $\neg K_1$, $\neg K_1$ $\vee$ $\neg K_9$ $\vee$ $\neg K_{14}$, $\neg K_2$ $\vee$ $\neg K_9$ $\vee$ $\neg K_{14}$, $\neg K_1$ $\vee$ $\neg K_9$ $\vee$ $\neg K_{13}$, $\neg K_2$ $\vee$ $\neg K_6$ $\vee$ $\neg K_{13}$, $\neg K_7$ $\vee$ $\neg K_{14}$, $\neg K_2$ $\vee$ $\neg K_9$ $\vee$ $\neg K_{13}$, 
\item \altcirc{7} $\A(\lam,\mu)=\{s_2s_3\}$ if $K_1$, $K_9$, $K_{14}$,$\neg K_5$, $\neg K_2$, $\neg K_{13}$, $\neg K_{11}$ $\vee$ $\neg K_{10}$, $\neg K_3$ $\vee$ $\neg K_6$ $\vee$ $\neg K_{12}$, $\neg K_3$ $\vee$ $\neg K_{11}$ $\vee$ $\neg K_8$, $\neg K_8$ $\vee$ $\neg K_{12}$ $\vee$ $\neg K_4$, $\neg K_3$ $\vee$ $\neg K_8$ $\vee$ $\neg K_{12}$, $\neg K_7$ $\vee$ $\neg K_4$, $\neg K_{11}$ $\vee$ $\neg K_8$ $\vee$ $\neg K_4$, 
\item \altcirc{8} $\A(\lam,\mu)=\{s_3s_1\}$ if $K_{14}$, $K_5$, $K_4$,$\neg K_{11}$, $\neg K_1$, $\neg K_7$, $\neg K_2$ $\vee$ $\neg K_6$ $\vee$ $\neg K_{12}$, $\neg K_3$ $\vee$ $\neg K_6$ $\vee$ $\neg K_{12}$, $\neg K_8$ $\vee$ $\neg K_{12}$, $\neg K_2$ $\vee$ $\neg K_9$, $\neg K_2$ $\vee$ $\neg K_6$ $\vee$ $\neg K_{13}$, $\neg K_3$ $\vee$ $\neg K_{10}$ $\vee$ $\neg K_{13}$, $\neg K_3$ $\vee$ $\neg K_6$ $\vee$ $\neg K_{13}$, 
\item \altcirc{9} $\A(\lam,\mu)=\{s_3s_2\}$ if $K_{10}$, $K_1$, $K_{13}$,$\neg K_{11}$, $\neg K_9$, $\neg K_3$, $\neg K_2$ $\vee$ $\neg K_7$ $\vee$ $\neg K_{12}$, $\neg K_2$ $\vee$ $\neg K_7$ $\vee$ $\neg K_{14}$, $\neg K_8$ $\vee$ $\neg K_{12}$ $\vee$ $\neg K_4$, $\neg K_{14}$ $\vee$ $\neg K_5$, $\neg K_{12}$ $\vee$ $\neg K_7$ $\vee$ $\neg K_4$, $\neg K_2$ $\vee$ $\neg K_6$, $\neg K_7$ $\vee$ $\neg K_{14}$ $\vee$ $\neg K_4$, 
\item \altcirc{10} $\A(\lam,\mu)=\{s_1s_2s_1\}$ if $K_3$, $K_{11}$, $K_8$,$\neg K_{10}$, $\neg K_{12}$, $\neg K_4$, $\neg K_1$ $\vee$ $\neg K_5$, $\neg K_2$ $\vee$ $\neg K_7$ $\vee$ $\neg K_{14}$, $\neg K_1$ $\vee$ $\neg K_9$ $\vee$ $\neg K_{14}$, $\neg K_2$ $\vee$ $\neg K_9$ $\vee$ $\neg K_{14}$, $\neg K_1$ $\vee$ $\neg K_9$ $\vee$ $\neg K_{13}$, $\neg K_6$ $\vee$ $\neg K_{13}$, $\neg K_2$ $\vee$ $\neg K_9$ $\vee$ $\neg K_{13}$, 
\item \altcirc{11} $\A(\lam,\mu)=\{s_1s_2s_3\}$ if $K_2$, $K_9$, $K_{14}$,$\neg K_7$, $\neg K_1$, $\neg K_{13}$, $\neg K_6$ $\vee$ $\neg K_{12}$, $\neg K_3$ $\vee$ $\neg K_{11}$ $\vee$ $\neg K_8$, $\neg K_8$ $\vee$ $\neg K_{12}$ $\vee$ $\neg K_4$, $\neg K_3$ $\vee$ $\neg K_{10}$ $\vee$ $\neg K_{11}$, $\neg K_3$ $\vee$ $\neg K_8$ $\vee$ $\neg K_{12}$, $\neg K_5$ $\vee$ $\neg K_4$, $\neg K_{11}$ $\vee$ $\neg K_8$ $\vee$ $\neg K_4$, 
\item \altcirc{12} $\A(\lam,\mu)=\{s_2s_3s_1\}$ if $K_7$, $K_{14}$, $K_4$,$\neg K_2$, $\neg K_{12}$, $\neg K_5$, $\neg K_{11}$ $\vee$ $\neg K_{10}$ $\vee$ $\neg K_1$, $\neg K_{10}$ $\vee$ $\neg K_1$ $\vee$ $\neg K_{13}$, $\neg K_1$ $\vee$ $\neg K_9$, $\neg K_3$ $\vee$ $\neg K_{10}$ $\vee$ $\neg K_{11}$, $\neg K_3$ $\vee$ $\neg K_{10}$ $\vee$ $\neg K_{13}$, $\neg K_3$ $\vee$ $\neg K_6$ $\vee$ $\neg K_{13}$, $\neg K_{11}$ $\vee$ $\neg K_8$, 
\item \altcirc{13} $\A(\lam,\mu)=\{s_2s_3s_2\}$ if $K_1$, $K_9$, $K_{13}$,$\neg K_{10}$, $\neg K_{14}$, $\neg K_2$, $\neg K_{11}$ $\vee$ $\neg K_5$, $\neg K_3$ $\vee$ $\neg K_{11}$ $\vee$ $\neg K_8$, $\neg K_8$ $\vee$ $\neg K_{12}$ $\vee$ $\neg K_4$, $\neg K_{12}$ $\vee$ $\neg K_7$ $\vee$ $\neg K_4$, $\neg K_3$ $\vee$ $\neg K_8$ $\vee$ $\neg K_{12}$, $\neg K_3$ $\vee$ $\neg K_6$, $\neg K_{11}$ $\vee$ $\neg K_8$ $\vee$ $\neg K_4$, 
\item \altcirc{14} $\A(\lam,\mu)=\{s_3s_1s_2\}$ if $K_3$, $K_{10}$, $K_{13}$,$\neg K_1$, $\neg K_{11}$, $\neg K_6$, $\neg K_2$ $\vee$ $\neg K_7$ $\vee$ $\neg K_{12}$, $\neg K_2$ $\vee$ $\neg K_7$ $\vee$ $\neg K_{14}$, $\neg K_{12}$ $\vee$ $\neg K_7$ $\vee$ $\neg K_4$, $\neg K_8$ $\vee$ $\neg K_{12}$, $\neg K_{14}$ $\vee$ $\neg K_5$ $\vee$ $\neg K_4$, $\neg K_7$ $\vee$ $\neg K_{14}$ $\vee$ $\neg K_4$, $\neg K_2$ $\vee$ $\neg K_9$, 
\item \altcirc{15} $\A(\lam,\mu)=\{s_3s_2s_1\}$ if $K_8$, $K_4$, $K_{12}$,$\neg K_7$, $\neg K_3$, $\neg K_{11}$, $\neg K_2$ $\vee$ $\neg K_6$, $\neg K_{10}$ $\vee$ $\neg K_1$ $\vee$ $\neg K_{13}$, $\neg K_1$ $\vee$ $\neg K_9$ $\vee$ $\neg K_{14}$, $\neg K_2$ $\vee$ $\neg K_9$ $\vee$ $\neg K_{14}$, $\neg K_1$ $\vee$ $\neg K_9$ $\vee$ $\neg K_{13}$, $\neg K_{14}$ $\vee$ $\neg K_5$, $\neg K_2$ $\vee$ $\neg K_9$ $\vee$ $\neg K_{13}$, 
\item \altcirc{16} $\A(\lam,\mu)=\{s_1s_2s_3s_1\}$ if $K_2$, $K_7$, $K_{14}$,$\neg K_{12}$, $\neg K_9$, $\neg K_4$, $\neg K_{11}$ $\vee$ $\neg K_{10}$ $\vee$ $\neg K_1$, $\neg K_{10}$ $\vee$ $\neg K_1$ $\vee$ $\neg K_{13}$, $\neg K_3$ $\vee$ $\neg K_{11}$ $\vee$ $\neg K_8$, $\neg K_1$ $\vee$ $\neg K_5$, $\neg K_3$ $\vee$ $\neg K_{10}$ $\vee$ $\neg K_{11}$, $\neg K_6$ $\vee$ $\neg K_{13}$, $\neg K_3$ $\vee$ $\neg K_{10}$ $\vee$ $\neg K_{13}$, 
\item \altcirc{17} $\A(\lam,\mu)=\{s_1s_2s_3s_2\}$ if $K_2$, $K_9$, $K_{13}$,$\neg K_{14}$, $\neg K_1$, $\neg K_6$, $\neg K_7$ $\vee$ $\neg K_{12}$, $\neg K_{11}$ $\vee$ $\neg K_5$ $\vee$ $\neg K_4$, $\neg K_3$ $\vee$ $\neg K_{11}$ $\vee$ $\neg K_8$, $\neg K_8$ $\vee$ $\neg K_{12}$ $\vee$ $\neg K_4$, $\neg K_3$ $\vee$ $\neg K_8$ $\vee$ $\neg K_{12}$, $\neg K_3$ $\vee$ $\neg K_{10}$, $\neg K_{11}$ $\vee$ $\neg K_8$ $\vee$ $\neg K_4$, 
\item \altcirc{18} $\A(\lam,\mu)=\{s_2s_3s_1s_2\}$ if $K_3$, $K_6$, $K_{13}$,$\neg K_{12}$, $\neg K_2$, $\neg K_{10}$, $\neg K_{11}$ $\vee$ $\neg K_1$ $\vee$ $\neg K_5$, $\neg K_{11}$ $\vee$ $\neg K_5$ $\vee$ $\neg K_4$, $\neg K_{11}$ $\vee$ $\neg K_8$, $\neg K_1$ $\vee$ $\neg K_{14}$ $\vee$ $\neg K_5$, $\neg K_1$ $\vee$ $\neg K_9$, $\neg K_{14}$ $\vee$ $\neg K_5$ $\vee$ $\neg K_4$, $\neg K_7$ $\vee$ $\neg K_{14}$ $\vee$ $\neg K_4$, 
\item \altcirc{19} $\A(\lam,\mu)=\{s_2s_3s_2s_1\}$ if $K_4$, $K_7$, $K_{12}$,$\neg K_2$, $\neg K_8$, $\neg K_{14}$, $\neg K_{11}$ $\vee$ $\neg K_{10}$ $\vee$ $\neg K_1$, $\neg K_{11}$ $\vee$ $\neg K_5$, $\neg K_{10}$ $\vee$ $\neg K_1$ $\vee$ $\neg K_{13}$, $\neg K_3$ $\vee$ $\neg K_6$, $\neg K_1$ $\vee$ $\neg K_9$ $\vee$ $\neg K_{13}$, $\neg K_3$ $\vee$ $\neg K_{10}$ $\vee$ $\neg K_{11}$, $\neg K_3$ $\vee$ $\neg K_{10}$ $\vee$ $\neg K_{13}$, 
\item \altcirc{20} $\A(\lam,\mu)=\{s_3s_1s_2s_1\}$ if $K_3$, $K_8$, $K_{12}$,$\neg K_6$, $\neg K_{11}$, $\neg K_4$, $\neg K_2$ $\vee$ $\neg K_7$, $\neg K_1$ $\vee$ $\neg K_9$ $\vee$ $\neg K_{14}$, $\neg K_1$ $\vee$ $\neg K_{14}$ $\vee$ $\neg K_5$, $\neg K_2$ $\vee$ $\neg K_9$ $\vee$ $\neg K_{14}$, $\neg K_1$ $\vee$ $\neg K_9$ $\vee$ $\neg K_{13}$, $\neg K_{10}$ $\vee$ $\neg K_{13}$, $\neg K_2$ $\vee$ $\neg K_9$ $\vee$ $\neg K_{13}$, 
\item \altcirc{21} $\A(\lam,\mu)=\{s_1s_2s_3s_1s_2\}$ if $K_2$, $K_6$, $K_{13}$,$\neg K_{12}$, $\neg K_3$, $\neg K_9$, $\neg K_{11}$ $\vee$ $\neg K_1$ $\vee$ $\neg K_5$, $\neg K_7$ $\vee$ $\neg K_{14}$, $\neg K_{11}$ $\vee$ $\neg K_5$ $\vee$ $\neg K_4$, $\neg K_{10}$ $\vee$ $\neg K_1$, $\neg K_1$ $\vee$ $\neg K_{14}$ $\vee$ $\neg K_5$, $\neg K_{14}$ $\vee$ $\neg K_5$ $\vee$ $\neg K_4$, $\neg K_{11}$ $\vee$ $\neg K_8$ $\vee$ $\neg K_4$, 
\item \altcirc{22} $\A(\lam,\mu)=\{s_1s_2s_3s_2s_1\}$ if $K_2$, $K_7$, $K_{12}$,$\neg K_6$, $\neg K_{14}$, $\neg K_4$, $\neg K_{11}$ $\vee$ $\neg K_1$ $\vee$ $\neg K_5$, $\neg K_{11}$ $\vee$ $\neg K_{10}$ $\vee$ $\neg K_1$, $\neg K_{10}$ $\vee$ $\neg K_1$ $\vee$ $\neg K_{13}$, $\neg K_3$ $\vee$ $\neg K_{10}$ $\vee$ $\neg K_{11}$, $\neg K_3$ $\vee$ $\neg K_8$, $\neg K_3$ $\vee$ $\neg K_{10}$ $\vee$ $\neg K_{13}$, $\neg K_9$ $\vee$ $\neg K_{13}$, 
\item \altcirc{23} $\A(\lam,\mu)=\{s_2s_3s_1s_2s_1\}$ if $K_3$, $K_6$, $K_{12}$,$\neg K_2$, $\neg K_8$, $\neg K_{13}$, $\neg K_{11}$ $\vee$ $\neg K_1$ $\vee$ $\neg K_5$, $\neg K_{11}$ $\vee$ $\neg K_5$ $\vee$ $\neg K_4$, $\neg K_1$ $\vee$ $\neg K_9$ $\vee$ $\neg K_{14}$, $\neg K_1$ $\vee$ $\neg K_{14}$ $\vee$ $\neg K_5$, $\neg K_{10}$ $\vee$ $\neg K_{11}$, $\neg K_7$ $\vee$ $\neg K_4$, $\neg K_{14}$ $\vee$ $\neg K_5$ $\vee$ $\neg K_4$, 
\item \altcirc{24} $\A(\lam,\mu)=\{s_1s_2s_3s_1s_2s_1\}$ if $K_2$, $K_6$, $K_{12}$,$\neg K_7$, $\neg K_3$, $\neg K_{13}$, $\neg K_{11}$ $\vee$ $\neg K_1$ $\vee$ $\neg K_5$, $\neg K_{11}$ $\vee$ $\neg K_{10}$ $\vee$ $\neg K_1$, $\neg K_{11}$ $\vee$ $\neg K_5$ $\vee$ $\neg K_4$, $\neg K_8$ $\vee$ $\neg K_4$, $\neg K_1$ $\vee$ $\neg K_{14}$ $\vee$ $\neg K_5$, $\neg K_9$ $\vee$ $\neg K_{14}$, $\neg K_{14}$ $\vee$ $\neg K_5$ $\vee$ $\neg K_4$, 
\item \altcirc{25} $\A(\lam,\mu)=\{1,s_1\}$ if $K_{11}$, $K_1$, $K_5$, $K_4$,$\neg K_{10}$, $\neg K_{14}$, $\neg K_8$, $\neg K_2$ $\vee$ $\neg K_6$ $\vee$ $\neg K_{12}$, $\neg K_3$ $\vee$ $\neg K_6$ $\vee$ $\neg K_{12}$, $\neg K_9$ $\vee$ $\neg K_{13}$, $\neg K_{12}$ $\vee$ $\neg K_7$, $\neg K_2$ $\vee$ $\neg K_6$ $\vee$ $\neg K_{13}$, $\neg K_3$ $\vee$ $\neg K_6$ $\vee$ $\neg K_{13}$, 
\item \altcirc{26} $\A(\lam,\mu)=\{1,s_2\}$ if $K_{11}$, $K_{10}$, $K_1$, $K_5$,$\neg K_4$, $\neg K_{13}$, $\neg K_{14}$, $\neg K_3$, $\neg K_2$ $\vee$ $\neg K_7$ $\vee$ $\neg K_{12}$, $\neg K_2$ $\vee$ $\neg K_6$ $\vee$ $\neg K_{12}$, 
\item \altcirc{27} $\A(\lam,\mu)=\{1,s_3\}$ if $K_{11}$, $K_1$, $K_{14}$, $K_5$,$\neg K_{10}$, $\neg K_4$, $\neg K_9$, $\neg K_2$ $\vee$ $\neg K_6$ $\vee$ $\neg K_{12}$, $\neg K_2$ $\vee$ $\neg K_7$, $\neg K_3$ $\vee$ $\neg K_6$ $\vee$ $\neg K_{12}$, $\neg K_3$ $\vee$ $\neg K_8$, $\neg K_2$ $\vee$ $\neg K_6$ $\vee$ $\neg K_{13}$, $\neg K_3$ $\vee$ $\neg K_6$ $\vee$ $\neg K_{13}$, 
\item \altcirc{28} $\A(\lam,\mu)=\{s_1,s_2s_1\}$ if $K_{11}$, $K_8$, $K_5$, $K_4$,$\neg K_1$, $\neg K_3$, $\neg K_{12}$, $\neg K_{14}$, $\neg K_2$ $\vee$ $\neg K_6$ $\vee$ $\neg K_{13}$, $\neg K_2$ $\vee$ $\neg K_9$ $\vee$ $\neg K_{13}$, 
\item \altcirc{29} $\A(\lam,\mu)=\{s_1,s_3s_1\}$ if $K_{11}$, $K_{14}$, $K_5$, $K_4$,$\neg K_1$, $\neg K_7$, $\neg K_8$, $\neg K_2$ $\vee$ $\neg K_6$ $\vee$ $\neg K_{12}$, $\neg K_3$ $\vee$ $\neg K_6$ $\vee$ $\neg K_{12}$, $\neg K_2$ $\vee$ $\neg K_9$, $\neg K_3$ $\vee$ $\neg K_{10}$, $\neg K_2$ $\vee$ $\neg K_6$ $\vee$ $\neg K_{13}$, $\neg K_3$ $\vee$ $\neg K_6$ $\vee$ $\neg K_{13}$, 
\item \altcirc{30} $\A(\lam,\mu)=\{s_2,s_1s_2\}$ if $K_{11}$, $K_{10}$, $K_1$, $K_3$,$\neg K_5$, $\neg K_{13}$, $\neg K_8$, $\neg K_2$ $\vee$ $\neg K_7$ $\vee$ $\neg K_{12}$, $\neg K_2$ $\vee$ $\neg K_7$ $\vee$ $\neg K_{14}$, $\neg K_6$ $\vee$ $\neg K_{12}$, $\neg K_9$ $\vee$ $\neg K_{14}$, $\neg K_{12}$ $\vee$ $\neg K_7$ $\vee$ $\neg K_4$, $\neg K_7$ $\vee$ $\neg K_{14}$ $\vee$ $\neg K_4$, 
\item \altcirc{31} $\A(\lam,\mu)=\{s_2,s_3s_2\}$ if $K_{11}$, $K_{10}$, $K_1$, $K_{13}$,$\neg K_5$, $\neg K_9$, $\neg K_3$, $\neg K_2$ $\vee$ $\neg K_7$ $\vee$ $\neg K_{12}$, $\neg K_2$ $\vee$ $\neg K_7$ $\vee$ $\neg K_{14}$, $\neg K_{12}$ $\vee$ $\neg K_7$ $\vee$ $\neg K_4$, $\neg K_2$ $\vee$ $\neg K_6$, $\neg K_7$ $\vee$ $\neg K_{14}$ $\vee$ $\neg K_4$, $\neg K_8$ $\vee$ $\neg K_4$, 
\item \altcirc{32} $\A(\lam,\mu)=\{s_3,s_2s_3\}$ if $K_1$, $K_9$, $K_5$, $K_{14}$,$\neg K_{11}$, $\neg K_2$, $\neg K_{13}$, $\neg K_4$, $\neg K_3$ $\vee$ $\neg K_6$ $\vee$ $\neg K_{12}$, $\neg K_3$ $\vee$ $\neg K_8$ $\vee$ $\neg K_{12}$, 
\item \altcirc{33} $\A(\lam,\mu)=\{s_3,s_3s_1\}$ if $K_1$, $K_{14}$, $K_5$, $K_4$,$\neg K_{11}$, $\neg K_9$, $\neg K_7$, $\neg K_2$ $\vee$ $\neg K_6$ $\vee$ $\neg K_{12}$, $\neg K_{10}$ $\vee$ $\neg K_{13}$, $\neg K_3$ $\vee$ $\neg K_6$ $\vee$ $\neg K_{12}$, $\neg K_8$ $\vee$ $\neg K_{12}$, $\neg K_2$ $\vee$ $\neg K_6$ $\vee$ $\neg K_{13}$, $\neg K_3$ $\vee$ $\neg K_6$ $\vee$ $\neg K_{13}$, 
\item \altcirc{34} $\A(\lam,\mu)=\{s_1s_2,s_1s_2s_1\}$ if $K_3$, $K_{11}$, $K_8$, $K_{10}$,$\neg K_1$, $\neg K_{12}$, $\neg K_{13}$, $\neg K_4$, $\neg K_2$ $\vee$ $\neg K_7$ $\vee$ $\neg K_{14}$, $\neg K_2$ $\vee$ $\neg K_9$ $\vee$ $\neg K_{14}$, 
\item \altcirc{35} $\A(\lam,\mu)=\{s_1s_2,s_3s_1s_2\}$ if $K_3$, $K_{10}$, $K_{13}$, $K_{11}$,$\neg K_1$, $\neg K_8$, $\neg K_6$, $\neg K_2$ $\vee$ $\neg K_7$ $\vee$ $\neg K_{12}$, $\neg K_2$ $\vee$ $\neg K_7$ $\vee$ $\neg K_{14}$, $\neg K_5$ $\vee$ $\neg K_4$, $\neg K_{12}$ $\vee$ $\neg K_7$ $\vee$ $\neg K_4$, $\neg K_7$ $\vee$ $\neg K_{14}$ $\vee$ $\neg K_4$, $\neg K_2$ $\vee$ $\neg K_9$, 
\item \altcirc{36} $\A(\lam,\mu)=\{s_2s_1,s_1s_2s_1\}$ if $K_3$, $K_{11}$, $K_8$, $K_4$,$\neg K_5$, $\neg K_{12}$, $\neg K_{10}$, $\neg K_1$ $\vee$ $\neg K_9$ $\vee$ $\neg K_{14}$, $\neg K_2$ $\vee$ $\neg K_9$ $\vee$ $\neg K_{14}$, $\neg K_1$ $\vee$ $\neg K_9$ $\vee$ $\neg K_{13}$, $\neg K_7$ $\vee$ $\neg K_{14}$, $\neg K_6$ $\vee$ $\neg K_{13}$, $\neg K_2$ $\vee$ $\neg K_9$ $\vee$ $\neg K_{13}$, 
\item \altcirc{37} $\A(\lam,\mu)=\{s_2s_1,s_3s_2s_1\}$ if $K_{11}$, $K_8$, $K_4$, $K_{12}$,$\neg K_5$, $\neg K_3$, $\neg K_7$, $\neg K_{10}$ $\vee$ $\neg K_1$, $\neg K_2$ $\vee$ $\neg K_6$, $\neg K_1$ $\vee$ $\neg K_9$ $\vee$ $\neg K_{14}$, $\neg K_2$ $\vee$ $\neg K_9$ $\vee$ $\neg K_{14}$, $\neg K_1$ $\vee$ $\neg K_9$ $\vee$ $\neg K_{13}$, $\neg K_2$ $\vee$ $\neg K_9$ $\vee$ $\neg K_{13}$, 
\item \altcirc{38} $\A(\lam,\mu)=\{s_2s_3,s_1s_2s_3\}$ if $K_2$, $K_1$, $K_9$, $K_{14}$,$\neg K_7$, $\neg K_5$, $\neg K_{13}$, $\neg K_{11}$ $\vee$ $\neg K_{10}$, $\neg K_6$ $\vee$ $\neg K_{12}$, $\neg K_3$ $\vee$ $\neg K_{11}$ $\vee$ $\neg K_8$, $\neg K_8$ $\vee$ $\neg K_{12}$ $\vee$ $\neg K_4$, $\neg K_3$ $\vee$ $\neg K_8$ $\vee$ $\neg K_{12}$, $\neg K_{11}$ $\vee$ $\neg K_8$ $\vee$ $\neg K_4$, 
\item \altcirc{39} $\A(\lam,\mu)=\{s_2s_3,s_2s_3s_2\}$ if $K_1$, $K_9$, $K_{13}$, $K_{14}$,$\neg K_{10}$, $\neg K_5$, $\neg K_2$, $\neg K_3$ $\vee$ $\neg K_{11}$ $\vee$ $\neg K_8$, $\neg K_8$ $\vee$ $\neg K_{12}$ $\vee$ $\neg K_4$, $\neg K_3$ $\vee$ $\neg K_8$ $\vee$ $\neg K_{12}$, $\neg K_7$ $\vee$ $\neg K_4$, $\neg K_3$ $\vee$ $\neg K_6$, $\neg K_{11}$ $\vee$ $\neg K_8$ $\vee$ $\neg K_4$, 
\item \altcirc{40} $\A(\lam,\mu)=\{s_3s_1,s_2s_3s_1\}$ if $K_7$, $K_{14}$, $K_5$, $K_4$,$\neg K_2$, $\neg K_{11}$, $\neg K_1$, $\neg K_{12}$, $\neg K_3$ $\vee$ $\neg K_{10}$ $\vee$ $\neg K_{13}$, $\neg K_3$ $\vee$ $\neg K_6$ $\vee$ $\neg K_{13}$, 
\item \altcirc{41} $\A(\lam,\mu)=\{s_3s_2,s_2s_3s_2\}$ if $K_{10}$, $K_1$, $K_9$, $K_{13}$,$\neg K_{11}$, $\neg K_{14}$, $\neg K_3$, $\neg K_2$, $\neg K_8$ $\vee$ $\neg K_{12}$ $\vee$ $\neg K_4$, $\neg K_{12}$ $\vee$ $\neg K_7$ $\vee$ $\neg K_4$, 
\item \altcirc{42} $\A(\lam,\mu)=\{s_3s_2,s_3s_1s_2\}$ if $K_3$, $K_{10}$, $K_1$, $K_{13}$,$\neg K_{11}$, $\neg K_9$, $\neg K_6$, $\neg K_2$ $\vee$ $\neg K_7$ $\vee$ $\neg K_{12}$, $\neg K_2$ $\vee$ $\neg K_7$ $\vee$ $\neg K_{14}$, $\neg K_{14}$ $\vee$ $\neg K_5$, $\neg K_{12}$ $\vee$ $\neg K_7$ $\vee$ $\neg K_4$, $\neg K_8$ $\vee$ $\neg K_{12}$, $\neg K_7$ $\vee$ $\neg K_{14}$ $\vee$ $\neg K_4$, 
\item \altcirc{43} $\A(\lam,\mu)=\{s_1s_2s_1,s_3s_1s_2s_1\}$ if $K_3$, $K_{11}$, $K_8$, $K_{12}$,$\neg K_6$, $\neg K_4$, $\neg K_{10}$, $\neg K_1$ $\vee$ $\neg K_5$, $\neg K_2$ $\vee$ $\neg K_7$, $\neg K_1$ $\vee$ $\neg K_9$ $\vee$ $\neg K_{14}$, $\neg K_2$ $\vee$ $\neg K_9$ $\vee$ $\neg K_{14}$, $\neg K_1$ $\vee$ $\neg K_9$ $\vee$ $\neg K_{13}$, $\neg K_2$ $\vee$ $\neg K_9$ $\vee$ $\neg K_{13}$, 
\item \altcirc{44} $\A(\lam,\mu)=\{s_1s_2s_3,s_1s_2s_3s_1\}$ if $K_2$, $K_9$, $K_7$, $K_{14}$,$\neg K_{12}$, $\neg K_1$, $\neg K_4$, $\neg K_{13}$, $\neg K_3$ $\vee$ $\neg K_{11}$ $\vee$ $\neg K_8$, $\neg K_3$ $\vee$ $\neg K_{10}$ $\vee$ $\neg K_{11}$, 
\item \altcirc{45} $\A(\lam,\mu)=\{s_1s_2s_3,s_1s_2s_3s_2\}$ if $K_2$, $K_9$, $K_{14}$, $K_{13}$,$\neg K_7$, $\neg K_1$, $\neg K_6$, $\neg K_3$ $\vee$ $\neg K_{11}$ $\vee$ $\neg K_8$, $\neg K_8$ $\vee$ $\neg K_{12}$ $\vee$ $\neg K_4$, $\neg K_3$ $\vee$ $\neg K_8$ $\vee$ $\neg K_{12}$, $\neg K_5$ $\vee$ $\neg K_4$, $\neg K_3$ $\vee$ $\neg K_{10}$, $\neg K_{11}$ $\vee$ $\neg K_8$ $\vee$ $\neg K_4$, 
\item \altcirc{46} $\A(\lam,\mu)=\{s_2s_3s_1,s_1s_2s_3s_1\}$ if $K_2$, $K_7$, $K_{14}$, $K_4$,$\neg K_{12}$, $\neg K_9$, $\neg K_5$, $\neg K_{11}$ $\vee$ $\neg K_{10}$ $\vee$ $\neg K_1$, $\neg K_{10}$ $\vee$ $\neg K_1$ $\vee$ $\neg K_{13}$, $\neg K_3$ $\vee$ $\neg K_{10}$ $\vee$ $\neg K_{11}$, $\neg K_6$ $\vee$ $\neg K_{13}$, $\neg K_3$ $\vee$ $\neg K_{10}$ $\vee$ $\neg K_{13}$, $\neg K_{11}$ $\vee$ $\neg K_8$, 
\item \altcirc{47} $\A(\lam,\mu)=\{s_2s_3s_1,s_2s_3s_2s_1\}$ if $K_4$, $K_7$, $K_{14}$, $K_{12}$,$\neg K_2$, $\neg K_8$, $\neg K_5$, $\neg K_{11}$ $\vee$ $\neg K_{10}$ $\vee$ $\neg K_1$, $\neg K_{10}$ $\vee$ $\neg K_1$ $\vee$ $\neg K_{13}$, $\neg K_3$ $\vee$ $\neg K_6$, $\neg K_1$ $\vee$ $\neg K_9$, $\neg K_3$ $\vee$ $\neg K_{10}$ $\vee$ $\neg K_{11}$, $\neg K_3$ $\vee$ $\neg K_{10}$ $\vee$ $\neg K_{13}$, 
\item \altcirc{48} $\A(\lam,\mu)=\{s_2s_3s_2,s_1s_2s_3s_2\}$ if $K_2$, $K_1$, $K_9$, $K_{13}$,$\neg K_{10}$, $\neg K_{14}$, $\neg K_6$, $\neg K_{11}$ $\vee$ $\neg K_5$, $\neg K_7$ $\vee$ $\neg K_{12}$, $\neg K_3$ $\vee$ $\neg K_{11}$ $\vee$ $\neg K_8$, $\neg K_8$ $\vee$ $\neg K_{12}$ $\vee$ $\neg K_4$, $\neg K_3$ $\vee$ $\neg K_8$ $\vee$ $\neg K_{12}$, $\neg K_{11}$ $\vee$ $\neg K_8$ $\vee$ $\neg K_4$, 
\item \altcirc{49} $\A(\lam,\mu)=\{s_3s_1s_2,s_2s_3s_1s_2\}$ if $K_3$, $K_{10}$, $K_6$, $K_{13}$,$\neg K_1$, $\neg K_{12}$, $\neg K_{11}$, $\neg K_2$, $\neg K_{14}$ $\vee$ $\neg K_5$ $\vee$ $\neg K_4$, $\neg K_7$ $\vee$ $\neg K_{14}$ $\vee$ $\neg K_4$, 
\item \altcirc{50} $\A(\lam,\mu)=\{s_3s_2s_1,s_2s_3s_2s_1\}$ if $K_7$, $K_8$, $K_4$, $K_{12}$,$\neg K_2$, $\neg K_3$, $\neg K_{14}$, $\neg K_{11}$, $\neg K_{10}$ $\vee$ $\neg K_1$ $\vee$ $\neg K_{13}$, $\neg K_1$ $\vee$ $\neg K_9$ $\vee$ $\neg K_{13}$, 
\item \altcirc{51} $\A(\lam,\mu)=\{s_3s_2s_1,s_3s_1s_2s_1\}$ if $K_3$, $K_8$, $K_4$, $K_{12}$,$\neg K_6$, $\neg K_{11}$, $\neg K_7$, $\neg K_1$ $\vee$ $\neg K_9$ $\vee$ $\neg K_{14}$, $\neg K_2$ $\vee$ $\neg K_9$ $\vee$ $\neg K_{14}$, $\neg K_1$ $\vee$ $\neg K_9$ $\vee$ $\neg K_{13}$, $\neg K_{14}$ $\vee$ $\neg K_5$, $\neg K_{10}$ $\vee$ $\neg K_{13}$, $\neg K_2$ $\vee$ $\neg K_9$ $\vee$ $\neg K_{13}$, 
\item \altcirc{52} $\A(\lam,\mu)=\{s_1s_2s_3s_1,s_1s_2s_3s_2s_1\}$ if $K_2$, $K_7$, $K_{14}$, $K_{12}$,$\neg K_6$, $\neg K_9$, $\neg K_4$, $\neg K_{11}$ $\vee$ $\neg K_{10}$ $\vee$ $\neg K_1$, $\neg K_{10}$ $\vee$ $\neg K_1$ $\vee$ $\neg K_{13}$, $\neg K_1$ $\vee$ $\neg K_5$, $\neg K_3$ $\vee$ $\neg K_{10}$ $\vee$ $\neg K_{11}$, $\neg K_3$ $\vee$ $\neg K_8$, $\neg K_3$ $\vee$ $\neg K_{10}$ $\vee$ $\neg K_{13}$, 
\item \altcirc{53} $\A(\lam,\mu)=\{s_1s_2s_3s_2,s_1s_2s_3s_1s_2\}$ if $K_2$, $K_9$, $K_6$, $K_{13}$,$\neg K_{12}$, $\neg K_{14}$, $\neg K_1$, $\neg K_3$, $\neg K_{11}$ $\vee$ $\neg K_5$ $\vee$ $\neg K_4$, $\neg K_{11}$ $\vee$ $\neg K_8$ $\vee$ $\neg K_4$, 
\item \altcirc{54} $\A(\lam,\mu)=\{s_2s_3s_1s_2,s_1s_2s_3s_1s_2\}$ if $K_3$, $K_2$, $K_6$, $K_{13}$,$\neg K_{12}$, $\neg K_{10}$, $\neg K_9$, $\neg K_{11}$ $\vee$ $\neg K_1$ $\vee$ $\neg K_5$, $\neg K_7$ $\vee$ $\neg K_{14}$, $\neg K_{11}$ $\vee$ $\neg K_5$ $\vee$ $\neg K_4$, $\neg K_{11}$ $\vee$ $\neg K_8$, $\neg K_1$ $\vee$ $\neg K_{14}$ $\vee$ $\neg K_5$, $\neg K_{14}$ $\vee$ $\neg K_5$ $\vee$ $\neg K_4$, 
\item \altcirc{55} $\A(\lam,\mu)=\{s_2s_3s_1s_2,s_2s_3s_1s_2s_1\}$ if $K_3$, $K_6$, $K_{13}$, $K_{12}$,$\neg K_2$, $\neg K_8$, $\neg K_{10}$, $\neg K_{11}$ $\vee$ $\neg K_1$ $\vee$ $\neg K_5$, $\neg K_{11}$ $\vee$ $\neg K_5$ $\vee$ $\neg K_4$, $\neg K_1$ $\vee$ $\neg K_{14}$ $\vee$ $\neg K_5$, $\neg K_1$ $\vee$ $\neg K_9$, $\neg K_7$ $\vee$ $\neg K_4$, $\neg K_{14}$ $\vee$ $\neg K_5$ $\vee$ $\neg K_4$, 
\item \altcirc{56} $\A(\lam,\mu)=\{s_2s_3s_2s_1,s_1s_2s_3s_2s_1\}$ if $K_2$, $K_4$, $K_7$, $K_{12}$,$\neg K_6$, $\neg K_{14}$, $\neg K_8$, $\neg K_{11}$ $\vee$ $\neg K_{10}$ $\vee$ $\neg K_1$, $\neg K_{11}$ $\vee$ $\neg K_5$, $\neg K_{10}$ $\vee$ $\neg K_1$ $\vee$ $\neg K_{13}$, $\neg K_3$ $\vee$ $\neg K_{10}$ $\vee$ $\neg K_{11}$, $\neg K_3$ $\vee$ $\neg K_{10}$ $\vee$ $\neg K_{13}$, $\neg K_9$ $\vee$ $\neg K_{13}$, 
\item \altcirc{57} $\A(\lam,\mu)=\{s_3s_1s_2s_1,s_2s_3s_1s_2s_1\}$ if $K_3$, $K_8$, $K_6$, $K_{12}$,$\neg K_2$, $\neg K_{11}$, $\neg K_4$, $\neg K_{13}$, $\neg K_1$ $\vee$ $\neg K_9$ $\vee$ $\neg K_{14}$, $\neg K_1$ $\vee$ $\neg K_{14}$ $\vee$ $\neg K_5$, 
\item \altcirc{58} $\A(\lam,\mu)=\{s_1s_2s_3s_1s_2,s_1s_2s_3s_1s_2s_1\}$ if $K_2$, $K_6$, $K_{13}$, $K_{12}$,$\neg K_7$, $\neg K_3$, $\neg K_9$, $\neg K_{11}$ $\vee$ $\neg K_1$ $\vee$ $\neg K_5$, $\neg K_{11}$ $\vee$ $\neg K_5$ $\vee$ $\neg K_4$, $\neg K_{10}$ $\vee$ $\neg K_1$, $\neg K_8$ $\vee$ $\neg K_4$, $\neg K_1$ $\vee$ $\neg K_{14}$ $\vee$ $\neg K_5$, $\neg K_{14}$ $\vee$ $\neg K_5$ $\vee$ $\neg K_4$, 
\item \altcirc{59} $\A(\lam,\mu)=\{s_1s_2s_3s_2s_1,s_1s_2s_3s_1s_2s_1\}$ if $K_2$, $K_7$, $K_6$, $K_{12}$,$\neg K_{14}$, $\neg K_3$, $\neg K_4$, $\neg K_{13}$, $\neg K_{11}$ $\vee$ $\neg K_1$ $\vee$ $\neg K_5$, $\neg K_{11}$ $\vee$ $\neg K_{10}$ $\vee$ $\neg K_1$, 
\item \altcirc{60} $\A(\lam,\mu)=\{s_2s_3s_1s_2s_1,s_1s_2s_3s_1s_2s_1\}$ if $K_3$, $K_2$, $K_6$, $K_{12}$,$\neg K_7$, $\neg K_8$, $\neg K_{13}$, $\neg K_{11}$ $\vee$ $\neg K_1$ $\vee$ $\neg K_5$, $\neg K_{11}$ $\vee$ $\neg K_5$ $\vee$ $\neg K_4$, $\neg K_1$ $\vee$ $\neg K_{14}$ $\vee$ $\neg K_5$, $\neg K_9$ $\vee$ $\neg K_{14}$, $\neg K_{10}$ $\vee$ $\neg K_{11}$, $\neg K_{14}$ $\vee$ $\neg K_5$ $\vee$ $\neg K_4$, 
\item \altcirc{61} $\A(\lam,\mu)=\{1,s_1,s_2\}$ if $K_{11}$, $K_{10}$, $K_1$, $K_5$, $K_4$,$\neg K_{13}$, $\neg K_{14}$, $\neg K_3$, $\neg K_8$, $\neg K_2$ $\vee$ $\neg K_6$ $\vee$ $\neg K_{12}$, $\neg K_{12}$ $\vee$ $\neg K_7$, 
\item \altcirc{62} $\A(\lam,\mu)=\{1,s_1,s_2s_1\}$ if $K_{11}$, $K_1$, $K_8$, $K_5$, $K_4$,$\neg K_{10}$, $\neg K_3$, $\neg K_{12}$, $\neg K_{14}$, $\neg K_9$ $\vee$ $\neg K_{13}$, $\neg K_2$ $\vee$ $\neg K_6$ $\vee$ $\neg K_{13}$, 
\item \altcirc{63} $\A(\lam,\mu)=\{1,s_2,s_3\}$ if $K_{11}$, $K_{10}$, $K_1$, $K_{14}$, $K_5$,$\neg K_4$, $\neg K_{13}$, $\neg K_9$, $\neg K_3$, $\neg K_2$ $\vee$ $\neg K_6$ $\vee$ $\neg K_{12}$, $\neg K_2$ $\vee$ $\neg K_7$, 
\item \altcirc{64} $\A(\lam,\mu)=\{1,s_2,s_1s_2\}$ if $K_{11}$, $K_{10}$, $K_1$, $K_3$, $K_5$,$\neg K_4$, $\neg K_{13}$, $\neg K_8$, $\neg K_{14}$, $\neg K_2$ $\vee$ $\neg K_7$ $\vee$ $\neg K_{12}$, $\neg K_6$ $\vee$ $\neg K_{12}$, 
\item \altcirc{65} $\A(\lam,\mu)=\{1,s_2,s_3s_2\}$ if $K_{11}$, $K_{10}$, $K_1$, $K_{13}$, $K_5$,$\neg K_4$, $\neg K_{14}$, $\neg K_9$, $\neg K_3$, $\neg K_2$ $\vee$ $\neg K_7$ $\vee$ $\neg K_{12}$, $\neg K_2$ $\vee$ $\neg K_6$, 
\item \altcirc{66} $\A(\lam,\mu)=\{1,s_3,s_2s_3\}$ if $K_{11}$, $K_1$, $K_9$, $K_5$, $K_{14}$,$\neg K_{10}$, $\neg K_4$, $\neg K_2$, $\neg K_{13}$, $\neg K_3$ $\vee$ $\neg K_6$ $\vee$ $\neg K_{12}$, $\neg K_3$ $\vee$ $\neg K_8$, 
\item \altcirc{67} $\A(\lam,\mu)=\{s_1,s_2s_1,s_3s_1\}$ if $K_{11}$, $K_8$, $K_{14}$, $K_5$, $K_4$,$\neg K_1$, $\neg K_3$, $\neg K_{12}$, $\neg K_7$, $\neg K_2$ $\vee$ $\neg K_9$, $\neg K_2$ $\vee$ $\neg K_6$ $\vee$ $\neg K_{13}$, 
\item \altcirc{68} $\A(\lam,\mu)=\{s_1,s_2s_1,s_1s_2s_1\}$ if $K_{11}$, $K_3$, $K_8$, $K_5$, $K_4$,$\neg K_1$, $\neg K_{12}$, $\neg K_{10}$, $\neg K_{14}$, $\neg K_6$ $\vee$ $\neg K_{13}$, $\neg K_2$ $\vee$ $\neg K_9$ $\vee$ $\neg K_{13}$, 
\item \altcirc{69} $\A(\lam,\mu)=\{s_1,s_2s_1,s_3s_2s_1\}$ if $K_{11}$, $K_8$, $K_{12}$, $K_5$, $K_4$,$\neg K_1$, $\neg K_3$, $\neg K_7$, $\neg K_{14}$, $\neg K_2$ $\vee$ $\neg K_6$, $\neg K_2$ $\vee$ $\neg K_9$ $\vee$ $\neg K_{13}$, 
\item \altcirc{70} $\A(\lam,\mu)=\{s_1,s_3s_1,s_2s_3s_1\}$ if $K_{11}$, $K_7$, $K_{14}$, $K_5$, $K_4$,$\neg K_1$, $\neg K_2$, $\neg K_{12}$, $\neg K_8$, $\neg K_3$ $\vee$ $\neg K_{10}$, $\neg K_3$ $\vee$ $\neg K_6$ $\vee$ $\neg K_{13}$, 
\item \altcirc{71} $\A(\lam,\mu)=\{s_2,s_1s_2,s_1s_2s_1\}$ if $K_{11}$, $K_{10}$, $K_1$, $K_8$, $K_3$,$\neg K_5$, $\neg K_{13}$, $\neg K_{12}$, $\neg K_4$, $\neg K_2$ $\vee$ $\neg K_7$ $\vee$ $\neg K_{14}$, $\neg K_9$ $\vee$ $\neg K_{14}$, 
\item \altcirc{72} $\A(\lam,\mu)=\{s_2,s_3s_2,s_2s_3s_2\}$ if $K_{11}$, $K_{10}$, $K_1$, $K_9$, $K_{13}$,$\neg K_5$, $\neg K_{14}$, $\neg K_3$, $\neg K_2$, $\neg K_{12}$ $\vee$ $\neg K_7$ $\vee$ $\neg K_4$, $\neg K_8$ $\vee$ $\neg K_4$, 
\item \altcirc{73} $\A(\lam,\mu)=\{s_3,s_2s_3,s_3s_1\}$ if $K_1$, $K_4$, $K_9$, $K_5$, $K_{14}$,$\neg K_{11}$, $\neg K_2$, $\neg K_{13}$, $\neg K_7$, $\neg K_3$ $\vee$ $\neg K_6$ $\vee$ $\neg K_{12}$, $\neg K_8$ $\vee$ $\neg K_{12}$, 
\item \altcirc{74} $\A(\lam,\mu)=\{s_3,s_2s_3,s_1s_2s_3\}$ if $K_2$, $K_1$, $K_9$, $K_5$, $K_{14}$,$\neg K_{11}$, $\neg K_7$, $\neg K_{13}$, $\neg K_4$, $\neg K_6$ $\vee$ $\neg K_{12}$, $\neg K_3$ $\vee$ $\neg K_8$ $\vee$ $\neg K_{12}$, 
\item \altcirc{75} $\A(\lam,\mu)=\{s_3,s_2s_3,s_2s_3s_2\}$ if $K_1$, $K_{13}$, $K_9$, $K_5$, $K_{14}$,$\neg K_{11}$, $\neg K_{10}$, $\neg K_2$, $\neg K_4$, $\neg K_3$ $\vee$ $\neg K_8$ $\vee$ $\neg K_{12}$, $\neg K_3$ $\vee$ $\neg K_6$, 
\item \altcirc{76} $\A(\lam,\mu)=\{s_3,s_3s_1,s_2s_3s_1\}$ if $K_1$, $K_7$, $K_{14}$, $K_5$, $K_4$,$\neg K_{11}$, $\neg K_2$, $\neg K_9$, $\neg K_{12}$, $\neg K_{10}$ $\vee$ $\neg K_{13}$, $\neg K_3$ $\vee$ $\neg K_6$ $\vee$ $\neg K_{13}$, 
\item \altcirc{77} $\A(\lam,\mu)=\{s_1s_2,s_2s_1,s_1s_2s_1\}$ if $K_3$, $K_{11}$, $K_8$, $K_4$, $K_{10}$,$\neg K_1$, $\neg K_5$, $\neg K_{12}$, $\neg K_{13}$, $\neg K_2$ $\vee$ $\neg K_9$ $\vee$ $\neg K_{14}$, $\neg K_7$ $\vee$ $\neg K_{14}$, 
\item \altcirc{78} $\A(\lam,\mu)=\{s_1s_2,s_1s_2s_1,s_3s_1s_2\}$ if $K_3$, $K_{11}$, $K_8$, $K_{13}$, $K_{10}$,$\neg K_1$, $\neg K_{12}$, $\neg K_6$, $\neg K_4$, $\neg K_2$ $\vee$ $\neg K_7$ $\vee$ $\neg K_{14}$, $\neg K_2$ $\vee$ $\neg K_9$, 
\item \altcirc{79} $\A(\lam,\mu)=\{s_1s_2,s_1s_2s_1,s_3s_1s_2s_1\}$ if $K_3$, $K_{11}$, $K_8$, $K_{12}$, $K_{10}$,$\neg K_1$, $\neg K_6$, $\neg K_4$, $\neg K_{13}$, $\neg K_2$ $\vee$ $\neg K_7$, $\neg K_2$ $\vee$ $\neg K_9$ $\vee$ $\neg K_{14}$, 
\item \altcirc{80} $\A(\lam,\mu)=\{s_1s_2,s_3s_1s_2,s_2s_3s_1s_2\}$ if $K_3$, $K_{10}$, $K_{13}$, $K_{11}$, $K_6$,$\neg K_1$, $\neg K_{12}$, $\neg K_8$, $\neg K_2$, $\neg K_5$ $\vee$ $\neg K_4$, $\neg K_7$ $\vee$ $\neg K_{14}$ $\vee$ $\neg K_4$, 
\item \altcirc{81} $\A(\lam,\mu)=\{s_2s_1,s_3s_2s_1,s_2s_3s_2s_1\}$ if $K_{11}$, $K_7$, $K_8$, $K_4$, $K_{12}$,$\neg K_2$, $\neg K_5$, $\neg K_3$, $\neg K_{14}$, $\neg K_{10}$ $\vee$ $\neg K_1$, $\neg K_1$ $\vee$ $\neg K_9$ $\vee$ $\neg K_{13}$, 
\item \altcirc{82} $\A(\lam,\mu)=\{s_2s_3,s_3s_2,s_2s_3s_2\}$ if $K_{10}$, $K_1$, $K_9$, $K_{13}$, $K_{14}$,$\neg K_{11}$, $\neg K_5$, $\neg K_2$, $\neg K_3$, $\neg K_8$ $\vee$ $\neg K_{12}$ $\vee$ $\neg K_4$, $\neg K_7$ $\vee$ $\neg K_4$, 
\item \altcirc{83} $\A(\lam,\mu)=\{s_2s_3,s_1s_2s_3,s_1s_2s_3s_1\}$ if $K_9$, $K_2$, $K_1$, $K_7$, $K_{14}$,$\neg K_{12}$, $\neg K_5$, $\neg K_{13}$, $\neg K_4$, $\neg K_{11}$ $\vee$ $\neg K_{10}$, $\neg K_3$ $\vee$ $\neg K_{11}$ $\vee$ $\neg K_8$, 
\item \altcirc{84} $\A(\lam,\mu)=\{s_3s_1,s_2s_3s_1,s_1s_2s_3s_1\}$ if $K_2$, $K_7$, $K_{14}$, $K_5$, $K_4$,$\neg K_{12}$, $\neg K_{11}$, $\neg K_1$, $\neg K_9$, $\neg K_6$ $\vee$ $\neg K_{13}$, $\neg K_3$ $\vee$ $\neg K_{10}$ $\vee$ $\neg K_{13}$, 
\item \altcirc{85} $\A(\lam,\mu)=\{s_3s_1,s_2s_3s_1,s_2s_3s_2s_1\}$ if $K_4$, $K_7$, $K_{14}$, $K_5$, $K_{12}$,$\neg K_2$, $\neg K_{11}$, $\neg K_8$, $\neg K_1$, $\neg K_3$ $\vee$ $\neg K_6$, $\neg K_3$ $\vee$ $\neg K_{10}$ $\vee$ $\neg K_{13}$, 
\item \altcirc{86} $\A(\lam,\mu)=\{s_3s_2,s_2s_3s_2,s_3s_1s_2\}$ if $K_3$, $K_{10}$, $K_1$, $K_9$, $K_{13}$,$\neg K_{11}$, $\neg K_{14}$, $\neg K_6$, $\neg K_2$, $\neg K_{12}$ $\vee$ $\neg K_7$ $\vee$ $\neg K_4$, $\neg K_8$ $\vee$ $\neg K_{12}$, 
\item \altcirc{87} $\A(\lam,\mu)=\{s_3s_2,s_2s_3s_2,s_1s_2s_3s_2\}$ if $K_{10}$, $K_1$, $K_9$, $K_{13}$, $K_2$,$\neg K_{11}$, $\neg K_{14}$, $\neg K_6$, $\neg K_3$, $\neg K_7$ $\vee$ $\neg K_{12}$, $\neg K_8$ $\vee$ $\neg K_{12}$ $\vee$ $\neg K_4$, 
\item \altcirc{88} $\A(\lam,\mu)=\{s_3s_2,s_3s_1s_2,s_2s_3s_1s_2\}$ if $K_3$, $K_{10}$, $K_1$, $K_6$, $K_{13}$,$\neg K_{11}$, $\neg K_{12}$, $\neg K_9$, $\neg K_2$, $\neg K_{14}$ $\vee$ $\neg K_5$, $\neg K_7$ $\vee$ $\neg K_{14}$ $\vee$ $\neg K_4$, 
\item \altcirc{89} $\A(\lam,\mu)=\{s_1s_2s_1,s_3s_1s_2s_1,s_2s_3s_1s_2s_1\}$ if $K_3$, $K_{11}$, $K_8$, $K_6$, $K_{12}$,$\neg K_2$, $\neg K_4$, $\neg K_{10}$, $\neg K_{13}$, $\neg K_1$ $\vee$ $\neg K_5$, $\neg K_1$ $\vee$ $\neg K_9$ $\vee$ $\neg K_{14}$, 
\item \altcirc{90} $\A(\lam,\mu)=\{s_1s_2s_3,s_2s_3s_1,s_1s_2s_3s_1\}$ if $K_2$, $K_9$, $K_7$, $K_{14}$, $K_4$,$\neg K_{12}$, $\neg K_1$, $\neg K_5$, $\neg K_{13}$, $\neg K_3$ $\vee$ $\neg K_{10}$ $\vee$ $\neg K_{11}$, $\neg K_{11}$ $\vee$ $\neg K_8$, 
\item \altcirc{91} $\A(\lam,\mu)=\{s_1s_2s_3,s_1s_2s_3s_1,s_1s_2s_3s_2\}$ if $K_2$, $K_9$, $K_7$, $K_{14}$, $K_{13}$,$\neg K_{12}$, $\neg K_1$, $\neg K_6$, $\neg K_4$, $\neg K_3$ $\vee$ $\neg K_{11}$ $\vee$ $\neg K_8$, $\neg K_3$ $\vee$ $\neg K_{10}$, 
\item \altcirc{92} $\A(\lam,\mu)=\{s_1s_2s_3,s_1s_2s_3s_1,s_1s_2s_3s_2s_1\}$ if $K_2$, $K_9$, $K_7$, $K_{14}$, $K_{12}$,$\neg K_6$, $\neg K_1$, $\neg K_4$, $\neg K_{13}$, $\neg K_3$ $\vee$ $\neg K_{10}$ $\vee$ $\neg K_{11}$, $\neg K_3$ $\vee$ $\neg K_8$, 
\item \altcirc{93} $\A(\lam,\mu)=\{s_1s_2s_3,s_1s_2s_3s_2,s_1s_2s_3s_1s_2\}$ if $K_2$, $K_9$, $K_{14}$, $K_{13}$, $K_6$,$\neg K_{12}$, $\neg K_7$, $\neg K_1$, $\neg K_3$, $\neg K_5$ $\vee$ $\neg K_4$, $\neg K_{11}$ $\vee$ $\neg K_8$ $\vee$ $\neg K_4$, 
\item \altcirc{94} $\A(\lam,\mu)=\{s_2s_3s_1,s_3s_2s_1,s_2s_3s_2s_1\}$ if $K_7$, $K_8$, $K_4$, $K_{14}$, $K_{12}$,$\neg K_2$, $\neg K_3$, $\neg K_5$, $\neg K_{11}$, $\neg K_{10}$ $\vee$ $\neg K_1$ $\vee$ $\neg K_{13}$, $\neg K_1$ $\vee$ $\neg K_9$, 
\item \altcirc{95} $\A(\lam,\mu)=\{s_2s_3s_2,s_1s_2s_3s_2,s_1s_2s_3s_1s_2\}$ if $K_2$, $K_1$, $K_9$, $K_{13}$, $K_6$,$\neg K_{12}$, $\neg K_{10}$, $\neg K_{14}$, $\neg K_3$, $\neg K_{11}$ $\vee$ $\neg K_5$, $\neg K_{11}$ $\vee$ $\neg K_8$ $\vee$ $\neg K_4$, 
\item \altcirc{96} $\A(\lam,\mu)=\{s_3s_1s_2,s_2s_3s_1s_2,s_1s_2s_3s_1s_2\}$ if $K_3$, $K_2$, $K_6$, $K_{13}$, $K_{10}$,$\neg K_{12}$, $\neg K_1$, $\neg K_{11}$, $\neg K_9$, $\neg K_7$ $\vee$ $\neg K_{14}$, $\neg K_{14}$ $\vee$ $\neg K_5$ $\vee$ $\neg K_4$, 
\item \altcirc{97} $\A(\lam,\mu)=\{s_3s_1s_2,s_2s_3s_1s_2,s_2s_3s_1s_2s_1\}$ if $K_3$, $K_{10}$, $K_6$, $K_{13}$, $K_{12}$,$\neg K_2$, $\neg K_1$, $\neg K_{11}$, $\neg K_8$, $\neg K_7$ $\vee$ $\neg K_4$, $\neg K_{14}$ $\vee$ $\neg K_5$ $\vee$ $\neg K_4$, 
\item \altcirc{98} $\A(\lam,\mu)=\{s_3s_2s_1,s_2s_3s_2s_1,s_3s_1s_2s_1\}$ if $K_3$, $K_7$, $K_8$, $K_4$, $K_{12}$,$\neg K_2$, $\neg K_6$, $\neg K_{11}$, $\neg K_{14}$, $\neg K_1$ $\vee$ $\neg K_9$ $\vee$ $\neg K_{13}$, $\neg K_{10}$ $\vee$ $\neg K_{13}$, 
\item \altcirc{99} $\A(\lam,\mu)=\{s_3s_2s_1,s_2s_3s_2s_1,s_1s_2s_3s_2s_1\}$ if $K_4$, $K_2$, $K_8$, $K_7$, $K_{12}$,$\neg K_6$, $\neg K_{14}$, $\neg K_3$, $\neg K_{11}$, $\neg K_{10}$ $\vee$ $\neg K_1$ $\vee$ $\neg K_{13}$, $\neg K_9$ $\vee$ $\neg K_{13}$, 
\item \altcirc{100} $\A(\lam,\mu)=\{s_3s_2s_1,s_3s_1s_2s_1,s_2s_3s_1s_2s_1\}$ if $K_3$, $K_8$, $K_4$, $K_6$, $K_{12}$,$\neg K_2$, $\neg K_{11}$, $\neg K_7$, $\neg K_{13}$, $\neg K_1$ $\vee$ $\neg K_9$ $\vee$ $\neg K_{14}$, $\neg K_{14}$ $\vee$ $\neg K_5$, 
\item \altcirc{101} $\A(\lam,\mu)=\{s_1s_2s_3s_1,s_1s_2s_3s_2s_1,s_1s_2s_3s_1s_2s_1\}$ if $K_{14}$, $K_2$, $K_7$, $K_6$, $K_{12}$,$\neg K_3$, $\neg K_9$, $\neg K_4$, $\neg K_{13}$, $\neg K_{11}$ $\vee$ $\neg K_{10}$ $\vee$ $\neg K_1$, $\neg K_1$ $\vee$ $\neg K_5$, 
\item \altcirc{102} $\A(\lam,\mu)=\{s_1s_2s_3s_2,s_2s_3s_1s_2,s_1s_2s_3s_1s_2\}$ if $K_3$, $K_2$, $K_9$, $K_6$, $K_{13}$,$\neg K_{12}$, $\neg K_{14}$, $\neg K_1$, $\neg K_{10}$, $\neg K_{11}$ $\vee$ $\neg K_5$ $\vee$ $\neg K_4$, $\neg K_{11}$ $\vee$ $\neg K_8$, 
\item \altcirc{103} $\A(\lam,\mu)=\{s_1s_2s_3s_2,s_1s_2s_3s_1s_2,s_1s_2s_3s_1s_2s_1\}$ if $K_2$, $K_9$, $K_6$, $K_{13}$, $K_{12}$,$\neg K_7$, $\neg K_3$, $\neg K_{14}$, $\neg K_1$, $\neg K_{11}$ $\vee$ $\neg K_5$ $\vee$ $\neg K_4$, $\neg K_8$ $\vee$ $\neg K_4$, 
\item \altcirc{104} $\A(\lam,\mu)=\{s_2s_3s_1s_2,s_3s_1s_2s_1,s_2s_3s_1s_2s_1\}$ if $K_3$, $K_8$, $K_6$, $K_{13}$, $K_{12}$,$\neg K_2$, $\neg K_{11}$, $\neg K_4$, $\neg K_{10}$, $\neg K_1$ $\vee$ $\neg K_{14}$ $\vee$ $\neg K_5$, $\neg K_1$ $\vee$ $\neg K_9$, 
\item \altcirc{105} $\A(\lam,\mu)=\{s_2s_3s_2s_1,s_1s_2s_3s_2s_1,s_1s_2s_3s_1s_2s_1\}$ if $K_2$, $K_4$, $K_7$, $K_6$, $K_{12}$,$\neg K_{14}$, $\neg K_3$, $\neg K_8$, $\neg K_{13}$, $\neg K_{11}$ $\vee$ $\neg K_{10}$ $\vee$ $\neg K_1$, $\neg K_{11}$ $\vee$ $\neg K_5$, 
\item \altcirc{106} $\A(\lam,\mu)=\{s_3s_1s_2s_1,s_2s_3s_1s_2s_1,s_1s_2s_3s_1s_2s_1\}$ if $K_3$, $K_2$, $K_8$, $K_6$, $K_{12}$,$\neg K_7$, $\neg K_{11}$, $\neg K_4$, $\neg K_{13}$, $\neg K_1$ $\vee$ $\neg K_{14}$ $\vee$ $\neg K_5$, $\neg K_9$ $\vee$ $\neg K_{14}$, 
\item \altcirc{107} $\A(\lam,\mu)=\{s_1s_2s_3s_1s_2,s_1s_2s_3s_2s_1,s_1s_2s_3s_1s_2s_1\}$ if $K_2$, $K_7$, $K_6$, $K_{13}$, $K_{12}$,$\neg K_{14}$, $\neg K_3$, $\neg K_4$, $\neg K_9$, $\neg K_{11}$ $\vee$ $\neg K_1$ $\vee$ $\neg K_5$, $\neg K_{10}$ $\vee$ $\neg K_1$, 
\item \altcirc{108} $\A(\lam,\mu)=\{s_1s_2s_3s_2s_1,s_2s_3s_1s_2s_1,s_1s_2s_3s_1s_2s_1\}$ if $K_3$, $K_2$, $K_7$, $K_6$, $K_{12}$,$\neg K_{14}$, $\neg K_4$, $\neg K_8$, $\neg K_{13}$, $\neg K_{11}$ $\vee$ $\neg K_1$ $\vee$ $\neg K_5$, $\neg K_{10}$ $\vee$ $\neg K_{11}$, 
\item \altcirc{109} $\A(\lam,\mu)=\{1,s_1,s_2,s_2s_1\}$ if $K_1$, $K_8$, $K_4$, $K_{11}$, $K_{10}$, $K_5$,$\neg K_{13}$, $\neg K_3$, $\neg K_{12}$, $\neg K_{14}$, 
\item \altcirc{110} $\A(\lam,\mu)=\{1,s_1,s_3,s_3s_1\}$ if $K_{11}$, $K_1$, $K_{14}$, $K_5$, $K_4$,$\neg K_{10}$, $\neg K_9$, $\neg K_7$, $\neg K_8$, $\neg K_2$ $\vee$ $\neg K_6$ $\vee$ $\neg K_{12}$, $\neg K_3$ $\vee$ $\neg K_6$ $\vee$ $\neg K_{12}$, $\neg K_2$ $\vee$ $\neg K_6$ $\vee$ $\neg K_{13}$, $\neg K_3$ $\vee$ $\neg K_6$ $\vee$ $\neg K_{13}$, 
\item \altcirc{111} $\A(\lam,\mu)=\{1,s_2,s_3,s_2s_3\}$ if $K_1$, $K_{11}$, $K_{14}$, $K_{10}$, $K_9$, $K_5$,$\neg K_4$, $\neg K_{13}$, $\neg K_2$, $\neg K_3$, 
\item \altcirc{112} $\A(\lam,\mu)=\{1,s_2,s_1s_2,s_1s_2s_1\}$ if $K_1$, $K_8$, $K_{11}$, $K_3$, $K_{10}$, $K_5$,$\neg K_4$, $\neg K_{13}$, $\neg K_{14}$, $\neg K_{12}$, 
\item \altcirc{113} $\A(\lam,\mu)=\{1,s_2,s_3s_2,s_2s_3s_2\}$ if $K_1$, $K_{11}$, $K_{10}$, $K_{13}$, $K_9$, $K_5$,$\neg K_4$, $\neg K_{14}$, $\neg K_3$, $\neg K_2$, 
\item \altcirc{114} $\A(\lam,\mu)=\{s_1,s_1s_2,s_2s_1,s_1s_2s_1\}$ if $K_8$, $K_4$, $K_{11}$, $K_3$, $K_{10}$, $K_5$,$\neg K_1$, $\neg K_{12}$, $\neg K_{14}$, $\neg K_{13}$, 
\item \altcirc{115} $\A(\lam,\mu)=\{s_1,s_2s_1,s_3s_1,s_2s_3s_1\}$ if $K_8$, $K_4$, $K_{11}$, $K_7$, $K_{14}$, $K_5$,$\neg K_1$, $\neg K_2$, $\neg K_3$, $\neg K_{12}$, 
\item \altcirc{116} $\A(\lam,\mu)=\{s_1,s_2s_1,s_3s_2s_1,s_2s_3s_2s_1\}$ if $K_8$, $K_4$, $K_{11}$, $K_7$, $K_5$, $K_{12}$,$\neg K_1$, $\neg K_2$, $\neg K_3$, $\neg K_{14}$, 
\item \altcirc{117} $\A(\lam,\mu)=\{s_2,s_1s_2,s_3s_2,s_3s_1s_2\}$ if $K_{11}$, $K_{10}$, $K_1$, $K_3$, $K_{13}$,$\neg K_5$, $\neg K_8$, $\neg K_9$, $\neg K_6$, $\neg K_2$ $\vee$ $\neg K_7$ $\vee$ $\neg K_{12}$, $\neg K_2$ $\vee$ $\neg K_7$ $\vee$ $\neg K_{14}$, $\neg K_{12}$ $\vee$ $\neg K_7$ $\vee$ $\neg K_4$, $\neg K_7$ $\vee$ $\neg K_{14}$ $\vee$ $\neg K_4$, 
\item \altcirc{118} $\A(\lam,\mu)=\{s_3,s_2s_3,s_3s_1,s_2s_3s_1\}$ if $K_1$, $K_4$, $K_7$, $K_{14}$, $K_9$, $K_5$,$\neg K_{11}$, $\neg K_2$, $\neg K_{13}$, $\neg K_{12}$, 
\item \altcirc{119} $\A(\lam,\mu)=\{s_3,s_2s_3,s_3s_2,s_2s_3s_2\}$ if $K_1$, $K_{14}$, $K_{10}$, $K_{13}$, $K_9$, $K_5$,$\neg K_{11}$, $\neg K_2$, $\neg K_4$, $\neg K_3$, 
\item \altcirc{120} $\A(\lam,\mu)=\{s_3,s_2s_3,s_1s_2s_3,s_1s_2s_3s_1\}$ if $K_1$, $K_7$, $K_{14}$, $K_2$, $K_9$, $K_5$,$\neg K_{11}$, $\neg K_{12}$, $\neg K_{13}$, $\neg K_4$, 
\item \altcirc{121} $\A(\lam,\mu)=\{s_1s_2,s_1s_2s_1,s_3s_1s_2,s_2s_3s_1s_2\}$ if $K_8$, $K_{11}$, $K_3$, $K_{10}$, $K_6$, $K_{13}$,$\neg K_1$, $\neg K_{12}$, $\neg K_2$, $\neg K_4$, 
\item \altcirc{122} $\A(\lam,\mu)=\{s_1s_2,s_1s_2s_1,s_3s_1s_2s_1,s_2s_3s_1s_2s_1\}$ if $K_8$, $K_{11}$, $K_3$, $K_{10}$, $K_6$, $K_{12}$,$\neg K_1$, $\neg K_2$, $\neg K_4$, $\neg K_{13}$, 
\item \altcirc{123} $\A(\lam,\mu)=\{s_2s_1,s_1s_2s_1,s_3s_2s_1,s_3s_1s_2s_1\}$ if $K_3$, $K_{11}$, $K_8$, $K_4$, $K_{12}$,$\neg K_5$, $\neg K_6$, $\neg K_{10}$, $\neg K_7$, $\neg K_1$ $\vee$ $\neg K_9$ $\vee$ $\neg K_{14}$, $\neg K_2$ $\vee$ $\neg K_9$ $\vee$ $\neg K_{14}$, $\neg K_1$ $\vee$ $\neg K_9$ $\vee$ $\neg K_{13}$, $\neg K_2$ $\vee$ $\neg K_9$ $\vee$ $\neg K_{13}$, 
\item \altcirc{124} $\A(\lam,\mu)=\{s_2s_3,s_1s_2s_3,s_2s_3s_2,s_1s_2s_3s_2\}$ if $K_2$, $K_1$, $K_9$, $K_{13}$, $K_{14}$,$\neg K_7$, $\neg K_{10}$, $\neg K_5$, $\neg K_6$, $\neg K_3$ $\vee$ $\neg K_{11}$ $\vee$ $\neg K_8$, $\neg K_8$ $\vee$ $\neg K_{12}$ $\vee$ $\neg K_4$, $\neg K_3$ $\vee$ $\neg K_8$ $\vee$ $\neg K_{12}$, $\neg K_{11}$ $\vee$ $\neg K_8$ $\vee$ $\neg K_4$, 
\item \altcirc{125} $\A(\lam,\mu)=\{s_3s_1,s_1s_2s_3,s_2s_3s_1,s_1s_2s_3s_1\}$ if $K_4$, $K_7$, $K_{14}$, $K_2$, $K_9$, $K_5$,$\neg K_{12}$, $\neg K_{11}$, $\neg K_1$, $\neg K_{13}$, 
\item \altcirc{126} $\A(\lam,\mu)=\{s_3s_1,s_2s_3s_1,s_3s_2s_1,s_2s_3s_2s_1\}$ if $K_8$, $K_4$, $K_7$, $K_{14}$, $K_5$, $K_{12}$,$\neg K_2$, $\neg K_{11}$, $\neg K_1$, $\neg K_3$, 
\item \altcirc{127} $\A(\lam,\mu)=\{s_3s_2,s_2s_3s_2,s_3s_1s_2,s_2s_3s_1s_2\}$ if $K_1$, $K_3$, $K_{10}$, $K_6$, $K_{13}$, $K_9$,$\neg K_{11}$, $\neg K_{12}$, $\neg K_{14}$, $\neg K_2$, 
\item \altcirc{128} $\A(\lam,\mu)=\{s_3s_2,s_2s_3s_2,s_1s_2s_3s_2,s_1s_2s_3s_1s_2\}$ if $K_1$, $K_{10}$, $K_6$, $K_{13}$, $K_2$, $K_9$,$\neg K_{11}$, $\neg K_{12}$, $\neg K_{14}$, $\neg K_3$, 
\item \altcirc{129} $\A(\lam,\mu)=\{s_1s_2s_3,s_1s_2s_3s_1,s_1s_2s_3s_2,s_1s_2s_3s_1s_2\}$ if $K_7$, $K_{14}$, $K_6$, $K_{13}$, $K_2$, $K_9$,$\neg K_{12}$, $\neg K_1$, $\neg K_4$, $\neg K_3$, 
\item \altcirc{130} $\A(\lam,\mu)=\{s_1s_2s_3,s_1s_2s_3s_1,s_1s_2s_3s_2s_1,s_1s_2s_3s_1s_2s_1\}$ if $K_7$, $K_{14}$, $K_6$, $K_2$, $K_9$, $K_{12}$,$\neg K_3$, $\neg K_1$, $\neg K_4$, $\neg K_{13}$, 
\item \altcirc{131} $\A(\lam,\mu)=\{s_2s_3s_1,s_1s_2s_3s_1,s_2s_3s_2s_1,s_1s_2s_3s_2s_1\}$ if $K_2$, $K_4$, $K_7$, $K_{14}$, $K_{12}$,$\neg K_6$, $\neg K_8$, $\neg K_9$, $\neg K_5$, $\neg K_{11}$ $\vee$ $\neg K_{10}$ $\vee$ $\neg K_1$, $\neg K_{10}$ $\vee$ $\neg K_1$ $\vee$ $\neg K_{13}$, $\neg K_3$ $\vee$ $\neg K_{10}$ $\vee$ $\neg K_{11}$, $\neg K_3$ $\vee$ $\neg K_{10}$ $\vee$ $\neg K_{13}$, 
\item \altcirc{132} $\A(\lam,\mu)=\{s_3s_1s_2,s_1s_2s_3s_2,s_2s_3s_1s_2,s_1s_2s_3s_1s_2\}$ if $K_3$, $K_{10}$, $K_6$, $K_{13}$, $K_2$, $K_9$,$\neg K_{12}$, $\neg K_1$, $\neg K_{14}$, $\neg K_{11}$, 
\item \altcirc{133} $\A(\lam,\mu)=\{s_3s_1s_2,s_2s_3s_1s_2,s_3s_1s_2s_1,s_2s_3s_1s_2s_1\}$ if $K_8$, $K_3$, $K_{10}$, $K_6$, $K_{13}$, $K_{12}$,$\neg K_2$, $\neg K_1$, $\neg K_{11}$, $\neg K_4$, 
\item \altcirc{134} $\A(\lam,\mu)=\{s_3s_2s_1,s_2s_3s_2s_1,s_3s_1s_2s_1,s_2s_3s_1s_2s_1\}$ if $K_8$, $K_4$, $K_7$, $K_3$, $K_6$, $K_{12}$,$\neg K_2$, $\neg K_{11}$, $\neg K_{14}$, $\neg K_{13}$, 
\item \altcirc{135} $\A(\lam,\mu)=\{s_3s_2s_1,s_2s_3s_2s_1,s_1s_2s_3s_2s_1,s_1s_2s_3s_1s_2s_1\}$ if $K_8$, $K_4$, $K_7$, $K_6$, $K_2$, $K_{12}$,$\neg K_{14}$, $\neg K_3$, $\neg K_{13}$, $\neg K_{11}$, 
\item \altcirc{136} $\A(\lam,\mu)=\{s_1s_2s_3s_2,s_1s_2s_3s_1s_2,s_1s_2s_3s_2s_1,s_1s_2s_3s_1s_2s_1\}$ if $K_7$, $K_6$, $K_{13}$, $K_2$, $K_9$, $K_{12}$,$\neg K_{14}$, $\neg K_3$, $\neg K_1$, $\neg K_4$, 
\item \altcirc{137} $\A(\lam,\mu)=\{s_2s_3s_1s_2,s_1s_2s_3s_1s_2,s_2s_3s_1s_2s_1,s_1s_2s_3s_1s_2s_1\}$ if $K_3$, $K_2$, $K_6$, $K_{13}$, $K_{12}$,$\neg K_7$, $\neg K_8$, $\neg K_{10}$, $\neg K_9$, $\neg K_{11}$ $\vee$ $\neg K_1$ $\vee$ $\neg K_5$, $\neg K_{11}$ $\vee$ $\neg K_5$ $\vee$ $\neg K_4$, $\neg K_1$ $\vee$ $\neg K_{14}$ $\vee$ $\neg K_5$, $\neg K_{14}$ $\vee$ $\neg K_5$ $\vee$ $\neg K_4$, 
\item \altcirc{138} $\A(\lam,\mu)=\{s_3s_1s_2s_1,s_1s_2s_3s_2s_1,s_2s_3s_1s_2s_1,s_1s_2s_3s_1s_2s_1\}$ if $K_8$, $K_7$, $K_3$, $K_6$, $K_2$, $K_{12}$,$\neg K_{14}$, $\neg K_{11}$, $\neg K_4$, $\neg K_{13}$, 
\item \altcirc{139} $\A(\lam,\mu)=\{1,s_1,s_2,s_3,s_3s_1\}$ if $K_1$, $K_4$, $K_{11}$, $K_{14}$, $K_{10}$, $K_5$,$\neg K_{13}$, $\neg K_9$, $\neg K_3$, $\neg K_7$, $\neg K_8$, $\neg K_2$ $\vee$ $\neg K_6$ $\vee$ $\neg K_{12}$, 
\item \altcirc{140} $\A(\lam,\mu)=\{1,s_1,s_3,s_2s_1,s_3s_1\}$ if $K_1$, $K_8$, $K_4$, $K_{11}$, $K_{14}$, $K_5$,$\neg K_{10}$, $\neg K_3$, $\neg K_{12}$, $\neg K_9$, $\neg K_7$, $\neg K_2$ $\vee$ $\neg K_6$ $\vee$ $\neg K_{13}$, 
\item \altcirc{141} $\A(\lam,\mu)=\{1,s_1,s_3,s_2s_3,s_3s_1\}$ if $K_1$, $K_4$, $K_{11}$, $K_{14}$, $K_9$, $K_5$,$\neg K_{10}$, $\neg K_2$, $\neg K_{13}$, $\neg K_7$, $\neg K_8$, $\neg K_3$ $\vee$ $\neg K_6$ $\vee$ $\neg K_{12}$, 
\item \altcirc{142} $\A(\lam,\mu)=\{1,s_1,s_3,s_3s_1,s_2s_3s_1\}$ if $K_1$, $K_4$, $K_{11}$, $K_7$, $K_{14}$, $K_5$,$\neg K_{10}$, $\neg K_2$, $\neg K_9$, $\neg K_{12}$, $\neg K_8$, $\neg K_3$ $\vee$ $\neg K_6$ $\vee$ $\neg K_{13}$, 
\item \altcirc{143} $\A(\lam,\mu)=\{1,s_2,s_1s_2,s_3s_2,s_3s_1s_2\}$ if $K_1$, $K_{11}$, $K_3$, $K_{10}$, $K_{13}$, $K_5$,$\neg K_4$, $\neg K_8$, $\neg K_{14}$, $\neg K_9$, $\neg K_6$, $\neg K_2$ $\vee$ $\neg K_7$ $\vee$ $\neg K_{12}$, 
\item \altcirc{144} $\A(\lam,\mu)=\{s_1,s_2s_1,s_1s_2s_1,s_3s_2s_1,s_3s_1s_2s_1\}$ if $K_8$, $K_4$, $K_{11}$, $K_3$, $K_5$, $K_{12}$,$\neg K_1$, $\neg K_6$, $\neg K_{10}$, $\neg K_7$, $\neg K_{14}$, $\neg K_2$ $\vee$ $\neg K_9$ $\vee$ $\neg K_{13}$, 
\item \altcirc{145} $\A(\lam,\mu)=\{s_2,s_1s_2,s_3s_2,s_1s_2s_1,s_3s_1s_2\}$ if $K_1$, $K_8$, $K_{11}$, $K_3$, $K_{10}$, $K_{13}$,$\neg K_5$, $\neg K_9$, $\neg K_{12}$, $\neg K_6$, $\neg K_4$, $\neg K_2$ $\vee$ $\neg K_7$ $\vee$ $\neg K_{14}$, 
\item \altcirc{146} $\A(\lam,\mu)=\{s_2,s_1s_2,s_3s_2,s_2s_3s_2,s_3s_1s_2\}$ if $K_1$, $K_{11}$, $K_3$, $K_{10}$, $K_{13}$, $K_9$,$\neg K_5$, $\neg K_8$, $\neg K_{14}$, $\neg K_6$, $\neg K_2$, $\neg K_{12}$ $\vee$ $\neg K_7$ $\vee$ $\neg K_4$, 
\item \altcirc{147} $\A(\lam,\mu)=\{s_2,s_1s_2,s_3s_2,s_3s_1s_2,s_2s_3s_1s_2\}$ if $K_1$, $K_{11}$, $K_3$, $K_{10}$, $K_6$, $K_{13}$,$\neg K_5$, $\neg K_{12}$, $\neg K_8$, $\neg K_9$, $\neg K_2$, $\neg K_7$ $\vee$ $\neg K_{14}$ $\vee$ $\neg K_4$, 
\item \altcirc{148} $\A(\lam,\mu)=\{s_3,s_2s_3,s_1s_2s_3,s_2s_3s_2,s_1s_2s_3s_2\}$ if $K_1$, $K_{14}$, $K_{13}$, $K_2$, $K_9$, $K_5$,$\neg K_{11}$, $\neg K_7$, $\neg K_{10}$, $\neg K_4$, $\neg K_6$, $\neg K_3$ $\vee$ $\neg K_8$ $\vee$ $\neg K_{12}$, 
\item \altcirc{149} $\A(\lam,\mu)=\{s_1s_2,s_2s_1,s_1s_2s_1,s_3s_2s_1,s_3s_1s_2s_1\}$ if $K_8$, $K_4$, $K_{11}$, $K_3$, $K_{10}$, $K_{12}$,$\neg K_1$, $\neg K_5$, $\neg K_6$, $\neg K_7$, $\neg K_{13}$, $\neg K_2$ $\vee$ $\neg K_9$ $\vee$ $\neg K_{14}$, 
\item \altcirc{150} $\A(\lam,\mu)=\{s_2s_1,s_1s_2s_1,s_3s_2s_1,s_2s_3s_2s_1,s_3s_1s_2s_1\}$ if $K_8$, $K_4$, $K_{11}$, $K_7$, $K_3$, $K_{12}$,$\neg K_2$, $\neg K_5$, $\neg K_6$, $\neg K_{10}$, $\neg K_{14}$, $\neg K_1$ $\vee$ $\neg K_9$ $\vee$ $\neg K_{13}$, 
\item \altcirc{151} $\A(\lam,\mu)=\{s_2s_1,s_1s_2s_1,s_3s_2s_1,s_3s_1s_2s_1,s_2s_3s_1s_2s_1\}$ if $K_8$, $K_4$, $K_{11}$, $K_3$, $K_6$, $K_{12}$,$\neg K_2$, $\neg K_5$, $\neg K_{10}$, $\neg K_7$, $\neg K_{13}$, $\neg K_1$ $\vee$ $\neg K_9$ $\vee$ $\neg K_{14}$, 
\item \altcirc{152} $\A(\lam,\mu)=\{s_2s_3,s_3s_2,s_1s_2s_3,s_2s_3s_2,s_1s_2s_3s_2\}$ if $K_1$, $K_{14}$, $K_{10}$, $K_{13}$, $K_2$, $K_9$,$\neg K_{11}$, $\neg K_7$, $\neg K_5$, $\neg K_6$, $\neg K_3$, $\neg K_8$ $\vee$ $\neg K_{12}$ $\vee$ $\neg K_4$, 
\item \altcirc{153} $\A(\lam,\mu)=\{s_2s_3,s_1s_2s_3,s_2s_3s_2,s_1s_2s_3s_1,s_1s_2s_3s_2\}$ if $K_1$, $K_7$, $K_{14}$, $K_{13}$, $K_2$, $K_9$,$\neg K_{12}$, $\neg K_{10}$, $\neg K_5$, $\neg K_6$, $\neg K_4$, $\neg K_3$ $\vee$ $\neg K_{11}$ $\vee$ $\neg K_8$, 
\item \altcirc{154} $\A(\lam,\mu)=\{s_2s_3,s_1s_2s_3,s_2s_3s_2,s_1s_2s_3s_2,s_1s_2s_3s_1s_2\}$ if $K_1$, $K_{14}$, $K_6$, $K_{13}$, $K_2$, $K_9$,$\neg K_{12}$, $\neg K_7$, $\neg K_{10}$, $\neg K_5$, $\neg K_3$, $\neg K_{11}$ $\vee$ $\neg K_8$ $\vee$ $\neg K_4$, 
\item \altcirc{155} $\A(\lam,\mu)=\{s_3s_1,s_2s_3s_1,s_1s_2s_3s_1,s_2s_3s_2s_1,s_1s_2s_3s_2s_1\}$ if $K_4$, $K_7$, $K_{14}$, $K_2$, $K_5$, $K_{12}$,$\neg K_6$, $\neg K_{11}$, $\neg K_8$, $\neg K_1$, $\neg K_9$, $\neg K_3$ $\vee$ $\neg K_{10}$ $\vee$ $\neg K_{13}$, 
\item \altcirc{156} $\A(\lam,\mu)=\{s_1s_2s_3,s_2s_3s_1,s_1s_2s_3s_1,s_2s_3s_2s_1,s_1s_2s_3s_2s_1\}$ if $K_4$, $K_7$, $K_{14}$, $K_2$, $K_9$, $K_{12}$,$\neg K_6$, $\neg K_8$, $\neg K_1$, $\neg K_5$, $\neg K_{13}$, $\neg K_3$ $\vee$ $\neg K_{10}$ $\vee$ $\neg K_{11}$, 
\item \altcirc{157} $\A(\lam,\mu)=\{s_2s_3s_1,s_3s_2s_1,s_1s_2s_3s_1,s_2s_3s_2s_1,s_1s_2s_3s_2s_1\}$ if $K_8$, $K_4$, $K_7$, $K_{14}$, $K_2$, $K_{12}$,$\neg K_6$, $\neg K_9$, $\neg K_3$, $\neg K_5$, $\neg K_{11}$, $\neg K_{10}$ $\vee$ $\neg K_1$ $\vee$ $\neg K_{13}$, 
\item \altcirc{158} $\A(\lam,\mu)=\{s_2s_3s_1,s_1s_2s_3s_1,s_2s_3s_2s_1,s_1s_2s_3s_2s_1,s_1s_2s_3s_1s_2s_1\}$ if $K_4$, $K_7$, $K_{14}$, $K_6$, $K_2$, $K_{12}$,$\neg K_3$, $\neg K_8$, $\neg K_9$, $\neg K_5$, $\neg K_{13}$, $\neg K_{11}$ $\vee$ $\neg K_{10}$ $\vee$ $\neg K_1$, 
\item \altcirc{159} $\A(\lam,\mu)=\{s_3s_1s_2,s_2s_3s_1s_2,s_1s_2s_3s_1s_2,s_2s_3s_1s_2s_1,s_1s_2s_3s_1s_2s_1\}$ if $K_3$, $K_{10}$, $K_6$, $K_{13}$, $K_2$, $K_{12}$,$\neg K_7$, $\neg K_1$, $\neg K_{11}$, $\neg K_8$, $\neg K_9$, $\neg K_{14}$ $\vee$ $\neg K_5$ $\vee$ $\neg K_4$, 
\item \altcirc{160} $\A(\lam,\mu)=\{s_1s_2s_3s_2,s_2s_3s_1s_2,s_1s_2s_3s_1s_2,s_2s_3s_1s_2s_1,s_1s_2s_3s_1s_2s_1\}$ if $K_3$, $K_6$, $K_{13}$, $K_2$, $K_9$, $K_{12}$,$\neg K_7$, $\neg K_{14}$, $\neg K_1$, $\neg K_8$, $\neg K_{10}$, $\neg K_{11}$ $\vee$ $\neg K_5$ $\vee$ $\neg K_4$, 
\item \altcirc{161} $\A(\lam,\mu)=\{s_2s_3s_1s_2,s_3s_1s_2s_1,s_1s_2s_3s_1s_2,s_2s_3s_1s_2s_1,s_1s_2s_3s_1s_2s_1\}$ if $K_8$, $K_3$, $K_6$, $K_{13}$, $K_2$, $K_{12}$,$\neg K_7$, $\neg K_{11}$, $\neg K_4$, $\neg K_{10}$, $\neg K_9$, $\neg K_1$ $\vee$ $\neg K_{14}$ $\vee$ $\neg K_5$, 
\item \altcirc{162} $\A(\lam,\mu)=\{s_2s_3s_1s_2,s_1s_2s_3s_1s_2,s_1s_2s_3s_2s_1,s_2s_3s_1s_2s_1,s_1s_2s_3s_1s_2s_1\}$ if $K_7$, $K_3$, $K_6$, $K_{13}$, $K_2$, $K_{12}$,$\neg K_{14}$, $\neg K_4$, $\neg K_8$, $\neg K_{10}$, $\neg K_9$, $\neg K_{11}$ $\vee$ $\neg K_1$ $\vee$ $\neg K_5$, 
\item \altcirc{163} $\A(\lam,\mu)=\{1,s_1,s_2,s_3,s_2s_1,s_3s_1\}$ if $K_1$, $K_8$, $K_4$, $K_{11}$, $K_{14}$, $K_{10}$, $K_5$,$\neg K_{13}$, $\neg K_3$, $\neg K_{12}$, $\neg K_9$, $\neg K_7$, 
\item \altcirc{164} $\A(\lam,\mu)=\{1,s_1,s_2,s_3,s_2s_3,s_3s_1\}$ if $K_1$, $K_4$, $K_{11}$, $K_{14}$, $K_{10}$, $K_9$, $K_5$,$\neg K_{13}$, $\neg K_2$, $\neg K_3$, $\neg K_7$, $\neg K_8$, 
\item \altcirc{165} $\A(\lam,\mu)=\{1,s_1,s_2,s_1s_2,s_2s_1,s_1s_2s_1\}$ if $K_1$, $K_8$, $K_4$, $K_{11}$, $K_3$, $K_{10}$, $K_5$,$\neg K_{13}$, $\neg K_{12}$, $\neg K_{14}$, 
\item \altcirc{166} $\A(\lam,\mu)=\{1,s_1,s_3,s_2s_1,s_3s_1,s_2s_3s_1\}$ if $K_1$, $K_8$, $K_4$, $K_{11}$, $K_7$, $K_{14}$, $K_5$,$\neg K_{10}$, $\neg K_2$, $\neg K_3$, $\neg K_{12}$, $\neg K_9$, 
\item \altcirc{167} $\A(\lam,\mu)=\{1,s_1,s_3,s_2s_3,s_3s_1,s_2s_3s_1\}$ if $K_1$, $K_4$, $K_{11}$, $K_7$, $K_{14}$, $K_9$, $K_5$,$\neg K_{10}$, $\neg K_2$, $\neg K_{13}$, $\neg K_{12}$, $\neg K_8$, 
\item \altcirc{168} $\A(\lam,\mu)=\{1,s_2,s_3,s_2s_3,s_3s_2,s_2s_3s_2\}$ if $K_1$, $K_{11}$, $K_{14}$, $K_{10}$, $K_{13}$, $K_9$, $K_5$,$\neg K_4$, $\neg K_2$, $\neg K_3$, 
\item \altcirc{169} $\A(\lam,\mu)=\{1,s_2,s_1s_2,s_3s_2,s_1s_2s_1,s_3s_1s_2\}$ if $K_1$, $K_8$, $K_{11}$, $K_3$, $K_{10}$, $K_{13}$, $K_5$,$\neg K_4$, $\neg K_{14}$, $\neg K_9$, $\neg K_{12}$, $\neg K_6$, 
\item \altcirc{170} $\A(\lam,\mu)=\{1,s_2,s_1s_2,s_3s_2,s_2s_3s_2,s_3s_1s_2\}$ if $K_1$, $K_{11}$, $K_3$, $K_{10}$, $K_{13}$, $K_9$, $K_5$,$\neg K_4$, $\neg K_8$, $\neg K_{14}$, $\neg K_6$, $\neg K_2$, 
\item \altcirc{171} $\A(\lam,\mu)=\{s_1,s_1s_2,s_2s_1,s_1s_2s_1,s_3s_2s_1,s_3s_1s_2s_1\}$ if $K_8$, $K_4$, $K_{11}$, $K_3$, $K_{10}$, $K_5$, $K_{12}$,$\neg K_1$, $\neg K_6$, $\neg K_7$, $\neg K_{14}$, $\neg K_{13}$, 
\item \altcirc{172} $\A(\lam,\mu)=\{s_1,s_2s_1,s_3s_1,s_2s_3s_1,s_3s_2s_1,s_2s_3s_2s_1\}$ if $K_8$, $K_4$, $K_{11}$, $K_7$, $K_{14}$, $K_5$, $K_{12}$,$\neg K_1$, $\neg K_2$, $\neg K_3$, 
\item \altcirc{173} $\A(\lam,\mu)=\{s_1,s_2s_1,s_1s_2s_1,s_3s_2s_1,s_2s_3s_2s_1,s_3s_1s_2s_1\}$ if $K_8$, $K_4$, $K_{11}$, $K_7$, $K_3$, $K_5$, $K_{12}$,$\neg K_1$, $\neg K_2$, $\neg K_6$, $\neg K_{10}$, $\neg K_{14}$, 
\item \altcirc{174} $\A(\lam,\mu)=\{s_2,s_1s_2,s_3s_2,s_1s_2s_1,s_3s_1s_2,s_2s_3s_1s_2\}$ if $K_1$, $K_8$, $K_{11}$, $K_3$, $K_{10}$, $K_6$, $K_{13}$,$\neg K_5$, $\neg K_{12}$, $\neg K_9$, $\neg K_2$, $\neg K_4$, 
\item \altcirc{175} $\A(\lam,\mu)=\{s_2,s_1s_2,s_3s_2,s_2s_3s_2,s_3s_1s_2,s_2s_3s_1s_2\}$ if $K_1$, $K_{11}$, $K_3$, $K_{10}$, $K_6$, $K_{13}$, $K_9$,$\neg K_5$, $\neg K_{12}$, $\neg K_8$, $\neg K_{14}$, $\neg K_2$, 
\item \altcirc{176} $\A(\lam,\mu)=\{s_3,s_2s_3,s_3s_1,s_1s_2s_3,s_2s_3s_1,s_1s_2s_3s_1\}$ if $K_1$, $K_4$, $K_7$, $K_{14}$, $K_2$, $K_9$, $K_5$,$\neg K_{11}$, $\neg K_{12}$, $\neg K_{13}$, 
\item \altcirc{177} $\A(\lam,\mu)=\{s_3,s_2s_3,s_3s_2,s_1s_2s_3,s_2s_3s_2,s_1s_2s_3s_2\}$ if $K_1$, $K_{14}$, $K_{10}$, $K_{13}$, $K_2$, $K_9$, $K_5$,$\neg K_{11}$, $\neg K_7$, $\neg K_4$, $\neg K_6$, $\neg K_3$, 
\item \altcirc{178} $\A(\lam,\mu)=\{s_3,s_2s_3,s_1s_2s_3,s_2s_3s_2,s_1s_2s_3s_1,s_1s_2s_3s_2\}$ if $K_1$, $K_7$, $K_{14}$, $K_{13}$, $K_2$, $K_9$, $K_5$,$\neg K_{11}$, $\neg K_{12}$, $\neg K_{10}$, $\neg K_4$, $\neg K_6$, 
\item \altcirc{179} $\A(\lam,\mu)=\{s_1s_2,s_2s_1,s_1s_2s_1,s_3s_2s_1,s_3s_1s_2s_1,s_2s_3s_1s_2s_1\}$ if $K_8$, $K_4$, $K_{11}$, $K_3$, $K_{10}$, $K_6$, $K_{12}$,$\neg K_1$, $\neg K_2$, $\neg K_5$, $\neg K_7$, $\neg K_{13}$, 
\item \altcirc{180} $\A(\lam,\mu)=\{s_1s_2,s_1s_2s_1,s_3s_1s_2,s_2s_3s_1s_2,s_3s_1s_2s_1,s_2s_3s_1s_2s_1\}$ if $K_8$, $K_{11}$, $K_3$, $K_{10}$, $K_6$, $K_{13}$, $K_{12}$,$\neg K_1$, $\neg K_2$, $\neg K_4$, 
\item \altcirc{181} $\A(\lam,\mu)=\{s_2s_1,s_1s_2s_1,s_3s_2s_1,s_2s_3s_2s_1,s_3s_1s_2s_1,s_2s_3s_1s_2s_1\}$ if $K_8$, $K_4$, $K_{11}$, $K_7$, $K_3$, $K_6$, $K_{12}$,$\neg K_2$, $\neg K_5$, $\neg K_{10}$, $\neg K_{14}$, $\neg K_{13}$, 
\item \altcirc{182} $\A(\lam,\mu)=\{s_2s_3,s_3s_2,s_1s_2s_3,s_2s_3s_2,s_1s_2s_3s_2,s_1s_2s_3s_1s_2\}$ if $K_1$, $K_{14}$, $K_{10}$, $K_6$, $K_{13}$, $K_2$, $K_9$,$\neg K_{11}$, $\neg K_{12}$, $\neg K_7$, $\neg K_5$, $\neg K_3$, 
\item \altcirc{183} $\A(\lam,\mu)=\{s_2s_3,s_1s_2s_3,s_2s_3s_2,s_1s_2s_3s_1,s_1s_2s_3s_2,s_1s_2s_3s_1s_2\}$ if $K_1$, $K_7$, $K_{14}$, $K_6$, $K_{13}$, $K_2$, $K_9$,$\neg K_{12}$, $\neg K_{10}$, $\neg K_5$, $\neg K_4$, $\neg K_3$, 
\item \altcirc{184} $\A(\lam,\mu)=\{s_3s_1,s_1s_2s_3,s_2s_3s_1,s_1s_2s_3s_1,s_2s_3s_2s_1,s_1s_2s_3s_2s_1\}$ if $K_4$, $K_7$, $K_{14}$, $K_2$, $K_9$, $K_5$, $K_{12}$,$\neg K_6$, $\neg K_{11}$, $\neg K_8$, $\neg K_1$, $\neg K_{13}$, 
\item \altcirc{185} $\A(\lam,\mu)=\{s_3s_1,s_2s_3s_1,s_3s_2s_1,s_1s_2s_3s_1,s_2s_3s_2s_1,s_1s_2s_3s_2s_1\}$ if $K_8$, $K_4$, $K_7$, $K_{14}$, $K_2$, $K_5$, $K_{12}$,$\neg K_6$, $\neg K_{11}$, $\neg K_1$, $\neg K_9$, $\neg K_3$, 
\item \altcirc{186} $\A(\lam,\mu)=\{s_3s_2,s_2s_3s_2,s_3s_1s_2,s_1s_2s_3s_2,s_2s_3s_1s_2,s_1s_2s_3s_1s_2\}$ if $K_1$, $K_3$, $K_{10}$, $K_6$, $K_{13}$, $K_2$, $K_9$,$\neg K_{11}$, $\neg K_{12}$, $\neg K_{14}$, 
\item \altcirc{187} $\A(\lam,\mu)=\{s_1s_2s_3,s_2s_3s_1,s_1s_2s_3s_1,s_2s_3s_2s_1,s_1s_2s_3s_2s_1,s_1s_2s_3s_1s_2s_1\}$ if $K_4$, $K_7$, $K_{14}$, $K_6$, $K_2$, $K_9$, $K_{12}$,$\neg K_3$, $\neg K_8$, $\neg K_1$, $\neg K_5$, $\neg K_{13}$, 
\item \altcirc{188} $\A(\lam,\mu)=\{s_1s_2s_3,s_1s_2s_3s_1,s_1s_2s_3s_2,s_1s_2s_3s_1s_2,s_1s_2s_3s_2s_1,s_1s_2s_3s_1s_2s_1\}$ if $K_7$, $K_{14}$, $K_6$, $K_{13}$, $K_2$, $K_9$, $K_{12}$,$\neg K_3$, $\neg K_1$, $\neg K_4$, 
\item \altcirc{189} $\A(\lam,\mu)=\{s_2s_3s_1,s_3s_2s_1,s_1s_2s_3s_1,s_2s_3s_2s_1,s_1s_2s_3s_2s_1,s_1s_2s_3s_1s_2s_1\}$ if $K_8$, $K_4$, $K_7$, $K_{14}$, $K_6$, $K_2$, $K_{12}$,$\neg K_3$, $\neg K_9$, $\neg K_5$, $\neg K_{13}$, $\neg K_{11}$, 
\item \altcirc{190} $\A(\lam,\mu)=\{s_3s_1s_2,s_1s_2s_3s_2,s_2s_3s_1s_2,s_1s_2s_3s_1s_2,s_2s_3s_1s_2s_1,s_1s_2s_3s_1s_2s_1\}$ if $K_3$, $K_{10}$, $K_6$, $K_{13}$, $K_2$, $K_9$, $K_{12}$,$\neg K_7$, $\neg K_1$, $\neg K_{14}$, $\neg K_{11}$, $\neg K_8$, 
\item \altcirc{191} $\A(\lam,\mu)=\{s_3s_1s_2,s_2s_3s_1s_2,s_3s_1s_2s_1,s_1s_2s_3s_1s_2,s_2s_3s_1s_2s_1,s_1s_2s_3s_1s_2s_1\}$ if $K_8$, $K_3$, $K_{10}$, $K_6$, $K_{13}$, $K_2$, $K_{12}$,$\neg K_7$, $\neg K_1$, $\neg K_{11}$, $\neg K_4$, $\neg K_9$, 
\item \altcirc{192} $\A(\lam,\mu)=\{s_3s_2s_1,s_2s_3s_2s_1,s_3s_1s_2s_1,s_1s_2s_3s_2s_1,s_2s_3s_1s_2s_1,s_1s_2s_3s_1s_2s_1\}$ if $K_8$, $K_4$, $K_7$, $K_3$, $K_6$, $K_2$, $K_{12}$,$\neg K_{14}$, $\neg K_{11}$, $\neg K_{13}$, 
\item \altcirc{193} $\A(\lam,\mu)=\{s_1s_2s_3s_2,s_2s_3s_1s_2,s_1s_2s_3s_1s_2,s_1s_2s_3s_2s_1,s_2s_3s_1s_2s_1,s_1s_2s_3s_1s_2s_1\}$ if $K_7$, $K_3$, $K_6$, $K_{13}$, $K_2$, $K_9$, $K_{12}$,$\neg K_{14}$, $\neg K_1$, $\neg K_4$, $\neg K_8$, $\neg K_{10}$, 
\item \altcirc{194} $\A(\lam,\mu)=\{s_2s_3s_1s_2,s_3s_1s_2s_1,s_1s_2s_3s_1s_2,s_1s_2s_3s_2s_1,s_2s_3s_1s_2s_1,s_1s_2s_3s_1s_2s_1\}$ if $K_8$, $K_7$, $K_3$, $K_6$, $K_{13}$, $K_2$, $K_{12}$,$\neg K_{14}$, $\neg K_{11}$, $\neg K_4$, $\neg K_{10}$, $\neg K_9$, 
\item $\A(\lam,\mu)=\varnothing$ otherwise.
\end{enumerate}

\newpage
\section{Theorem \ref{thm:qKWMF}: The associated \texorpdfstring{$Z_i$}{Zi} polynomials}
\label{appendix:ztable}
This appendix provides the formulas needed in Theorem \ref{thm:qKWMF}.

\subsection{\textbf{Formula for} \texorpdfstring{$Z_1$}{Z1}}
\zcase{Z_1}{P_1}{Q_1}{R_1}
\subsection{\textbf{Formula for} \texorpdfstring{$Z_2$}{Z2}}
\zcase{Z_2}{P_3}{Q_4}{R_1}
\subsection{\textbf{Formula for} \texorpdfstring{$Z_3$}{Z3}}
\zcase{Z_3}{P_1}{Q_1}{R_4}
\subsection{\textbf{Formula for} \texorpdfstring{$Z_4$}{Z4}}
\zcase{Z_4}{P_3}{Q_6}{R_1}
\subsection{\textbf{Formula for} \texorpdfstring{$Z_5$}{Z5}}
\zcase{Z_5}{P_1}{Q_5}{R_4}
\subsection{\textbf{Formula for} \texorpdfstring{$Z_6$}{Z6}}
\zcase{Z_6}{P_4}{Q_1}{R_1}
\subsection{\textbf{Formula for} \texorpdfstring{$Z_7$}{Z7}}
\zcase{Z_7}{P_1}{Q_5}{R_3}
\subsection{\textbf{Formula for} \texorpdfstring{$Z_8$}{Z8}}
\zcase{Z_8}{P_4}{Q_4}{R_1}
\subsection{\textbf{Formula for} \texorpdfstring{$Z_{9}$}{Z9}}
\zcase{Z_9}{P_4}{Q_1}{R_4}
\subsection{\textbf{Formula for} \texorpdfstring{$Z_{10}$}{Z10}}
\zcase{Z_{10}}{P_1}{Q_6}{R_3}
\subsection{\textbf{Formula for} \texorpdfstring{$Z_{11}$}{Z11}}
\zcase{Z_{11}}{P_1}{Q_6}{R_1}

\end{document}